\documentclass[12pt]{article}
\usepackage{amssymb,amsmath,graphicx,amsthm,mathrsfs}
\usepackage{color}
\usepackage[colorlinks=true]{hyperref}
\hypersetup{linkcolor=blue, urlcolor=green, citecolor=red}
\newtheorem{theorem}{Theorem}[section]
\newtheorem{lemma}[theorem]{Lemma}

\newtheorem{definition}[theorem]{Definition}

\newtheorem{Remark}[theorem]{Remark}
\numberwithin{equation}{section}

\title{Time optimal sampled-data  controls for  heat equations}

\author{Gengsheng Wang\thanks{School of Mathematics and Statistics, Wuhan University, Wuhan, 430072, China (wanggs62@yeah.net). The  author was partially supported by the National Natural Science Foundation of China
under grant 11571264.}
\and Donghui Yang\thanks{School of Mathematical
Sciences, Central South  University, Changsha, 410075, China
(donghyang@139.com). The  author was partially supported by the National Natural Science Foundation of China
under grant 11371375.}
\and Yubiao Zhang\thanks{Center for Applied Mathematics, Tianjin University, Tianjin, 300072, China (yubiao\b{ }zhang@whu.edu.cn).} }

\begin{document}

\date{ }
\maketitle

\begin{abstract}
In this paper, we first design a time optimal control problem for
the heat equation with sampled-data controls, and then use it to approximate
a time optimal control problem for the heat equation with distributed controls.
Our design  is reasonable from perspective of sampled-data controls. And
it
might provide a right way  for the numerical approach of a time optimal distributed control problem, via the corresponding semi-discretized (in time variable) time optimal control problem.

The study of such a time optimal sampled-data control problem is not easy, because it
may have infinitely many optimal controls. We find connections among this problem, a minimal norm sampled-data control problem and a minimization problem. And obtain some properties on these problems.
Based on these,  we not only build up error estimates for optimal time and optimal controls between the time optimal sampled-data control problem and the time optimal distributed control  problem, in terms of  the sampling period,
but also prove that such estimates  are {\it optimal} in some sense.

\end{abstract}



\bigskip

\section{Introduction}

\subsection{Motivation and problems}

In most  published literature on time optimal control problems,  controls
are distributed in time, i.e., they can vary at each instant of time.
However,
 in practical application, it is more convenient to use  controls which vary only finite times.  Sampled-data controls are such  kind of  controls.
  In this paper, we will design and study  a time optimal control problem for
the heat equation with sampled-data controls. And then we use it to approximate
a   time optimal control problem for the heat equation with distributed controls,
through building up several  error estimates for optimal time and optimal controls between these two problems, in terms of  the sampling period.
Such errors estimates have laid  foundation for us  to replace distributed
controls by sampled-data controls in time optimal control problems for heat equations.

Throughout this paper, $\mathbb{R}^+\triangleq(0,\infty)$; $\Omega\subset \mathbb R^d$ ($d\in\mathbb N^+\triangleq\{1, 2, \dots\}$)  is a bounded domain with a $C^2$ boundary $\partial\Omega$;
$\omega\subset\Omega$ is an open and nonempty subset with its characteristic function $\chi_\omega$; $\lambda_1$ is the first eigenvalue of $-\Delta$ with the homogeneous Dirichlet boundary condition over $\Omega$; $B_r(0)$ denotes the
 closed ball in $L^2(\Omega)$, centered at $0$ and  of radius  $r>0$; for each measurable set $\mathcal{A}$ in $\mathbb{R}$, $|\mathcal{A}|$ denotes its Lebesgue measure;
$\langle \cdot,\cdot \rangle$ and $\|\cdot\|$ denote the usual inner product and    norm of $L^2(\Omega)$, respectively;  $\langle \cdot,\cdot \rangle_\omega$ and $\|\cdot\|_\omega$ stand for the usual inner product and  norm in $L^2(\omega)$, respectively.

First, we introduce a  time optimal distributed control problem for the heat equation. Take  $B_r(0)$ (with
$r>0$ arbitrarily fixed) as our target. For each  $M\geq 0$ and $y_0\in L^2(\Omega)\setminus B_r(0)$,
consider the following time optimal distributed control problem:
\begin{eqnarray}\label{time-1}
 (TP)^{M,y_0}:\;\;T(M,y_0)=\inf\left\{\hat t>0
 \; :\;
 \exists\,\hat u\in \mathcal U^M
 \;\;\mbox{s.t.}\;\;
 y(\hat t;y_0,\hat u)\in B_r(0)\right\},
\end{eqnarray}
where
\begin{eqnarray*}\label{U-ad}
\mathcal U^M\triangleq
\left\{ u\in L^2(\mathbb R^+;L^2(\Omega))
\; :\;
\|u\|_{L^2(\mathbb R^+;L^2(\Omega))}\leq M \right\},
\end{eqnarray*}
and where  $y(\cdot;y_0,u)$  is the solution to the following distributed controlled heat  equation:
\begin{eqnarray}\label{heat-1}
\left\{
\begin{array}{lll}
\partial_t y-\Delta y=\chi_\omega u &\mbox{in}&\Omega\times \mathbb R^+,\\
y=0 &\mbox{on}&\partial \Omega\times\mathbb R^+,\\
y(0)=y_0 &\mbox{in}&\Omega.
\end{array}
\right.
\end{eqnarray}
  Since $y(t;y_0,0)\rightarrow 0$ as $t\rightarrow\infty$, we find that  $T(M,y_0)<\infty$, for all  $M\geq 0$ and $y_0\in L^2(\Omega)\setminus B_r(0)$.
About $(TP)^{M,y_0}$,
we introduce some concepts in the following definition:
\begin{definition}\label{Yuanyuanwangdefinition1.1}
(i) The number $T(M,y_0)$ is called the optimal time; $\hat u\in \mathcal U^M$ is called an admissible control if $y(\hat t;y_0,\hat u)\in B_r(0)$ for some $\hat t>0$;  $u^*\in \mathcal U^M$ is called an optimal control  if $y(T(M,y_0);y_0,u^*)\in B_r(0)$. (ii) Two optimal controls are said to be different (or the same),
if they are different (or the same)  on their effective domain $\big(0,T(M,y_0)\big)$.
      \end{definition}
Several  notes on the problem $(TP)^{M,y_0}$ are given in order:
\begin{itemize}
    \item   It is shown  in Theorem \ref{Lemma-existences-TP} that for each $M\geq0$
    and   $y_0\in L^2(\Omega)\setminus B_r(0)$, $(TP)^{M,y_0}$ has a unique optimal control.
  \item In many time optimal distributed control problems for  heat equations, controls are taken from $L^\infty(\mathbb{R}^+;L^2(\Omega))$. However, the current setting is also significant (see, for instance,   \cite{GL} and \cite{WXZ}).
\end{itemize}

Next, we are  going to design a time optimal sampled-data control problem for the heat equation. For this purpose, we  define the
 following  space of sampled-data controls (where $\delta>0$ is arbitrarily fixed):
\begin{eqnarray}\label{heat-Linfty-delta}
 L_{\delta}^2(\mathbb R^+;L^2(\Omega))\triangleq\Big\{u_\delta \in L^2(\mathbb R^+;L^2(\Omega))
 \;:\;
 u_{\delta}\triangleq \sum_{i=1}^\infty\chi_{((i-1)\delta,i\delta]} u^i,
 \;\{u^i\}_{i=1}^\infty\subset L^2(\Omega) \Big\},
\end{eqnarray}
endowed with the $L^2(\mathbb R^+;L^2(\Omega))$-norm.
Here and in what follows, $\chi_{((i-1)\delta,i\delta]}$ denotes the characteristic function of the interval $\big((i-1)\delta,i\delta\big]$ for each $i\in\mathbb N^+$. The numbers $\delta$,\,$2\delta$,\,$\ldots$,\,$i\delta$,\,$\ldots$ are called the sampling instants, while $\delta$ is called the sampling period.
Each   $u_{\delta}$  in the  space $L_{\delta}^2(\mathbb R^+;L^2(\Omega))$ is called a sampled-data control.
 For each $u_{\delta}\in L_{\delta}^2(\mathbb R^+;L^2(\Omega))$ and each $y_0\in L^2(\Omega)$, write $y(\cdot;y_0,u_\delta)$ for the solution to the following sampled-data controlled heat equation:
\begin{equation}\label{heat-2}
\left\{
\begin{array}{lll}
\partial_t y-\Delta y=\chi_\omega u_{\delta} &\mbox{in}&\Omega\times \mathbb R^+,\\
y=0 &\mbox{on}&\partial \Omega\times \mathbb R^+,\\
y(0)=y_0 &\mbox{in}&\Omega.
\end{array}
\right.
\end{equation}
For each  $M\geq0$, $y_0\in L^2(\Omega)\setminus B_r(0)$ and each $\delta>0$,
 consider the following time  optimal sampled-data control problem:
\begin{eqnarray}\label{time-2}
 (TP)_{\delta}^{M,y_0}:\;\;\;T_\delta(M,y_0)=\inf\left\{ k\delta
 \;:\;
 \exists\; k\in\mathbb N^+,
 \exists\,u_\delta\in \mathcal U^M_\delta
 \;\mbox{s.t.}\;
  y(k\delta;y_0,u_\delta)\in B_r(0) \right\},
\end{eqnarray}
where
\begin{eqnarray}\label{U-ad-delta}
\mathcal U^M_\delta\triangleq
\left\{ u_\delta\in L_\delta^2(\mathbb R^+;L^2(\Omega))
\;\; :\;\;
\|u_\delta\|_{L^2(\mathbb R^+;L^2(\Omega))} \leq M \right\}.
\end{eqnarray}
 Since $y(t;y_0,0)\rightarrow 0$ as $t\rightarrow\infty$, we see  that
 $T_\delta(M,y_0)<\infty$, for all  $M\geq0$, $y_0\in L^2(\Omega)\setminus B_r(0)$ and  $\delta>0$.
With respect to $(TP)_{\delta}^{M,y_0}$, we introduce some concepts in the following definition:
 \begin{definition}\label{wgsdefinition1.1}
 (i) The number $T_\delta(M,y_0)$ is called  the optimal time; $u_{\delta}\in \mathcal U^M_\delta$ is called  an admissible control if $y(\hat k\delta;y_0,u_{\delta})\in B_r(0)$ for some $\hat k\in\mathbb N^+$;  $u_{\delta}^*\in \mathcal U^M_\delta$ is called an optimal control if $y(T_{\delta}(M,y_0);y_0,u_{\delta}^*)\in B_r(0)$.
 (ii)  A control $u^*_{\delta}$ is called the optimal control with the minimal norm, if $u^*_{\delta}$ is an optimal control and satisfies that
$\|u^*_{\delta}\|_{L^2(0, T_\delta(M,y_0);L^2(\Omega))}
\leq
\|v^*_{\delta}\|_{L^2(0, T_\delta(M,y_0);L^2(\Omega))}$
for any optimal control $v^*_{\delta}$.
(iii) Two optimal controls are said to be different (or the same), if they are different
(or the same) over $(0, T_{\delta}(M,y_0))$.
\end{definition}

 Several notes on this problem are given in order:
\begin{itemize}
      \item  The optimal time $T_\delta(M,y_0)$ is a multiple of $\delta$ (see (\ref{time-2})).  For each $M\geq0$, each $y_0\in L^2(\Omega)\setminus B_r(0)$ and each $\delta>0$, $(TP)_\delta^{M,y_0}$  has
  a unique optimal control with the minimal norm (see (ii) in  Theorem \ref{Lemma-existences-TP}); Given $y_0\in L^2(\Omega)\setminus B_r(0)$, there are infinitely many pairs  $(M,\delta)$ so that  $(TP)_\delta^{M,y_0}$  has infinitely  many different optimal controls (see
  Theorem \ref{Lemma-nonunique-TP}).
  \item We may  design a time optimal sampled-data control problem
   in another way: To find a control $u_\delta^*$ in $\mathcal{U}_\delta^M$ so that $y(\cdot; y_0,u_\delta^*)$ enters $B_r(0)$ in the shortest time $\hat T_\delta(M,y_0)$ (which may not be a multiple of $\delta$). We denote  this problem by $(\widehat{TP})_{\delta}^{M,y_0}$.  Several reasons for us to design  time optimal sampled-data control problem to be
      $(TP)_{\delta}^{M,y_0}$ are as follows: (i) Each sampled-data control
      $u_\delta$ has the form:  $\sum_{i=1}^\infty\chi_{((i-1)\delta,i\delta]}  u^i$ with some $\{u^i\}_{i=1}^\infty\subset L^2(\Omega)$.
      From the perspective of sampled-data controls, each $u^i$ should be active in the whole subinterval $((i-1)\delta, i\delta]$. Thus, our definition for   $T(M,y_0)$ is reasonable.  (ii)
       In the definition $(\widehat{TP})_{\delta}^{M,y_0}$,
         in order to make sure if the control process should be finished,
         we need to observe the solution (of the controlled equation) at each time.
         However, in our definition $(TP)_{\delta}^{M,y_0}$, we only need to observe the solution at time points $i\delta$, with $i=1,2,\dots$. (iii) Our  design on  $(TP)_{\delta}^{M,y_0}$ might provide a right way to approach numerically $(TP)^{M,y_0}$ via a discretized time optimal control problem. For instance,
         if we semi-discretize $(TP)^{M,y_0}$ in time variable, then our
          design on  $(TP)_{\delta}^{M,y_0}$ can be borrowed to define a semi-discretized (in the time variable) time optimal control problem.

\end{itemize}

%\begin{Remark}\label{WANGremark1.2}
%(i) The problem $(TP)^{M,y_0}$ has a unique optimal control, see (i) of Theorem \ref{Lemma-existences-TP};
%(ii) For some $M>0$ and $\delta>0$, $(TP)_\delta^{M,y_0}$ may  have infinitely  many optimal controls, see Theorem \ref{Lemma-nonunique-TP}.
%
%
%\end{Remark}

\subsection{Main results}
The main results of this paper are presented in the following three  theorems.

\begin{theorem}\label{Theorem-time-order}
Let $y_0\in L^2(\Omega)\setminus B_r(0)$. Then the following conclusions are true:

(i) For each $(M_1, M_2)$ with  $0<M_1\leq M_2$, there is  $\delta_0\triangleq \delta_0(M_1,M_2,y_0,r)>0$ so that
\begin{eqnarray}\label{TP-order-0}
0\leq T_\delta(M,y_0)-T(M,y_0)\leq   2\delta
  \;\;\mbox{for all}\;\; \delta\in(0,\delta_0)
  \;\;\mbox{and}\;\; M\in[M_1,M_2].
\end{eqnarray}
Moreover, $\delta_0$  depends on  $M_1$, $M_2$,  $y_0$ and $r$ continuously.

(ii) For each  $M>0$ and $\eta\in(0,1)$, there exists  a measurable set $\mathcal A_{M,\eta}\subset(0,1)$ (depending also on $y_0$ and $r$) with $\lim_{h\rightarrow 0^+} \frac{1}{h} |\mathcal A_{M,\eta}\cap (0,h)|=\eta$ so that
\begin{eqnarray}\label{TP-order-1}
 \delta> T_\delta(M,y_0)-T(M,y_0) > (1-\eta) \delta
  \;\;\mbox{for each}\;\; \delta\in  \mathcal A_{M,\eta}.
\end{eqnarray}
\end{theorem}

\begin{theorem}\label{Theorem-time-control-convergence}
Let $y_0\in L^2(\Omega)\setminus B_r(0)$. For each $M>0$ and $\delta>0$, let $u^*_M$
and $u^*_{M,\delta}$
be  the optimal control to  $(TP)^{M,y_0}$ and   the optimal control with the minimal norm   to $(TP)_\delta^{M,y_0}$, respectively. Then the following conclusions are true:

 (i) For each $(M_1,M_2)$ with $0<M_1\leq M_2$, there is $C\triangleq C(M_1,M_2,y_0,r)>0$ so that
\begin{eqnarray} \label{TP-control-converge-0}
  \sup_{M_1\leq M\leq M_2}\|u_{M,\delta}^*-u_M^*\|_{L^2(0,T(M,y_0);L^2(\Omega))}
  &\leq& C \delta
  \;\;\mbox{for each}\;\; \delta>0.
  \end{eqnarray}
Moreover, $C$ depends on   $M_1$, $M_2$,  $y_0$  and $r$ continuously.

(ii) For each  $M>0$ and $\eta\in(0,1)$, there is   a measurable set $\mathcal A_{M,\eta}\subset(0,1)$ (depending also on $y_0$ and $r$) with $\lim_{h\rightarrow 0^+} \frac{1}{h} |\mathcal A_{M,\eta}\cap (0,h)|=\eta$ so that
\begin{eqnarray}\label{TP-control-converge-0-1}
  \|u_{M,\delta}^*-u_M^*\|_{L^2(0,T(M,y_0);L^2(\Omega))}
  &\geq& \frac{1}{2}\lambda_1^{3/2}r (1-\eta)\delta
  \;\;\mbox{for each}\;\; \delta \in \mathcal A_{M,\eta}.
\end{eqnarray}
\end{theorem}

\begin{theorem}\label{Theorem-time-control-convergence-1}
Let $y_0\in L^2(\Omega)\setminus B_r(0)$. For each $M>0$,   let $u^*_M$ be  the optimal control to  $(TP)^{M,y_0}$. Then the following conclusions are true:

(i) For each $(M_1,M_2)$ with $0<M_1\leq M_2$, there is $C\triangleq C(M_1,M_2,y_0,r)>0$ so that
\begin{eqnarray}\label{TP-control-converge-1}
  \sup_{M_1\leq M\leq M_2}\|u_{M,\delta}-u_M^*\|_{L^2(0,T(M,y_0);L^2(\Omega))}
  &\leq& C \sqrt{\delta}
  \;\;\mbox{for each}\;\; \delta>0,
\end{eqnarray}
where
$u_{M,\delta}$ is any optimal control to $(TP)_\delta^{M,y_0}$. Moreover,
$C$ depends on   $M_1$, $M_2$,  $y_0$  and $r$ continuously.

(ii) For each  $M>0$ and $\eta\in(0,1)$, there is   a measurable set $\mathcal A_{M,\eta}\subset(0,1)$ (depending also on $y_0$ and $r$) with $\lim_{h\rightarrow 0^+} \frac{1}{h} |\mathcal A_{M,\eta}\cap (0,h)|=\eta$ so that for each $\delta\in \mathcal A_{M,\eta}$, there is an optimal control $\hat u_{M,\delta}$ to $(TP)_\delta^{M,y_0}$ so that
\begin{eqnarray} \label{TP-control-converge-1-1}
  \|\hat u_{M,\delta}-u_M^*\|_{L^2(0,T(M,y_0);L^2(\Omega))}
  \geq C_{M} \sqrt{(1-\eta)\delta},
\end{eqnarray}
for some positive constant $C_{M}\triangleq C_M(y_0,r)$.
\end{theorem}

Several remarks on the main results are given in order.
\begin{itemize}
  \item Theorem \ref{Theorem-time-order} and Theorem \ref{Theorem-time-control-convergence} present two facts. First,  the error between $T_\delta(M,y_0)$ and $T(M,y_0)$ and the error between $u_M^*$ and $u^*_{M,\delta}$ have the order $1$ with respect to the sampling period $\delta$. Second, this order is {\it optimal}, because of the lower bound estimates (\ref{TP-order-1}) and (\ref{TP-control-converge-0-1}), and because of the property that $\lim_{h\rightarrow 0^+} \frac{1}{h} |\mathcal A_{M,\eta}\cap (0,h)|=\eta$ with any $\eta\in(0,1)$. Notice that when $\delta\in (0,1)\setminus\mathcal A_{M,\eta}$, (\ref{TP-order-1}) may not be true (see Theorem~\ref{theorem6.2-wang},
      as well as Remark~\ref{remark6.3-yuan}).

\item  Theorem~\ref{Theorem-time-control-convergence-1}, as well as Theorem \ref{Theorem-time-control-convergence},
 presents two facts. First, in $(TP)_\delta^{M,y_0}$,
 the  optimal control with the minimal norm differs from some of other optimal controls, from perspective of the order of the errors. Second, the order of the error between any optimal control of $(TP)_\delta^{M,y_0}$ and the optimal control to $(TP)^{M,y_0}$ is
 $1/2$, with respect to $\delta$. Moreover, this order is {\it optimal} in the sense (ii) of
Theorem~\ref{Theorem-time-control-convergence-1}.

  \item  Since we aim  to approximate $u^*_M$ by $u^*_{M,\delta}$ and because the
  efficient domain of $u^*_M$  is $(0,T(M,y_0))$, we  take the $L^2(0,T(M,y_0);L^2(\Omega))$-norm in the estimates in  Theorem \ref{Theorem-time-control-convergence} and Theorem~\ref{Theorem-time-control-convergence-1}.

  \item  There have been many  publications on  optimal sampled-data control problems (with fixed ending time point).
             In \cite{BourdinTrelat-1} (see also \cite{BourdinTrelat-2}), the authors built up the Pontryagin maximum principle for some optimal sampled-data control problems. In  \cite{BourdinTrelat}, the authors showed that for some LQ problem,
        the optimal sampled-data control converges to the optimal distributed control as  the sampling period tends to zero. In \cite{TWZ}, the authors built up some  error estimates
        between  the optimal distributed control and the optimal sampled-data control for some periodic  heat equations. About more works on sampled-data controls, we would
like to mention \cite{Ackermann,ACM,BB,CF, FR,GS, IK,Imura, LandauZito,ST} and the
references therein.

  \item There have been some literatures on the approximations of time optimal control problems for the parabolic equations. We refer
      to \cite{GY, WZheng} for  semi-discrete finite element approximations,
      and \cite{TWW,Y} for  perturbations of equations. About more works on time optimal control problems, we would
like to mention \cite{AEWZ,HOF,HOF1,GL,GuoXuYang,GuoYang,KL1,KL2,LiYong,{LZ},
PWCZ,PWZ,Wang,WX,WXZ,WCZ,WangZhang,WZ,CZ-2} and the
references therein.

\item About approximations of time optimal sampled-data controls,
we have not found any literature in the past publications.

\end{itemize}

\subsection{The strategy to get the main results}

 The  strategy to prove the main theorems
 is as follows: We first introduce
 two norm optimal control problems which correspond  to time optimal control problems $(TP)^{M,y_0}$ and $(TP)_\delta^{M,y_0}$ respectively; then get error estimates between the above two  norm optimal  control problems (in terms of $\delta$); finally,  obtain the desired error estimates between
 $(TP)^{M,y_0}$ and $(TP)_\delta^{M,y_0}$ (in terms of $\delta$),
 through using  connections between the time optimal control problems and the corresponding norm optimal control problems (see (iii) of Theorem \ref{Lemma-existences-TP} and Theorem \ref{Theorem-eq-cont-discrete}, respectively).

To explain our strategy more clearly, we will  introduce   two norm optimal control problems. The first one corresponds to $(TP)^{M,y_0}$ and is as:
   \begin{eqnarray}\label{0113-NP-0}
 (NP)^{T,y_0}:~
 N(T,y_0)\triangleq\inf\{\|v\|_{L^2(0,T;L^2(\Omega))}\;:\;y(T;y_0,v)\in B_r(0)\},
\end{eqnarray}
where $y_0\in L^2(\Omega)\setminus B_r(0)$, $T>0$ and $y(\cdot;y_0,v)$ is the solution of
(\ref{heat-1}) with $u$ being replaced by the zero extension of $v$ over $\mathbb{R}^+$. The second one corresponds to $(TP)_\delta^{M,y_0}$
and is defined by
\begin{eqnarray}\label{0113-NP-delta}
 (NP)_{\delta}^{k\delta,y_0}:~
 N_\delta(k\delta,y_0)\triangleq\inf\{\|v_\delta\|_{L_\delta^2(0,k\delta;L^2(\Omega))}
 ~:~y(k\delta;y_0,v_\delta)\in B_r(0)\},
\end{eqnarray}
where $y_0\in L^2(\Omega)\setminus B_r(0)$, $T>0$, $(\delta,k)\in \mathbb R^+ \times \mathbb N^+$,
\begin{eqnarray}\label{L2-delta}
 L_\delta^2(0,k\delta;L^2(\Omega))\triangleq \{f|_{(0,k\delta]}\;:\;f\in L_\delta^2(\mathbb R^+;L^2(\Omega))\},
 \end{eqnarray}
 and $y(\cdot;y_0,v_\delta)$ is the solution of
(\ref{heat-1}) with $u$ being replaced by the zero extension of $v_\delta$ over $\mathbb{R}^+$.

Some concepts about the above two norm optimal control problems are given in the following definition:
\begin{definition} (i) In the problem $(NP)^{T,y_0}$, $N(T,y_0)$ is called the optimal norm;  $v\in L^2(0,T;L^2(\Omega))$ is called an admissible control if $y(T;y_0,v)\in B_r(0)$;  $v^*$ is called an optimal control if
$y(T;y_0,v^*)\in B_r(0)$ and  $\|v^*\|_{L^2(0,T;L^2(\Omega))}=N(T,y_0)$.

(ii) In the problem $(NP)_\delta^{k\delta,y_0}$, $N_\delta(k\delta,y_0)$ is called the optimal norm;  $v_\delta\in L_\delta^2(0,k\delta;L^2(\Omega))$ is called an admissible control  if $y(k\delta;y_0,v_\delta)\in B_r(0)$; and $v_\delta^*$ is called an optimal control if
$y(k\delta;y_0,v_\delta^*)\in B_r(0)$ and $\|v_\delta^*\|_{L_\delta^2(0,k\delta;L^2(\Omega))}=N_\delta(k\delta,y_0)$.

\end{definition}
We mention that both $(NP)^{T,y_0}$ and $(NP)_{\delta}^{k\delta,y_0}$
have  unique nonzero minimizers in $L^2(\Omega)$ (see Theorems \ref{Lemma-Norm-fun}-\ref{Lemma-NP-bangbang}).

Inspired by \cite{C. Fabre}, we study the above two minimal norm control problems by
 two minimization problems. The first one  corresponds to $(NP)^{T,y_0}$ and reads
\begin{eqnarray}\label{fun-0}
(JP)^{T,y_0}:~
V(T,y_0)\triangleq\inf_{z\in L^2(\Omega)} J^{T,y_0}(z)
&\triangleq&\inf_{z\in L^2(\Omega)}\Big[\frac{1}{2} \|\chi_\omega\varphi(\cdot;T,z)\|_{L^2(0,T;L^2(\Omega))}^2 \nonumber\\
         & & +\langle y_0,\varphi(0;T,z) \rangle + r \|z\|\Big],
\end{eqnarray}
where $\varphi(\cdot;T,z)$ is the solution to the adjoint heat equation:
\begin{equation}\label{adjoint-1}
\left\{
\begin{array}{lll}
\partial_t \varphi+\Delta \varphi=0 &\mbox{in}&\Omega\times (0,T),\\
\varphi=0 &\mbox{on}&\partial \Omega\times(0,T),\\
\varphi(T)=z\in L^2(\Omega).
\end{array}
\right.
\end{equation}
The second minimization problem corresponds to
 $(NP)^{T,y_0}_\delta$ and is as:
\begin{eqnarray}\label{fun-P}
(JP)_{\delta}^{k\delta,y_0}:~
V_\delta(k\delta,y_0)\triangleq\inf_{z\in L^2(\Omega)} J_{\delta}^{k\delta,y_0}(z)
&\triangleq&\inf_{z\in L^2(\Omega)}\Big[\frac{1}{2} \|\chi_\omega\overline{\varphi}_\delta(\cdot;k\delta,z)\|_{L^2(0,k\delta;L^2(\Omega))}^2  \nonumber\\
         & & +\langle y_0,\varphi(0;k\delta,z) \rangle + r \|z\|\Big],
\end{eqnarray}
where $\overline{\varphi}_\delta(\cdot;k\delta,z)$ is defined by \begin{eqnarray}\label{adjoint-solution-P}
  \overline{\varphi}_\delta(t;k\delta,z)\triangleq\sum_{i= 1}^k \chi_{((i-1)\delta,i\delta]}(t)   \frac{1}{\delta} \int_{(i-1)\delta}^{i\delta} \varphi(s;k\delta,z) \,\mathrm ds
   \;\;\mbox{for each}\;\;    t\in (0,k\delta].
\end{eqnarray}
We mention that both $(JP)^{T,y_0}$ and $(JP)_{\delta}^{k\delta,y_0}$
have  unique nonzero minimizers in $L^2(\Omega)$ (see Theorems \ref{Lemma-Norm-fun}-\ref{Lemma-NP-bangbang}).

We prove  Theorem \ref{Theorem-time-order} by the following steps:
 \begin{itemize}
   %\item[(a)] Build up an equivalence between $(TP)^{M,y_0}$ and $(NP)^{T(M,y_0),y_0}$  (see Theorem \ref{Theorem-eq-cont-discrete});
%     Obtain  connections between   $(NP)^{T(M,y_0),y_0}$ and $(JP)^{T(M,y_0),y_0}$ (see Theorem \ref{Lemma-Norm-fun}).
%   \item[(b)]  By the equivalence and the connections obtained in (a), we present the time optimal control of $(TP)^{M,y_0}$ in terms of  the minimizer of $(JP)^{T(M,y_0),y_0}$.
%       With the aid of this, we construct an admissible control of $(TP)^{M,y_0}_\delta$, through   discretizing the optimal control of $(TP)^{M,y_0}$. By such construction, we get an estimate for the difference of the optimal time to $(TP)^{M,y_0}$
%        and the optimal time to $(TP)^{M,y_0}_\delta$, in terms of $\delta$ and the $H_0^1(\Omega)$-norm
%         of the minimizer of $(JP)^{T(M,y_0),y_0}$ (see (ii) of Theorem \ref{Lemma-minizer-P-L2-uni-bdd}).
%   \item[(c)]  By studying $(JP)^{T(M,y_0),y_0}$, we get an estimate for
%    the $H_0^1(\Omega)$-norm of the minimizer in terms of $M$, $r$ and $y_0$.
%    Combining this with the estimate obtained in (b), we are led to the desired   error estimate in Theorem \ref{Theorem-time-order}.
\item[(a)] Build up connections between $(TP)^{M,y_0}_\delta$ and  $(NP)^{T_\delta(M,y_0),y_0}_\delta$ (see (iii) of Theorem \ref{Lemma-existences-TP}); and connections between $(TP)^{M,y_0}$
    and $(NP)^{T(M,y_0),y_0}$ (see Theorem \ref{Theorem-eq-cont-discrete}).

    \item[(b)] Obtain the lower  and upper bounds for the derivative of the map $T\rightarrow N(T,y_0)$ (see Theorem \ref{Proposition-NP-Lip-T}).
    \item[(c)] Compute the error estimate between the map $T\rightarrow N(T,y_0)$ and the map  $k\delta\rightarrow N_\delta(k\delta,y_0)$ (see Theorem \ref{theorem-NP-delta-error}).
    \item[(d)]  Get (i) of Theorem \ref{Theorem-time-order}, with the aid of the above (a)-(c).
    \item[(e)] By using the above (a)-(c) again, we build up  sets $\mathcal{A}_{M,\eta}$ and obtain related properties in  Theorem \ref{theorem-optimality}, which leads to (ii) of Theorem \ref{Theorem-time-order}.
 \end{itemize}
The steps to prove Theorem \ref{Theorem-time-control-convergence} and Theorem~\ref{Theorem-time-control-convergence-1}
are as follows:
\begin{itemize}
  \item[(1)]   Build up  connections between  $(TP)^{M,y_0}_\delta$ and $(NP)^{T_\delta(M,y_0),y_0}_\delta$ (see (iii) of Theorem \ref{Lemma-existences-TP}),
      and connections between $(TP)^{M,y_0}$ and $(NP)^{T(M,y_0),y_0}$  (see Theorem \ref{Theorem-eq-cont-discrete}).
  \item[(2)] With the aid of the connections  obtained in (1), we can transfer the  estimate
  in (i) of
      Theorem \ref{Theorem-time-control-convergence} into an  estimate between optimal controls of $(NP)^{T(M,y_0),y_0}$ and $(NP)^{T_\delta(M,y_0),y_0}_\delta$.
  \item[(3)]  Find connections between  $(NP)^{T_\delta(M,y_0),y_0}_\delta$ (or
    $(NP)^{T(M,y_0),y_0}$) and $(JP)^{T_\delta(M,y_0),y_0}_\delta$ (or $(JP)^{T(M,y_0),y_0}$) (see Theorem \ref{Lemma-NP-bangbang} and Theorem \ref{Lemma-Norm-fun}, respectively).
  \item[(4)]   Obtain the error estimate between the minimizers of $(JP)^{T(M,y_0),y_0}$ and $(JP)^{T_\delta(M,y_0),y_0}_\delta$.
  \item[(5)] Using the connections obtained in (3) and the estimate obtained in (4), we get an
       error estimate between optimal controls of $(NP)^{T(M,y_0),y_0}$ and $(NP)^{T_\delta(M,y_0),y_0}_\delta$.
       This, along with results in (2), leads to the  estimate
       in (i) of  Theorem \ref{Theorem-time-control-convergence}.
       \item[(6)] Using connections obtained in (1) and (3), and using Theorem \ref{theorem-optimality}, we prove the estimate in (ii) of Theorem \ref{Theorem-time-control-convergence}.

   \item[(7)] Obtain the least order of the diameter of the  set $\mathcal{O}_{M,\delta}$ (in $
       L^2(0,T(M,y_0); L^2(\Omega)))$, in terms of $\delta$, (see Lemma \ref{0912-diam-O}). Here,
       \begin{equation}\label{geometic-OMdelta}
       \mathcal{O}_{M,\delta}\triangleq\{u_\delta|_{(0,T(M,y_0))}\; :\; u_\delta\;\;
       \mbox{is an optimal control to}\;\; (TP)^{M,y_0}_\delta\}.
       \end{equation}

   \item[(8)] Derive  the  estimates in  Theorem~\ref{Theorem-time-control-convergence-1}, with the aid of
       Lemma \ref{0912-diam-O} and the estimates in Theorem \ref{Theorem-time-control-convergence}.

   \end{itemize}

We would like to give the following note:
\begin{itemize}
    \item  The above introduced  strategy was  used to study other properties of time optimal distributed control problems (see, for instance,  \cite{TWW} and \cite{WZ}). It could be used to study  numerical approximations of time optimal distributed control problems, via
        discrete time optimal control problems.
\end{itemize}

The rest of the paper is organized as follows: Section 2 shows
a kind of approximate null controllability for the equation (\ref{heat-2}). Section 3 concerns with the existence and uniqueness of time optimal control problems. Section 4 provides some connections among time optimal control problems, norm optimal control problems and some minimization problems.  Section 5 presents several auxiliary estimates. Section 6 proves the main results. Section 7 (Appendix) gives two lemmas. The first one presents
an equivalence between controllability and observability in an abstract setting, which was taken from \cite{WangWangZ}. The second one gives an inequality which was quoted from \cite{PWX}. Since both \cite{WangWangZ} and \cite{PWX} have not appeared, we put them and their proofs in Appendix.

\bigskip

\section{ $L^2$-approximate null controllability with a cost}

In this section, we   present   a kind of approximate null controllability for the sampled-data controlled equation (\ref{heat-2}).
Such controllability will be defined in the next Definition~\ref{Def-ap-controllable-P} and
 will play a key role in getting some estimates in Section 5.

\begin{definition}\label{Def-ap-controllable-P}
  (i) Let $(\delta,k)\in\mathbb{R}^+\times\mathbb N^+$. Equation (\ref{heat-2}) is said to have the  $L^2$-approximate null controllability with a cost over $[0,k\delta]$,  if for any $\varepsilon>0$, there is  $C(\varepsilon,\delta,k)>0$  so that for each
   $y_0\in L^2(\Omega)$, there is  $u^{y_0}_\delta\in L_\delta^2(0,k\delta;L^2(\Omega))$ (see (\ref{L2-delta})) satisfying that
  \begin{eqnarray}\label{P-p-appro-controllable}
  \frac{1}{C(\varepsilon,\delta,k)} \|u^{y_0}_\delta\|_{L^2(0,k\delta;L^2(\Omega))}^2 + \frac{1}{\varepsilon} \|y(k\delta;y_0,u^{y_0}_\delta)\|^2
   \leq \|y_0\|^2.
 \end{eqnarray}

  (ii) Equation (\ref{heat-2}) is said to have the  $L^2$-approximate null controllability with a cost, if it has the  $L^2$-approximate null controllability with a cost over $[0,k\delta]$, for each
  $(\delta,k)\in\mathbb{R}^+\times\mathbb N^+$.

 \end{definition}

To prove  the  $L^2$-approximate null controllability with a cost for
Equation (\ref{heat-2}), we need some preliminaries.
  For each $f\in L^2(\mathbb R^+;L^2(\Omega))$ and $\delta>0$, we let
\begin{eqnarray}\label{adjoint-solution-P-1}
  \bar f_\delta(t)\triangleq\sum_{i= 1}^\infty \chi_{((i-1)\delta,i\delta]}(t)   \frac{1}{\delta} \int_{(i-1)\delta}^{i\delta} f(s) \,\mathrm ds
   \;\;\mbox{for each}\;\;    t\in \mathbb R^+.
\end{eqnarray}
 \begin{lemma}\label{lemma-control-op-dual}
For each $f,\,g\in L^2(\mathbb R^+;L^2(\Omega))$ and each   $\delta>0$,
\begin{eqnarray}\label{delta-dual-eq}
 \langle \bar f_\delta,g \rangle_{L^2(\mathbb R^+;L^2(\Omega))}
 =  \langle f, \bar g_\delta \rangle_{L^2(\mathbb R^+;L^2(\Omega))}
 = \langle \bar f_\delta, \bar g_\delta \rangle_{L^2(\mathbb R^+;L^2(\Omega))}.
\end{eqnarray}

\end{lemma}

\begin{proof}
Arbitrarily fix $\delta>0$ and $f,\,g\in L^2(\mathbb R^+;L^2(\Omega))$. To prove (\ref{delta-dual-eq}), it suffices to show
\begin{eqnarray}\label{0111-sec2-1}
 \langle \bar f_\delta, g \rangle_{L^2(\mathbb R^+;L^2(\Omega))}
 =\langle \bar f_\delta, \bar g_\delta \rangle_{L^2(\mathbb R^+;L^2(\Omega))}.
\end{eqnarray}
By (\ref{adjoint-solution-P-1}), one can directly check that
\begin{eqnarray*}
&& \langle \bar f_\delta, g \rangle_{L^2(\mathbb R^+;L^2(\Omega))}
 = \sum_{i=1}^\infty\langle \bar f_\delta, g\rangle_{L^2((i-1)\delta, i\delta; L^2(\Omega))}
  = \sum_{i=1}^\infty      \Big\langle \bar f_\delta (i\delta), \int_{(i-1)\delta}^{i\delta} g(t) \,\mathrm dt  \Big\rangle
  \nonumber\\
  &=& \sum_{i=1}^\infty      \big\langle \bar f_\delta (i\delta), \bar g_\delta (i\delta)  \big\rangle \delta
  =  \sum_{i=1}^\infty\langle \bar f_\delta, \bar g_\delta \rangle_{L^2((i-1)\delta, i\delta; L^2(\Omega))}
    = \langle \bar f_\delta, \bar g_\delta \rangle_{L^2(\mathbb R^+;L^2(\Omega))},
\end{eqnarray*}
which leads to (\ref{0111-sec2-1}). This ends the proof of this lemma.
\end{proof}

The following interpolation inequality plays an important  role in the proof of the  $L^2$-approximate null controllability with a cost.

\begin{lemma}\label{Theorem-ob-int-0}
 There exists  $C\triangleq C(\Omega,\omega)>0$ so that when  $0<S<T$,
\begin{eqnarray}\label{ob-int-0}
  \|\varphi(0;T,z)\|\leq e^{C(1+\frac{1}{T-S})} \|z\|^{1/2}  \Big\|\frac{1}{S}\int_0^{S}\varphi(t;T,z) \,\mathrm dt \Big\|_\omega^{1/2}\;\;\mbox{for all}\;\;z\in L^2(\Omega).
\end{eqnarray}

\end{lemma}

\begin{proof}
Let $0<S<T$. Arbitrarily fix $z\in L^2(\Omega)$. We define a function $f^z$ over $\Omega$ by
\begin{eqnarray}\label{ob-int-01}
 f^z \triangleq \frac{1}{S} \int_0^S e^{\Delta(S-t)}z \,\mathrm dt.
\end{eqnarray}
By \cite[(iii) of Theorem 2.1]{PWX}, (Since the paper \cite{PWX} has not appeared, we put \cite[(iii) of Theorem 2.1]{PWX} and its proof in Appendix, see Lemma \ref{lemma-1009-interpolation}.), there is $C\triangleq C(\Omega,\omega)>0$ so that
\begin{eqnarray}\label{ob-int-02}
 \|e^{\Delta(T-S)}f^z\| \leq e^{C(1+\frac{1}{T-S})}
 \|f^z\|^{1/2}  \|\chi_\omega e^{\Delta(T-S)}f^z\|^{1/2}.
\end{eqnarray}
Two facts are given in order: First, it follows from (\ref{ob-int-01}) that
\begin{eqnarray}\label{yuanyuan2.8}
 \|f^z\|\leq \|z\|
 \;\;\mbox{and}\;\;
  e^{\Delta(T-S)}f^z=\frac{1}{S}\int_0^{S}\varphi(t;T,z) \,\mathrm dt.
\end{eqnarray}
Second, write $\{\lambda_j\}_{j=1}^\infty$ for the family of all eigenvalues of $-\Delta$ with the zero Dirichlet boundary condition so that $\lambda_1<\lambda_2\leq\cdots$. Let  $\{e_j\}_{j=1}^\infty$ be the family of  the corresponding normalized eigenvectors. Let $z=\sum_{j=1}^\infty z_j e_j$ for some $\{z_j\}_{j=1}^\infty\in  l^2$. Then it follows that
\begin{eqnarray*}
 \frac{1}{S} \int_0^S e^{\Delta(T-t)} z \,\mathrm dt
 =\sum_{j=1}^\infty \Big(\frac{1}{S} \int_0^S e^{\lambda_j t} \,\mathrm dt\Big) e^{-\lambda_j T} z_j e_j.
\end{eqnarray*}
Since $\frac{1}{S} \int_0^S e^{\lambda_j t} \,\mathrm dt \geq 1$ for each $j\in\mathbb N^+$, it follows from (\ref{ob-int-01}) and the above equality that
\begin{eqnarray}\label{yuanyuan2.9}
\|e^{\Delta(T-S)}f^z\| &=&
 \Big\| \frac{1}{S} \int_0^{S} e^{\Delta (T-t)}z\, \mathrm dt\Big\|
  \geq  \|e^{\Delta T} z\|= \|\varphi(0;T,z)\|.
\end{eqnarray}
Finally, the facts (\ref{yuanyuan2.8}) and (\ref{yuanyuan2.9}), along with (\ref{ob-int-02}), lead to (\ref{ob-int-0}). This ends the proof.
\end{proof}
The next Theorem~\ref{Theorem-eq-controllable-observable} contains  the main results of this section. The conclusion (iii) in Theorem~\ref{Theorem-eq-controllable-observable} will play an important role in our further studies.
\begin{theorem}\label{Theorem-eq-controllable-observable}
The following conclusions are true:

 (i) Equation (\ref{heat-2}) has the  $L^2$-approximate null controllability with a cost
   if and only if  given $\varepsilon>0$, $\delta>0$ and $k\in \mathbb{N}^+$, there is $C(\varepsilon,\delta,k)>0$ (which also depends on $\Omega$ and $\omega$) so that
 \begin{eqnarray}\label{ob-0}
  \|\varphi(0;k\delta,z)\|^2 \leq C(\varepsilon,\delta,k) \|\chi_\omega \overline \varphi_\delta(\cdot;k\delta,z)\|_{L^2(0,k\delta;L^2(\Omega))}^2
  +\varepsilon\|z\|^2
  \;\;\mbox{for all}\;\;z\in L^2(\Omega),
 \end{eqnarray}
 where $\overline \varphi_\delta(\cdot;k\delta,z)$ is given by (\ref{adjoint-solution-P}).

  (ii) Given $\delta>0$ and $k\geq 2$, Equation (\ref{heat-2}) has the  $L^2$-approximate null controllability with a cost over $[0,k\delta]$.

 (iii) Given $\varepsilon>0$, $\delta>0$ and  $k\geq 2$, the  constants $C(\varepsilon,\delta,k)$ in (\ref{ob-0}) and (\ref{P-p-appro-controllable}) can be taken as
\begin{equation}\label{WGS2.13}
C(\varepsilon,\delta,k)=e^{C[1+1/(k\delta)]}/\varepsilon\;\;\mbox{with}\;\;C\triangleq C(\Omega,\omega).
\end{equation}
\end{theorem}
\begin{proof}
We first prove the conclusion (i). Arbitrarily fix $\delta>0$, $k\in \mathbb{N}^+$ and $\varepsilon>0$. We will put our problems under the framework of \cite[Lemma 5.1]{WangWangZ} in the following manner: Let
$X\triangleq
L^2(\Omega)$, $Y\triangleq
L^2_\delta(0,k\delta;L^2(\Omega))$ and $Z\triangleq
 L^2(\Omega)$.
 Define operators $ \mathcal R: Z\rightarrow X$ and $ \mathcal O: Z\rightarrow Y$ by
\begin{eqnarray*}
\mathcal R z  \triangleq  \varphi(0;k\delta,z)\;\;\mbox{and}\;\;\mathcal O z\triangleq  \chi_\omega \overline \varphi_\delta(\cdot;k\delta,z)\;\;\mbox{for all}\;\;z\in Z.
\end{eqnarray*}
One can directly check that $\mathcal R^*: X^*\rightarrow Z^*$ and $
\mathcal O^*: Y^*\rightarrow Z^*$ are given respectively  by
\begin{eqnarray*}
\mathcal R^* y_0=y(k\delta;y_0,0),~y_0\in L^2(\Omega);\;\;
\mathcal O^* u_\delta=y(k\delta;0,u_\delta),~u_\delta\in L^2(0,k\delta;L^2(\Omega)).
\end{eqnarray*}
From these,  Definition~\ref{Def-ap-controllable-P} and (\ref{ob-0}), we can apply
   \cite[Lemma 5.1]{WangWangZ} to get the conclusion (i) of Theorem~\ref{Theorem-eq-controllable-observable}.
 (Since the paper \cite{WangWangZ} has not appeared, we put \cite[Lemma 5.1]{WangWangZ} and its proof in Appendix, see Lemma \ref{lemma-0428-fn}.)

We next prove the conclusions (ii) and (iii).   Arbitrarily fix $\varepsilon>0$, $\delta>0$ and $k\geq 2$. By the conclusion (i), we find that it suffices to show (\ref{ob-0}) with the triplet $(\varepsilon,\delta,k)$. To this end,  we use   (\ref{ob-int-0}) (where $T=k\delta$ and $S=[k/2]\delta$, with $[k/2]$  the integer so that $k/2-1 < [k/2]\leq k/2$) to get that for each $z\in L^2(\Omega)$,
\begin{equation*}
 \|\varphi(0;k\delta,z)\|^2
 \leq  e^{2C(1+\frac{1}{k\delta-[k/2]\delta})}
  \Big\| \frac{1}{[k/2]\delta} \int_0^{[k/2]\delta}\varphi(t;k\delta,z) \,\mathrm dt \Big\|_\omega\|z\|,
\end{equation*}
where $C$ is given by (\ref{ob-int-0}).
Then by  Young's inequality, we find that for  each $z\in L^2(\Omega)$,
\begin{eqnarray}\label{wwang2.32}
\|\varphi(0;k\delta,z)\|^2
%&\leq&  e^{2C(1+\frac{1}{k\delta-[k/2]\delta})} \Big\| \frac{1}{[k/2]\delta} \int_0^{[k/2]\delta}\varphi(t;k\delta,z)\, \mathrm dt \Big\|_\omega  \|z\|
% \nonumber\\
  &\leq&    \frac{1}{\varepsilon} e^{4C(1+\frac{1}{k\delta-[k/2]\delta})} \frac{1}{([k/2]\delta)^2}
  \Big\|\int_0^{[k/2]\delta}\varphi(t;k\delta,z) \,\mathrm dt \Big\|_\omega^2
  + \varepsilon\|z\|^2. ~~~
 \end{eqnarray}
Two observations are given in order: First, it follows from (\ref{adjoint-solution-P}) that for each $z\in L^2(\Omega)$,
 \begin{eqnarray}\label{wwang2.33}
  & & \Big\|\int_0^{[k/2]\delta}\varphi(t;k\delta,z)\,\mathrm dt\Big\|_\omega
  \leq  \sum_{i=1}^{[k/2]} \Big\|\frac{1}{\delta}\int_{(i-1)\delta}^{i\delta}\varphi(t;k\delta,z)\, \mathrm dt \Big\|_\omega \delta
\nonumber\\
&=&
 \|\chi_\omega\overline\varphi_\delta(\cdot;k\delta,z)\|_{L^1(0,[k/2]\delta; L^2(\Omega))}
\leq  \sqrt{[k/2]\delta} \|\chi_\omega\overline\varphi_\delta(\cdot;k\delta,z)\|_{L^2(0,k\delta; L^2(\Omega))};
 \end{eqnarray}
 Second, since
\begin{eqnarray*}
k/4\leq [k/2]\leq k/2\;\;\mbox{and}\;\;{1}/([k/2]\delta)\leq e^{4+{1}/{k\delta}},
\end{eqnarray*}
 one can directly check that
\begin{eqnarray}\label{wwang2.331}
e^{4C\big(1+\frac{1}{k\delta-[k/2]\delta}\big)} \frac{1}{[k/2]\delta}
\leq  e^{8(C+1)(1+\frac{1}{k\delta})}.
\end{eqnarray}
 Finally, from  (\ref{wwang2.32}), (\ref{wwang2.33}) and (\ref{wwang2.331}), we get
 (\ref{ob-0}), with $C(\varepsilon,\delta,k)$  given by (\ref{WGS2.13}), where
 $C(\Omega,\omega)$ may differ from that in (\ref{wwang2.331}). This proves (ii), as well as (iii).

In summary, we end  the proof of Theorem~\ref{Theorem-eq-controllable-observable}.
\end{proof}

\section{Existence and uniqueness of optimal controls}

In this section, we will prove that for each $(M,y_0)$, $(TP)^{M,y_0}$ has the unique optimal control, while for some
 $(M,y_0,\delta)$,
  $(TP)^{M,y_0}_\delta$  has infinitely many optimal controls. The later may cause difficulties in our studies.  Fortunately, we observe that the optimal control with the minimal norm to $(TP)^{M,y_0}_\delta$ (see Definition \ref{wgsdefinition1.1}) is unique.
The first main theorem in this section is stated in the next Theorem~\ref{Lemma-existences-TP}.
It deserves mentioning what follows: The conclusion (iii) of Theorem~\ref{Lemma-existences-TP} should belong to the materials in the next section.  The reason that we put it here is that we will use it in the proof of
the non-uniqueness of optimal controls to $(TP)_\delta^{M,y_0}$. More precisely,
in the proof of Lemma~\ref{0806-exist-l2}, we will use it.

\begin{theorem}\label{Lemma-existences-TP}
Let $y_0\in L^2(\Omega)\setminus B_r(0)$. Let $M\geq 0$. The following conclusions are true:

(i) The problem   $(TP)^{M,y_0}$ has a unique optimal control.

 (ii) For  each   $\delta >0$, $(TP)_\delta^{M,y_0}$ has a unique  optimal control with the minimal norm.

 (iii) Let $u^*_\delta$ (with $\delta>0$) be the  optimal control with the minimal norm
to $(TP)_\delta^{M,y_0}$. Then $u^*_\delta|_{(0,T_\delta(M,y_0))}$
(the restriction of $u^*_\delta$ over $(0,T_\delta(M,y_0))$) is an optimal control to $(NP)_\delta^{T_\delta(M,y_0),y_0}$ and
the $L^2(0,T_\delta(M,y_0);L^2(\Omega))$-norm of $u^*_\delta$ is $N_\delta(T_\delta(M,y_0),y_0)$.
\end{theorem}

\begin{proof}
Arbitrarily fix $M\geq 0$ and $y_0\in L^2(\Omega) \setminus  B_r(0)$. We will prove conclusions (i), (ii) and (iii) one by one.

(i) Because $y(t;y_0,0)\rightarrow 0$ as $t\rightarrow\infty$,  the null control is an admissible control to  $(TP)^{M,y_0}$, which implies that $(TP)^{M,y_0}$ has an admissible control.
Then  by the standard way as that used in the proof of \cite[Lemma 1.1]{HOF}, one can show that $(TP)^{M,y_0}$ has an optimal control.

To show the uniqueness of the optimal control to $(TP)^{M,y_0}$, we first notice that each optimal control $u^*$ to $(TP)^{M,y_0}$ has the property that
\begin{eqnarray}\label{0322-eq-lemma-1}
 \|u^*\|_{L^2(0,T(M,y_0);L^2(\Omega))}=M.
\end{eqnarray}
(The  property (\ref{0322-eq-lemma-1})
can be proved by the same way as that used to show  \cite[Lemma 4.3]{GL}.)
 Next, we notice that if $u^*_1$ and $u^*_2$ are optimal controls to $(TP)^{M,y_0}$,
 then $(u^*_1+u^*_2)/2$ is also an optimal control to $(TP)^{M,y_0}$. From this,  (\ref{0322-eq-lemma-1}) and the parallelogram law in $L^2(0,T(M,y_0);L^2(\Omega))$, we can easily use the contradiction argument to get
 the uniqueness. This   ends the proof of the conclusion (i).

 (ii) Arbitrarily fix $\delta>0$.  We first show that $(TP)_\delta^{M,y_0}$ has an optimal control. Indeed, since the null control is clearly an admissible control
 to $(TP)^{M,y_0}_\delta$, it follows by the definition of $T_\delta(M,y_0)$ (see (\ref{time-2})) that there exists  $\hat k\in\mathbb N^+$ so that
\begin{eqnarray}\label{0118-sec3-17}
  T_\delta(M,y_0)=\hat k\delta.
\end{eqnarray}
Meanwhile, since $y_0\in L^2(\Omega)\setminus B_r(0)$,
 by the definition of the infimum in  (\ref{time-2}), we see that there is   $k_0 \in \mathbb N^+$ and   $u_\delta^0\in L^2_\delta(\mathbb R^+;L^2(\Omega))$ so that
\begin{equation}\label{0118-sec3-18}
 T_\delta(M,y_0) \leq k_0\delta \leq T_\delta(M,y_0) +\delta/2;
 \end{equation}
\begin{equation}\label{0118-sec3-18-1}
y(k_0\delta;y_0,u_\delta^0)\in B_r(0)
\;\;\mbox{and}\;\;
 \|u_\delta^0\|_{L^2(\mathbb R^+;L^2(\Omega))} \leq M.
\end{equation}
From  (\ref{0118-sec3-17}) and (\ref{0118-sec3-18}),
 we find that $\hat k\delta\leq k_0\delta\leq \hat k\delta+\delta/2$, which leads to that $k_0=\hat k$. This, along with (\ref{0118-sec3-17}) and  (\ref{0118-sec3-18-1}), yields that
$T_\delta(M,y_0)=k_0\delta\in(0,\infty)$,
which, together with  (\ref{0118-sec3-18-1}), implies that $u_\delta^0$ is an optimal control to  $(TP)_\delta^{M,y_0}$.

  Next, we will prove that  $(TP)_\delta^{M,y_0}$ has a unique
   optimal control with the minimal norm. Indeed, since $L^2_\delta(\mathbb R^+;L^2(\Omega))$ is a closed subspace of $L^2(\mathbb R^+;L^2(\Omega))$,  by  Definition \ref{wgsdefinition1.1},  one can use a standard way (i.e., taking a minimization sequence) to  show the existence of the optimal control with the minimal norm to $(TP)_\delta^{M,y_0}$. To show the uniqueness, we let $u_1$ and $u_2$ be two optimal controls with the minimal norm. By Definition~\ref{wgsdefinition1.1}, one can easily check that $(u_1+u_2)/2$ is also an optimal control with the minimal norm to $(TP)^{M,y_0}_\delta$. By making use of Definition~\ref{wgsdefinition1.1} again,   we find that
   \begin{eqnarray*}
   \|u_1\|_{L^2(0,T_\delta(M,y_0); L^2(\Omega))}=\|u_2\|_{L^2(0,T_\delta(M,y_0); L^2(\Omega))}=\|(u_1+u_2)/2\|_{L^2(0,T_\delta(M,y_0); L^2(\Omega))}.
   \end{eqnarray*}
   These, along with
   the parallelogram law for $L^2(0,T_\delta(M,y_0);L^2(\Omega))$, yields that
   \begin{eqnarray*}
   (u_1-u_2)/2=0\;\;\mbox{in}\;\;L^2(0,T_\delta(M_\delta,y_0);L^2(\Omega)),
   \;\;\mbox{i.e.,}\;\;u_1=u_2.
   \end{eqnarray*}
    So $(TP)_\delta^{M,y_0}$ has a unique optimal control with the minimal norm.

  (iii) Let  $u^*_\delta$ be  the optimal control with the minimal norm to $(TP)_\delta^{M,y_0}$.
  We will show that $u^*_\delta|_{(0,T_\delta(M,y_0)}$
     is  an optimal control to $(NP)_\delta^{T_\delta(M,y_0),y_0}$.
 Indeed, we have that
 \begin{eqnarray}\label{0118-sec4-301}
   y(T_\delta(M,y_0);y_0,u^*_\delta)\in B_r(0)
   \;\;\mbox{and}\;\;
   \|u^*_\delta\|_{L^2_\delta(\mathbb R^+;L^2(\Omega))}\leq M,
 \end{eqnarray}
 from which, one can easily check that $u^*_\delta|_{(0,T_\delta(M,y_0))}$  is an admissible control to $(NP)_\delta^{T_\delta(M,y_0),y_0}$. Then by the optimality of $N_\delta(T_\delta(M,y_0),y_0)$ and the second inequality in (\ref{0118-sec4-301}), we see that
 \begin{eqnarray}\label{TP-minimal-1}
  N_\delta(T_\delta(M,y_0),y_0)\leq
  \|u^*_\delta|_{(0,T_\delta(M,y_0))}\|_{L^2_\delta(L^2_\delta
  (0,T_\delta(M,y_0);L^2(\Omega))} \leq M<\infty.
 \end{eqnarray}
 Meanwhile, since $(NP)_\delta^{T_\delta(M,y_0),y_0}$ has an admissible control,
 we can use a standard argument (see for instance the proof of \cite[Lemma 1.1]{HOF}) to show that $(NP)_\delta^{T_\delta(M,y_0),y_0}$ has an optimal control $v^*_\delta$.  Write $\widetilde v^*_\delta$ for the zero extension of $v^*_\delta$ over $\mathbb R^+$.
  Then we have that
 \begin{eqnarray}\label{TP-minimal-2}
 y(T_\delta(M,y_0);y_0,\widetilde v^*_\delta)\in B_r(0)
  \;\;\mbox{and}
  \;\; \|\widetilde v^*_\delta\|_{L^2_\delta(\mathbb{R}^+; L^2(\Omega))}=N_\delta(T_\delta(M,y_0),y_0).
 \end{eqnarray}
  From  (\ref{TP-minimal-2}) and (\ref{TP-minimal-1}), it follows that  $\widetilde v^*_\delta$ is an optimal control to $(TP)_\delta^{M,y_0}$. Since $u^*_\delta$ is the optimal control with the minimal norm to $(TP)_\delta^{M,y_0}$, we see from (\ref{TP-minimal-1}), (ii) of Definition~\ref{wgsdefinition1.1} and the second equality in (\ref{TP-minimal-2}) that
 \begin{eqnarray*}
   N_\delta(T_\delta(M,y_0),y_0)
   &\leq& \|u^*_\delta\|_{L^2_\delta(0,T_\delta(M,y_0);L^2(\Omega))}
   \nonumber\\
   &\leq&
   \|\widetilde v^*_\delta\|_{L^2_\delta(0,T_\delta(M,y_0);L^2(\Omega))}
   =N_\delta(T_\delta(M,y_0),y_0).
 \end{eqnarray*}
 The above, together with the first conclusion in (\ref{0118-sec4-301}), implies that $u^*_\delta|_{(0,T_\delta(M,y_0))}$ is an optimal control to $(NP)_\delta^{T_\delta(M,y_0),y_0}$ and that
 \begin{eqnarray*}
 \|u^*_\delta\|_{L^2_\delta(0,T_\delta(M,y_0);L^2(\Omega))}
 = N_\delta(T_\delta(M,y_0),y_0).
 \end{eqnarray*}

  In summary, we end the proof of Theorem~\ref{Theorem-time-control-convergence-1}.
\end{proof}
The next theorem  concerns with the non-uniqueness of optimal controls to $(TP)_{\delta}^{M,y_0}$.
\begin{theorem}\label{Lemma-nonunique-TP}
Let $y_0\in L^2(\Omega)\setminus B_r(0)$. Then  there are    sequences $\{M_n\}$ dense in $\mathbb R^+$ and   $\{\delta_n\}\subset\mathbb R^+$, with $\lim_{n\rightarrow\infty} \delta_n=0$, so that for each $n$, the problem $(TP)_{\delta_n}^{M_n,y_0}$ has infinitely many different optimal controls.
\end{theorem}
To prove Theorem \ref{Lemma-nonunique-TP}, we need  two lemmas.
\begin{lemma}\label{0806-exist-l1}
Let $y_0\in L^2(\Omega)\setminus B_r(0)$. Then for each $(M,\delta)\in \mathbb R^+\times \mathbb R^+$, with $2\delta\leq T_\delta(M,y_0)<\infty$, it stands that
 \begin{eqnarray}\label{eq-delta-00}
  N_\delta(T_\delta(M,y_0),y_0) \leq M <  N_\delta(T_\delta(M,y_0)-\delta,y_0).
 \end{eqnarray}
\end{lemma}
\begin{proof}
Let $(M,\delta)\in \mathbb R^+\times \mathbb R^+$ so that $2\delta\leq T_\delta(M,y_0)<\infty$. Then by (\ref{time-2}), we see  that
 \begin{eqnarray}\label{0117-step2-7}
  T_\delta(M,y_0) = \hat k \delta
  \;\;\mbox{for some integer}\;\;
  \hat k\geq 2.
 \end{eqnarray}
Thus,  (\ref{eq-delta-00}) is equivalent to the following inequality:
 \begin{eqnarray}\label{0117-step2-6}
   N_\delta(\hat k\delta,y_0) \leq M < N_\delta((\hat k-1)\delta,y_0) .
 \end{eqnarray}
 To prove (\ref{0117-step2-6}), we  let $u_\delta^1$ be an optimal control to $(TP)_\delta^{M,y_0}$.   Then we have that
 \begin{eqnarray}\label{0117-step2-11}
  \|u_\delta^1\|_{L_\delta^2(\mathbb R^+;L^2(\Omega))} \leq M
   \;\;\mbox{and}\;\;
   y(T_\delta(M,y_0);y_0,u_\delta^1)\in B_r(0).
 \end{eqnarray}
According to (\ref{0117-step2-11}) and (\ref{0117-step2-7}), $u_\delta^1|_{(0,\hat k\delta)}$ is an admissible control to $(NP)_\delta^{\hat k\delta,y_0}$. Then by the optimality of $N_\delta(\hat k\delta,y_0)$ and the first inequality in (\ref{0117-step2-11}), we get that
 \begin{eqnarray*}
   N_\delta(\hat k\delta,y_0) \leq \|u_\delta^1\|_{L_\delta^2(0,\hat k\delta;L^2(\Omega))} \leq M,
 \end{eqnarray*}
 which leads to the first inequality in (\ref{0117-step2-6}).

 We now show the second inequality in (\ref{0117-step2-6}). By contradiction, we suppose that
 \begin{eqnarray}\label{0118-sec3-50}
  N_\delta((\hat k-1)\delta,y_0) \leq M.
 \end{eqnarray}
  Then we would obtain  from (\ref{0118-sec3-50})  that $(NP)_\delta^{(\hat k-1)\delta,y_0}$ has an admissible control, since $M<\infty$. Thus, by a standard way (see for instance the proof of \cite[Lemma 1.1]{HOF}), one can prove that $(NP)_\delta^{(\hat k-1)\delta,y_0}$ has
  an optimal control $v_\delta^1$. Hence,
      \begin{eqnarray}\label{0117-step2-8}
 \|v_\delta^1\|_{L_\delta^2(0,(\hat k-1)\delta;L^2(\Omega))} =N_\delta((\hat k-1)\delta,y_0)
   \;\;\mbox{and}\;\;
   y((\hat k-1)\delta;y_0,v_\delta^1)\in B_r(0).
 \end{eqnarray}
 Write $\widetilde v_\delta^1$ for the zero extension of $v_\delta^1$ over $\mathbb R^+$. From (\ref{0117-step2-8}) and (\ref{0118-sec3-50}), we find that $\widetilde v_\delta^1$ is an admissible control to $(TP)_\delta^{M,y_0}$. Then by the optimality of $T_\delta(M,y_0)$, we get  that
 $T_\delta(M,y_0) \leq (\hat k-1)\delta$,
  which contradicts (\ref{0117-step2-7}). Thus, the second inequality in (\ref{0117-step2-6}) is true.
 We end the proof of this lemma.
\end{proof}
\begin{lemma}\label{0806-exist-l2}
Let $y_0\in L^2(\Omega)\setminus B_r(0)$. Then for each $M>0$ and $N>0$, there exists an integer $n\geq N$ so that
$2/2^n \leq T_{1/2^n}(M,y_0)<\infty$.
\end{lemma}
\begin{proof}
Arbitrarily fix $y_0\in L^2(\Omega)\setminus B_r(0)$.
It is clear that $T_{1/2^n}(M,y_0)<\infty$ for all $M>0$ and $n\in \mathbb{N}^+$. Thus, we only need to show that for any $M>0$ and $N>0$, $2/2^n \leq T_{1/2^n}(M,y_0)$, when $n\geq N$. By contradiction,  suppose that it were not true. Then there would be   $M>0$ and $N>0$ so that
\begin{eqnarray}\label{0806-exist-l2-pf-1}
 T_{1/2^n}(M,y_0) < 2/2^n\;\;\mbox{for all}\;\;n\geq N.
 \end{eqnarray}
Let $u^*_n$, with $n\geq N$, be an optimal control to $(TP)_{1/2^n}^{M,y_0}$ (see   (ii) of Theorem \ref{Lemma-existences-TP}). Then we have that
\begin{eqnarray}\label{wang3.16}
 y(T_{1/2^n}(M,y_0);y_0,u_n^*)\in B_r(0)
 \;\;\mbox{and}\;\;
 \|u^*_n\|_{L^2(\mathbb R^+;L^2(\Omega))} \leq M\;\;\mbox{for all}\;\;n\geq N.
\end{eqnarray}
  By the last inequality in (\ref{wang3.16}),
  H\"{o}lder's inequality  and (\ref{0806-exist-l2-pf-1}), we can easily check that
\begin{eqnarray*}
\int_0^{T_{1/2^n}(M,y_0)}e^{\Delta(T_{1/2^n}(M,y_0)-t}\chi_\omega u^*_n(t)dt\rightarrow 0,\;\;\mbox{as}\;\;n\rightarrow\infty.
\end{eqnarray*}
This, along with (\ref{0806-exist-l2-pf-1}) and the first  conclusion in (\ref{wang3.16}), yields that
\begin{eqnarray*}
y_0=\lim_{n\rightarrow\infty} y(T_{1/2^n}(M,y_0);y_0,u_n^*)\in B_r(0),
\end{eqnarray*}
which contradicts the assumption that $y_0\in L^2(\Omega)\setminus B_r(0)$.  This ends the proof.
\end{proof}
We are now on the position to prove Theorem \ref{Lemma-nonunique-TP}.
\begin{proof}[Proof of Theorem \ref{Lemma-nonunique-TP}]
Choose a sequence $\{M_n\}_{n=1}^\infty$ dense in $\mathbb R^+$ so that
\begin{eqnarray}\label{0806-th2.3-1}
  \{M_n\}_{n=1}^\infty
  \subset
  \mathbb R^+ \setminus
  \{N_{1/2^k}(j/2^k,y_0)~:~k,\,j\in\mathbb N^+\}.
\end{eqnarray}
By Lemma \ref{0806-exist-l2}, there exists an increasing subsequence  $\{k_n\}_{n=1}^\infty$ (in $\mathbb N^+$), with $\lim_{n\rightarrow\infty}k_n=\infty$, so that
\begin{eqnarray}\label{0806-th2.3-2}
 2/2^{k_n}\leq T_{1/2^{k_n}}(M_n,y_0)<\infty
 \;\;\mbox{for each}\;\;
 n\in\mathbb N^+.
\end{eqnarray}
Write $\delta_n\triangleq1/2^{k_n}$, $n\in \mathbb{N}^+$.
Then, by (\ref{0806-th2.3-2}), we can apply  Lemma \ref{0806-exist-l1} to get that
\begin{eqnarray*}
 N_{\delta_n}(T_{\delta_n}(M_n,y_0),y_0)\leq M_n
 <N_{\delta_n}(T_{\delta_n}(M_n,y_0)-\delta_n,y_0).
\end{eqnarray*}
This, along with (\ref{0806-th2.3-1}), yields that
\begin{eqnarray}\label{0806-th2.3-3}
 N_{\delta_n}(T_{\delta_n}(M_n,y_0),y_0)< M_n
 <N_{\delta_n}(T_{\delta_n}(M_n,y_0)-\delta_n,y_0).
\end{eqnarray}

The key to show Theorem \ref{Lemma-nonunique-TP} is to  claim that for each $n\in \mathbb{N}^+$, $(TP)^{M_n,y_0}_{\delta_{n}}$  has at least two different optimal controls.
By contradiction, we suppose that for some $n_0\in\mathbb N^+$, $(TP)^{M_{n_0},y_0}_{\delta_{n_0}}$ had a unique optimal control.
To get a contradiction, we define two convex subsets in $L^2(\Omega)$ as follows:
\begin{eqnarray*}
 A_{n_0} &\triangleq& \left\{y(T_{\delta_{n_0}}(M_{n_0},y_0);y_0,u_{\delta_{n_0}})
 ~:~\|u_{\delta_{n_0}}\|_{L^2_{\delta_{n_0}}(\mathbb R^+;L^2(\Omega))} \leq N_{\delta_{n_0}}(T_{\delta_{n_0}}(M_{n_0},y_0),y_0)\right\},
 \nonumber\\
 B_{n_0} &\triangleq& \left\{y(T_{\delta_{n_0}}(M_{n_0},y_0);0,v_{\delta_{n_0}})
 ~:~\|v_{\delta_{n_0}}\|_{L^2_{\delta_{n_0}}(\mathbb R^+;L^2(\Omega))} \leq M_{n_0}-N_{\delta_{n_0}}(T_{\delta_{n_0}}(M_{n_0},y_0),y_0)\right\}.
\end{eqnarray*}
We first show that
\begin{eqnarray}\label{WANG3.21}
A_{n_0}\cap B_r(0)=\{\hat\eta\}\;\;\mbox{for some}\;\;\hat\eta\in L^2(\Omega),
\end{eqnarray}
i.e.,   $A_{n_0}\cap B_r(0)$ contains  only one element.
In fact, by (ii) and (iii) of Theorem \ref{Lemma-existences-TP}, the optimal control with the minimal norm $u^*_{\delta_{n_0}}$ to $(TP)^{M_{n_0},y_0}_{\delta_{n_0}}$ satisfies that
\begin{eqnarray*}
 y(T_{\delta_{n_0}}(M_{n_0},y_0);y_0,u^*_{\delta_{n_0}})\in B_r(0)
 \;\;\mbox{and}\;\;
 \|u^*_{\delta_{n_0}}\|_{L^2_{\delta_{n_0}}(\mathbb R^+;L^2(\Omega))} = N_{\delta_{n_0}}(T_{\delta_{n_0}}(M_{n_0},y_0),y_0).
\end{eqnarray*}
These imply that $A_{n_0}\cap B_r(0)\neq\emptyset$.
We next show that $A_{n_0}\cap B_r(0)$ contains only one element. Suppose, by contradiction, that it contained two different elements $y_1$ and $y_2$. Then by
the definition of $A_{n_0}$, there would be two different controls $u_1$ and $u_2$ so that
\begin{eqnarray}\label{wang3.21}
y_1=y(T_{\delta_{n_0}}(M_{n_0},y_0);y_0,u_1),\;\;\|u_1\|_{L^2_{\delta_{n_0}}(\mathbb R^+;L^2(\Omega))}\leq N_{\delta_{n_0}}(T_{\delta_{n_0}}(M_{n_0},y_0),y_0);
\end{eqnarray}
\begin{eqnarray}\label{wang3.22}
y_2=y(T_{\delta_{n_0}}(M_{n_0},y_0);y_0,u_2),\;\;\|u_2\|_{L^2_{\delta_{n_0}}(\mathbb R^+;L^2(\Omega))}\leq N_{\delta_{n_0}}(T_{\delta_{n_0}}(M_{n_0},y_0),y_0).
\end{eqnarray}
Since $y_1, y_2\in B_r(0)$, we have that ${(y_1+y_2)}/{2}\in B_r(0)$. From this  (\ref{wang3.21}) and (\ref{wang3.22}), one can easily check that
\begin{eqnarray}\label{wang3.23}
(y_1+y_2)/{2}\in A_{n_0}\cap B_r(0).
\end{eqnarray}
Meanwhile, since $u_1\neq u_2$, by the second inequality in (\ref{wang3.21}) and the second inequality in (\ref{wang3.22}), using the parallelogram law, we find that
\begin{eqnarray}\label{wang3.24}
\left\|(u_1+u_2)/{2}\right\|_{L^2_{\delta_{n_0}}(\mathbb R^+;L^2(\Omega))}<N_{\delta_{n_0}}(T_{\delta_{n_0}}(M_{n_0},y_0),y_0),
\end{eqnarray}
which, together with (\ref{0806-th2.3-3}), indicates that
\begin{eqnarray*}
\|(u_1+u_2)/{2}\|_{L^2_{\delta_{n_0}}(\mathbb R^+;L^2(\Omega))}<M_{n_0}.
\end{eqnarray*}
From this and  (\ref{wang3.23}), we see that $(u_1+u_2)/2$ is an optimal control to $(TP)^{M_{n_0},y_0}_{\delta_{n_0}}$. This, along with Definition~\ref{wgsdefinition1.1} and the conclusions (ii) and (iii) of Theorem~\ref{Lemma-existences-TP}, yields that
\begin{eqnarray*}
\left\|(u_1+u_2)/{2}\right\|_{L^2_{\delta_{n_0}}(\mathbb R^+;L^2(\Omega))}\geq N_{\delta_{n_0}}(T_{\delta_{n_0}}(M_{n_0},y_0),y_0),
\end{eqnarray*}
which contradicts (\ref{wang3.24}). Hence,  (\ref{WANG3.21}) is true.

Next, by the definitions of  $A_{n_0}$ and $B_{n_0}$, one can easily check that
each element of $(A_{n_0}+B_{n_0})\cap B_r(0)$ can be expressed as:
\begin{eqnarray}\label{wang3.27}
y(T_{\delta_{n_0}}(M_{n_0},y_0);y_0,u_{n_0}^*)\;\;\mbox{with}\;\;u_{n_0}^*
\;\;\mbox{an optimal control to}\;\; (TP)^{M_{n_0},y_0}_{\delta_{n_0}}.
\end{eqnarray}
Since it was assumed that $(TP)^{M_{n_0},y_0}_{\delta_{n_0}}$ had a unique optimal control,
it follows from (\ref{wang3.27}) that $(A_{n_0}+B_{n_0})\cap B_r(0)$ contains only one element. This, along with (\ref{WANG3.21}), yields that
\begin{eqnarray}\label{0806-th2.3-4}
 (A_{n_0}+B_{n_0})\cap B_r(0)=A_{n_0}\cap B_r(0)=\{\hat\eta\}
 \;\;\mbox{for some}\;\;
 \hat \eta\in L^2(\Omega).
\end{eqnarray}
By (\ref{0806-th2.3-4}), we can apply  the Hahn-Banach separation theorem to find  $\eta^*\in L^2(\Omega)$, with $\|\eta^*\|=r>0$, so that
\begin{eqnarray*}
\sup_{w\in A_{n_0}+B_{n_0}} \langle w,\eta^* \rangle
 \leq \inf_{z\in B_r(0)} \langle z,\eta^* \rangle.
 \end{eqnarray*}
This, along with  (\ref{0806-th2.3-4}), yields that
\begin{eqnarray}\label{wang3.30}
 \sup_{w\in \hat \eta+B_{n_0}} \langle w,\eta^* \rangle
 \leq   \langle \hat \eta,\eta^* \rangle,
 \;\;\mbox{i.e.,}\;\;
 \sup_{w\in B_{n_0}} \langle w,\eta^* \rangle\leq 0.
\end{eqnarray}

{\it From now on and throughout the proof of Theorem~\ref{Lemma-nonunique-TP},
we simply write $T_{\delta_{n_0}}$ for $T_{\delta_{n_0}}(M_{n_0},y_0)$;
  simply write $\varphi(\cdot)$ and $\overline\varphi_{\delta_{n_0}}(\cdot)$ for
$\varphi(\cdot;T_{\delta_{n_0}},\eta^*)$ (see (\ref{adjoint-1}))
and $\overline\varphi_{\delta_{n_0}}(t;T_{\delta_{n_0}},\eta^*)$ (see (\ref{adjoint-solution-P})), respectively.}

 Arbitrarily fix $u_{\delta_{n_0}}\in L^2_{\delta_{n_0}}(\mathbb R^+;L^2(\Omega))$.
Three facts are given in order. Fact one: Since  $M_{n_0}>N_{\delta_{n_0}}(T_{\delta_{n_0}},y_0)$ (see (\ref{0806-th2.3-3})), it follows from the definition of $B_{n_0}$
that
\begin{eqnarray*}
 y(T_{\delta_{n_0}};0,u_{\delta_{n_0}})\in \lambda B_{n_0},\;\;\mbox{with}\;\;
 \lambda=\frac{\|u_{\delta_{n_0}}\|_{L^2_{\delta_{n_0}}(\mathbb{R}^+;L^2(\Omega))}}
 {M_{n_0}-N_{\delta_{n_0}}(T_{\delta_{n_0}},y_0)}.
\end{eqnarray*}
This, along with (\ref{wang3.30}), yields that
\begin{eqnarray}\label{Wang3.31}
\langle y(T_{\delta_{n_0}};0,u_{\delta_{n_0}}), \eta^*\rangle\leq 0.
\end{eqnarray}
Fact two: One can directly check that
\begin{eqnarray}\label{wang3.32}
\langle u_{\delta_{n_0}},\chi_\omega \varphi\rangle_{L^2(0,T_{\delta_{n_0}}; L^2(\Omega))}
 =\langle y(T_{\delta_{n_0}};0,u_{\delta_{n_0}}), \eta^*\rangle.
\end{eqnarray}
Fact three:  we have that
\begin{eqnarray}\label{wang3.33}
\langle u_{\delta_{n_0}},\chi_\omega \varphi\rangle_{L^2(0,T_{\delta_{n_0}}; L^2(\Omega))}
=\langle u_{\delta_{n_0}},\chi_\omega \overline\varphi_{\delta_{n_0}}\rangle_{L^2(0,T_{\delta_{n_0}}; L^2(\Omega))}.
\end{eqnarray}
 The proof of (\ref{wang3.33}) is as follows: Let $f=u_{\delta_{n_0}}$ and let
$g$ be the zero extension of $\chi_\omega\varphi(\cdot)$
over $\mathbb{R}^+$. Since $u_{\delta_{n_0}}\in L^2_{\delta_{n_0}}(\mathbb R^+;L^2(\Omega))$, it follows by (\ref{heat-Linfty-delta}), (\ref{adjoint-solution-P-1}) and
(\ref{adjoint-solution-P}) that $\bar f_{\delta_{n_0}}=u_{\delta_{n_0}}$
and $\bar g_{\delta_{n_0}}(\cdot)
=\overline\varphi_{\delta_{n_0}}(\cdot)$
(where $\overline\varphi_{\delta_{n_0}}(\cdot)$ is treated as its zero extension over $\mathbb{R}^+$).
 Then, by Lemma \ref{lemma-control-op-dual},
we obtain  (\ref{wang3.33}).

Now, from facts (\ref{Wang3.31}),  (\ref{wang3.32}) and  (\ref{wang3.33}), we see that
\begin{eqnarray*}
  \langle u_{\delta_{n_0}},\chi_\omega \overline\varphi_{\delta_{n_0}}\rangle_{L^2(0,T_{\delta_{n_0}}; L^2(\Omega))} \leq 0.
\end{eqnarray*}
Since $u_{\delta_{n_0}}$ was arbitrarily taken from $L^2_{\delta_{n_0}}(\mathbb R^+;L^2(\Omega))$, the above inequality  implies that
\begin{eqnarray}\label{0806-th2.3-6}
  \chi_\omega \overline\varphi_{\delta_{n_0}}(t)
  =0
  \;\;\mbox{a.e.}\;\;
  t\in(0,T_{\delta_{n_0}}).
\end{eqnarray}
Since $T_{\delta_{n_0}}\geq 2\delta_{n_0}$ (see (\ref{0806-th2.3-2})),   we apply (\ref{0806-th2.3-6}) and Lemma \ref{Theorem-ob-int-0} (where $T=T_{\delta_{n_0}}$ and $S=\delta_{n_0}$) to get that
$\varphi(0)=0$ in $L^2(\Omega)$.
Then from the backward uniqueness property for the heat equation (see, for instance, \cite{FHLin}), we deduce that $\eta^*=0$.
This leads to a contradiction. Hence, we ends the proof of the key claim: For each $n\in \mathbb{N}^+$, $(TP)^{M_n,y_0}_{\delta_{n}}$ has at least two different  optimal controls.

Finally, we observe that any convex combination of optimal controls to $(TP)^{M_n,y_0}_{\delta_{n}}$ (with $n\in\mathbb N^+$) is still an optimal control to $(TP)^{M_n,y_0}_{\delta_{n}}$. Therefore, for each $n\in\mathbb N^+$, $(TP)^{M_n,y_0}_{\delta_{n}}$ has infinitely many different optimal controls.
This ends the proof of Theorem~\ref{Lemma-nonunique-TP}.
\end{proof}

\section{Connections among different problems}
This section presents  connections among  $(TP)^{M,y_0}_\delta$, $(NP)^{k\delta,y_0}_\delta$ and $(JP)_\delta^{k\delta,y_0}$ (and  among
$(TP)^{M,y_0}$, $(NP)^{T,y_0}$ and $(JP)^{T,y_0}$).
We define,   for each $y_0\in L^2(\Omega)\setminus B_r(0)$,
\begin{eqnarray}\label{T*}
 T^*_{y_0}\triangleq \sup\{t> 0\;:\;e^{\Delta t}y_0\not\in B_r(0)\};
\end{eqnarray}
\begin{eqnarray}\label{P-T*}
 \mathcal P_{T^*_{y_0}}\triangleq\{(\delta,k)\;:\;\delta>0,~k\in\mathbb N^+
 \;\;\mbox{s.t.}\;\;  2\delta\leq k\delta<T^*_{y_0}\}.
\end{eqnarray}
 We mention that $0<T^*_{y_0}<\infty$ for each $y_0\in L^2(\Omega)\setminus B_r(0)$ (because the semigroup $\{e^{\Delta t}\}_{t\geq 0}$ has the  exponential decay).

\subsection{Connections between time optimal control problems and norm optimal control problems}

We first present the following  equivalence theorem. We will omit its proof, because it can be proved by the same way as one of proofs of
 \cite[Proposition 4.1]{TWW}, \cite[Proposition 3.1]{Y} and
 \cite[Theorem 1.1 and Theorem 2.1]{WZ}.

\begin{theorem}\label{Theorem-eq-cont-discrete}
Let $y_0\in L^2(\Omega)\setminus B_r(0)$. Let $T_{y_0}^*$ be given by (\ref{T*}).
Then the
following conclusions are true:

(i) The
 function $N(\cdot,y_0)$ is strictly decreasing and continuous  from $(0,T^*_{y_0})$ onto $(0,+\infty)$. Moreover, $\lim_{T\rightarrow {T^*_{y_0}}^-} N(T,y_0)=0$.

   (ii) When $M>0$ and  $T\in(0,T^*_{y_0})$,
 $N(T(M,y_0),y_0)=M$ and $T(N(T,y_0),y_0)=T$.

   (iii) The function $T(\cdot,y_0)$ is strictly decreasing and continuous
   from $(0,+\infty)$ onto $(0,T^*_{y_0})$.

   (iv) For each $M>0$, the optimal control to  $(TP)^{M,y_0}$, when restricted on
   $(0, T(M,y_0))$, is the optimal control to $(NP)^{T(M,y_0),y_0}$. For each $T\in
   (0,T^*_{y_0})$, the zero extension of the  optimal control to $(NP)^{T(M,y_0),y_0}$ is the optimal control to  $(TP)^{M,y_0}$.
\end{theorem}
We next recall (iii) of Theorem \ref{Lemma-existences-TP} for the connections between
 $(TP)_\delta^{M,y_0}$ and $(NP)_\delta^{T_\delta(M,y_0),y_0}$.

\subsection{Connections between norm optimal control problems and the minimization problems }

The first theorem of this subsection concerns with connections between  problems $(NP)^{T,y_0}$ and   $(JP)^{T,y_0}$ (given by (\ref{fun-0})).
  Its proof can be done by  the same methods as those in the proofs of  of Lemma 3.5 and Proposition 3.6 in \cite{TWW}. We omit it here.

\begin{theorem}\label{Lemma-Norm-fun}
Let $y_0\in L^2(\Omega)\setminus B_r(0)$ and let $T\in (0,T^*_{y_0})$,
with $T_{y_0}^*$  given by (\ref{T*}).
 Then the following conclusions are true:

(i) The functional $J^{T,y_0}$ has a unique nonzero minimizer $z^*$ in $L^2(\Omega)$.

(ii) Problem $(NP)^{T,y_0}$ has a unique optimal control $v^*$, which satisfies that
\begin{eqnarray}\label{op-control-0}
 v^*(t) = \chi_\omega \varphi(t;T,z^*)
 \;\;\mbox{a.e.}\;\;   t\in (0,T),
\end{eqnarray}
and that
\begin{eqnarray}\label{op-final-dual}
 y(T;y_0,v^*)=-r z^*/\|z^*\|.
\end{eqnarray}

(iii) It holds that $V(T,y_0)=-\frac{1}{2} N(T,y_0)^2=-\frac{1}{2} \|\chi_\omega\varphi(\cdot;T,z^*)\|_{L^2(0,T;L^2(\Omega))}^2$.
\end{theorem}
The next theorem deals with connections between  $(NP)^{k\delta,y_0}_\delta$ (given by (\ref{0113-NP-delta}))
and $(JP)^{k\delta,y_0}_\delta$ (given by (\ref{fun-P})). Recall (\ref{adjoint-solution-P})
for the definition of  $\overline \varphi(\cdot;k\delta,z)$.

\begin{theorem}\label{Lemma-NP-bangbang}
Let $y_0\in L^2(\Omega)\setminus B_r(0)$. Let  $(\delta,k)\in \mathcal P_{T^*_{y_0}}$ (given (\ref{P-T*})). Then the following conclusions are true:

(i) The functional $J_{\delta}^{k\delta,y_0}$ has a unique  minimizer $z_\delta^*$ in $L^2(\Omega)$. Moreover, $z_\delta^*\neq 0$ and
\begin{eqnarray}\label{nonzero-varphi-minimizer}
\chi_\omega\overline\varphi_\delta(t;k\delta,z_\delta^*)\neq 0  \;\;\mbox{for all}\;\;
 t\in \big(0,(k-1)\delta \big].
\end{eqnarray}

(ii) Problem $(NP)_{\delta}^{k\delta,y_0}$ has a unique optimal control $v_\delta^*$,  which verifies that
\begin{eqnarray}\label{PMP-Ndelta-0}
   v_\delta^*(t)=\chi_\omega\overline\varphi_\delta(t;k\delta,z_\delta^*)
   \;\;\mbox{for all}\;\;  t\in (0,k\delta],
 \end{eqnarray}
 (where $z_\delta^*$ is the minimizer of $J_{\delta}^{k\delta,y_0}(\cdot)$)
 and that
 \begin{eqnarray}\label{op-control-final-state}
  y(k\delta;y_0,v_\delta^*)= -r z_\delta^*/\|z_\delta^*\|.
 \end{eqnarray}

 (iii)  $V_\delta(k\delta,y_0)=-\frac{1}{2} N_\delta(k\delta,y_0)^2=-\frac{1}{2} \|\chi_\omega\overline\varphi_\delta(\cdot;k\delta, z_\delta^*)\|_{L^2(0,k\delta;L^2(\Omega))}^2$.
 \end{theorem}

\begin{proof}
(i)  First of all, we show the existence of minimizers of $(JP)_{\delta}^{k\delta,y_0}$. Indeed,
 by (\ref{fun-P}), one can easily see that  $J_{\delta}^{k\delta,y_0}$ is continuous and convex over   $L^2(\Omega)$. We now show its  coercivity.
  Since  $(\delta,k)\in \mathcal P_{T^*_{y_0}}$, we have that $k\geq 2$ (see
  (\ref{P-T*})). Thus, we can apply Theorem~\ref{Theorem-eq-controllable-observable} to see that both (\ref{ob-0}) and (\ref{WGS2.13}) are true.
  By taking $\varepsilon=\big(\frac{r}{2\|y_0\|}\big)^2$ in (\ref{ob-0}), we find that for each $z\in L^2(\Omega)$,
\begin{eqnarray*}
  \|\varphi(0;k\delta,z)\|^2
 &\leq& e^{C(1+\frac{1}{k\delta})} \Big(\frac{2\|y_0\|}{r}\Big)^2 \|\chi_\omega\overline\varphi_\delta(\cdot;k\delta,z)\|^2_{L^2(0,k\delta;L^2(\Omega))}  + \Big(\frac{r}{2\|y_0\|}\Big)^2 \|z\|^2
 \nonumber\\
 &\leq&  \Big( e^{\frac{C}{2}(1+\frac{1}{k\delta})} \frac{2\|y_0\|}{r} \|\chi_\omega\overline\varphi_\delta(\cdot;k\delta,z)\|_{L^2(0,k\delta;L^2(\Omega))}  + \frac{r}{2\|y_0\|} \|z\| \Big)^2,
\end{eqnarray*}
where $C\triangleq C(\Omega,\omega)$ is given by (\ref{WGS2.13}). The above yields that for each $z\in L^2(\Omega)$,
\begin{eqnarray*}
\langle y_0,\varphi(0;k\delta,z)\rangle
\geq - \big(2e^{\frac{C}{2}(1+\frac{1}{k\delta})} \|y_0\|^{2}r^{-1}\big) \|\chi_\omega\overline\varphi_\delta(\cdot;k\delta,z)\|_{L^2(0,k\delta;L^2(\Omega))} - \frac{r}{2}\|z\| .
\end{eqnarray*}
From this and  (\ref{fun-P}), one can easily check that
\begin{eqnarray}\label{op-control-1}
 J_\delta^{k\delta,y_0}(z)
 \geq \frac{r}{2} \|z\| - 2e^{C(1+\frac{1}{k\delta})} \|y_0\|^{4}r^{-2}\;\;\mbox{for each}\;\;z\in L^2(\Omega),
 \end{eqnarray}
 which leads to the coercivity of $J_{\delta}^{k\delta,y_0}$ over $L^2(\Omega)$.  Hence,   $J_{\delta}^{k\delta,y_0}$ has at least one   minimizer in   $L^2(\Omega)$.

 Next, we claim that $0$ is not a minimizer of $J_{\delta}^{k\delta,y_0}$. By contradiction, suppose that it were not true. Then  we would find from (\ref{fun-P}) that for all $z\in L^2(\Omega)$ and $\varepsilon>0$,
 \begin{eqnarray*}
  0\leq \frac{J_\delta^{k\delta,y_0}(\varepsilon z)-J_\delta^{k\delta,y_0}(0)}{\varepsilon}=\frac{\varepsilon}{2} \|\chi_\omega\overline\varphi_\delta(\cdot;k\delta,z)\|_{L^2(0,k\delta;L^2(\Omega))}^2 +\langle y_0,\varphi(0;k\delta,z) \rangle + r \|z\|.
 \end{eqnarray*}
 Sending $\varepsilon$ to $0$ in the above  leads to that
 \begin{eqnarray*}
   \langle e^{\Delta k\delta}y_0,z \rangle + r \|z\|
   =\langle y_0,\varphi(0;k\delta,z) \rangle + r \|z\|
   \geq 0
   \;\;\mbox{for all}\;\;   z\in L^2(\Omega).
 \end{eqnarray*}
 This yields that
 \begin{eqnarray*}
 \|e^{\Delta k\delta}y_0\|=\sup_{z\in L^2(\Omega)\setminus\{0\}} {\langle e^{\Delta k\delta}y_0,z \rangle}/{\|z\|} \leq r.
 \end{eqnarray*}
  Since $y_0\in L^2(\Omega)\setminus B_r(0)$, the above, along with (\ref{T*}), indicates that $k\delta\geq T^*_{y_0}$, which contradicts the assumption that $(\delta,k)\in \mathcal P_{T^*_{y_0}}$ (given by (\ref{P-T*})). Thus, $0$ is not a minimizer of $J_{\delta}^{k\delta,y_0}$.

 We now show the uniqueness of the minimizer of $J_{\delta}^{k\delta,y_0}$. To this end, we claim that the first term
 on the right hand side of  (\ref{fun-P}) is strictly convex.
 When this claim is proved, it follows from
 (\ref{fun-P}) that $J_{\delta}^{k\delta,y_0}$ is strictly convex
 over $L^2(\Omega)$. So its minimizer is unique.

   To show the above claim, we first observe from (\ref{adjoint-solution-P}) that
  \begin{eqnarray}\label{YUwanggeng4.9}
 \overline\varphi_\delta(t;k\delta,\lambda z_1+\mu z_2)=
  \lambda\overline\varphi_\delta(t;k\delta,z_1)
  + \mu\overline\varphi_\delta(t;k\delta,z_2)\;\;\mbox{for all}\;\; \lambda, \mu\in \mathbb{R}.
 \end{eqnarray}
 By this, we see  that the first term
 on the right hand side of  (\ref{fun-P}) is  convex.
 Next, we suppose, by contradiction, that this term
 were not strictly convex. Then, by the convexity of this term, there would be $\hat\lambda\in(0,1)$ and $z_1,\,z_2\in L^2(\Omega)$, with $z_1\neq z_2$, so that
 \begin{eqnarray*}
  & & \int_0^{k\delta} \|\chi_\omega \overline\varphi_\delta(t;k\delta,\hat\lambda z_1+(1-\hat\lambda)z_2)\|^2 \,\mathrm dt
  \nonumber\\
  &=& \hat\lambda \int_0^{k\delta} \| \chi_\omega \overline\varphi_\delta(t;k\delta,z_1)\|^2 \,\mathrm dt
  + (1-\hat\lambda)\int_0^{k\delta} \| \chi_\omega \overline\varphi_\delta(t;k\delta,z_2)\|^2 \,\mathrm dt,
 \end{eqnarray*}
 which, along with (\ref{YUwanggeng4.9}), yields that for each $t\in (0,k\delta)$,
 \begin{eqnarray*}
 \|\hat\lambda\chi_\omega \overline\varphi_\delta(t;k\delta,z_1)+(1-\hat\lambda)\chi_\omega \overline\varphi_\delta(t;k\delta,z_1)\|^2
 =\hat\lambda\|\chi_\omega\overline\varphi_\delta(t;k\delta,z_1)\|^2
 +(1-\hat\lambda)\|\chi_\omega\overline\varphi_\delta(t;k\delta,z_2)\|^2.
  \end{eqnarray*}
 From this  and the strict convexity of $\|\cdot\|^2$, we see  that for each $t\in (0,k\delta)$,
  \begin{eqnarray}\label{wang4.21}
  \chi_\omega \overline\varphi_\delta(t;k\delta,z_1)= \chi_\omega \overline\varphi_\delta(t;k\delta,z_2),\;\;\mbox{i.e.,}\;\;\chi_\omega \overline\varphi_\delta(t;k\delta,z_1-z_2)=0.
  \end{eqnarray}
   Notice that  $k\geq 2$. Thus, we can apply Lemma \ref{Theorem-ob-int-0} (where $S=(k-1)\delta$, $T=k\delta$ and $z=z_1-z_2$), and use  (\ref{wang4.21}) to  obtain that $\varphi(0; k\delta,z_1-z_2)=0$. This, together with the backward uniqueness of the heat equation, yields that  $z_1=z_2$ in $L^2(\Omega)$, which leads to a contradiction. Hence, the first term
 on the right hand side of  (\ref{fun-P}) is strictly convex.

   In summary, conclude that
$J_{\delta}^{k\delta,y_0}$ has a unique minimizer $z_\delta^*\neq 0$.

Finally, we prove that the minimizer $z_\delta^*$ satisfies  (\ref{nonzero-varphi-minimizer}). By contradiction, suppose  that
 it were not true. Then we would have that
\begin{eqnarray}\label{unique-11-1}
 \chi_\omega\overline\varphi_\delta(t_0;k\delta,z_\delta^*)= 0
 \;\;\mbox{for some}\;\;
 t_0\in \big(0,(k-1)\delta\big].
\end{eqnarray}
 Since $\overline\varphi_\delta(\cdot;k\delta,z_\delta^*)$ is a piece-wise constant function from $(0,k\delta]$ to $L^2(\Omega)$ (see (\ref{adjoint-solution-P})), it follows from  (\ref{unique-11-1}) that
 \begin{eqnarray}\label{wang4.24}
  \chi_\omega\overline\varphi_\delta(\cdot;k\delta,z_\delta^*)= 0 \;\;\mbox{over}\;\; \big((i_0-1)\delta,i_0\delta\big]
  \;\;\mbox{for some}\;\;  i_0\in\{1,\cdots,k-1\}.
 \end{eqnarray}
 By (\ref{wang4.24}), we can apply Lemma \ref{Theorem-ob-int-0} (where $T=(k+1-i_0)\delta$, $S=\delta$ and $z=e^{\Delta(i_0-1)\delta)}z_\delta^*$) to get that
 \begin{eqnarray*}
 0=\varphi(0;(k+1-i_0)\delta,e^{\Delta(i_0-1)\delta)}z_\delta^*)= \varphi((i_0-1)\delta;k\delta,z_\delta^*).
 \end{eqnarray*}
  This, along with the backward uniqueness  for the heat equation, yields that
  $z_\delta^*=0$ in $ L^2(\Omega)$,
   which leads to a contradiction. Therefore, (\ref{nonzero-varphi-minimizer}) holds.  This ends the proof of the conclusion (i) of Theorem~\ref{Lemma-NP-bangbang}.

(ii) Let $z_\delta^*$ be the minimizer of $J_{\delta}^{k\delta,y_0}$.
 Let  $v_\delta^*$ be given by (\ref{PMP-Ndelta-0}).
It suffices to show that  $v_\delta^*$ is the unique optimal control to $(NP)_\delta^{k\delta,y_0}$ and satisfies (\ref{op-control-final-state}).
{\it From now on and throughout the proof of Theorem~\ref{Lemma-NP-bangbang},
 we simply write $\varphi(\cdot)$ and $\overline\varphi_\delta(\cdot)$ for $\varphi(\cdot;k\delta,z_\delta^*)$ and $\overline\varphi_\delta(\cdot;k\delta,z_\delta^*)$; and simply write $L^2(0,k\delta)$ for $L^2(0,k\delta;L^2(\Omega))$.}

We first show that $v_\delta^*$ is an admissible control to $(NP)_\delta^{k\delta,y_0}$ and satisfies (\ref{op-control-final-state}).
By  (\ref{fun-P}), one can easily check that the   Euler-Lagrange equation associated with the minimizer $z_\delta^*$ is as follows:
\begin{eqnarray}\label{0221-NP-delta-ii-1}
 \langle \chi_\omega\overline\varphi_\delta(\cdot), \overline\varphi_\delta(\cdot;k\delta,z) \rangle_{L^2(0,k\delta)}
 + \langle y_0,e^{\Delta k\delta}z \rangle
 + \langle r \frac{z_\delta^*}{\|z_\delta^*\|},z \rangle =0,
 ~\forall\,
 z\in L^2(\Omega).
\end{eqnarray}
We claim that for each $z\in L^2(\Omega)$,
\begin{eqnarray}\label{Wang4.26}
\langle \chi_\omega\overline\varphi_\delta(\cdot), \chi_\omega\overline\varphi_\delta(\cdot;k\delta,z)\rangle _{L^2(0,k\delta)}
  =\langle \chi_\omega\overline\varphi_\delta(\cdot), \chi_\omega\varphi(\cdot;k\delta,z)\rangle_{L^2(0,k\delta)}.
\end{eqnarray}
To this end, we arbitrarily fix   $z\in L^2(\Omega)$. Let $f(\cdot)$ and $g(\cdot)$ be the zero extensions of $\chi_\omega\varphi(\cdot)$ and $\chi_\omega\varphi(\cdot;k\delta,z)$
 over $\mathbb{R}^+$.
  Then  by (\ref{adjoint-solution-P-1}) and
(\ref{adjoint-solution-P}), we see that
\begin{eqnarray*}
\bar f_{\delta}(\cdot)=\chi_\omega\overline\varphi_{\delta}(\cdot)\;\;\mbox{and}\;\;
\bar g_{\delta}(\cdot)=\chi_\omega\overline\varphi_{\delta}(\cdot; k\delta,z)\;\;\mbox{over}\;\;
\mathbb{R}^+,
\end{eqnarray*}
where $\overline\varphi_{\delta}(\cdot)$
and $\overline\varphi_{\delta}(\cdot; k\delta,z)$ are
 treated as their zero extensions over $\mathbb{R}^+$. Then by
 Lemma \ref{lemma-control-op-dual},
 we have that
 \begin{eqnarray*}
 \langle \bar f_\delta, \bar g_\delta \rangle_{L^2(\mathbb R^+;L^2(\Omega))}=\langle \bar f_\delta,g \rangle_{L^2(\mathbb R^+;L^2(\Omega))},
 \end{eqnarray*}
 which leads to  (\ref{Wang4.26}). Now, from (\ref{0221-NP-delta-ii-1}) and (\ref{Wang4.26}), it follows that for each
$z\in L^2(\Omega)$,
\begin{eqnarray*}
 \langle \chi_\omega\overline\varphi_\delta(\cdot), \chi_\omega\overline\varphi_\delta(\cdot;k\delta,z) \rangle_{L^2(0,k\delta)}
  =\langle v_\delta^*(\cdot),\chi_\omega \varphi(\cdot;k\delta,z) \rangle_{L^2(0,k\delta)}
  =\langle y(k\delta;0,v_\delta^*) , z \rangle.
\end{eqnarray*}
 This, along with (\ref{0221-NP-delta-ii-1}), yields that
\begin{eqnarray}\label{0221-NP-delta-ii-2}
 y(k\delta;y_0,v_\delta^*) + r z_\delta^*/\|z_\delta^*\|=0.
\end{eqnarray}
From (\ref{0221-NP-delta-ii-2}),  $v_\delta^*$ is an admissible control to $(NP)_\delta^{k\delta,y_0}$, and satisfies (\ref{op-control-final-state}).

We next prove that $v_\delta^*$ is an optimal control to $(NP)_\delta^{k\delta,y_0}$. To this end, we arbitrarily fix an admissible control $v_\delta$  to $(NP)_\delta^{k\delta,y_0}$. Then we have that
$\|y(k\delta;y_0,v_\delta)\| \leq r$.
This, together with (\ref{0221-NP-delta-ii-2}), implies that
\begin{eqnarray}\label{wang4.28}
& & \langle  y(k\delta;0,v_\delta^*), z_\delta^* \rangle
= \langle  y(k\delta;y_0,v_\delta^*), z_\delta^* \rangle -\langle e^{\Delta k\delta}y_0,z_\delta^* \rangle
 = - r \|z_\delta^*\| -\langle e^{\Delta k\delta}y_0,z_\delta^* \rangle
 \nonumber\\
 &\leq& \langle  y(k\delta;y_0,v_\delta), z_\delta^* \rangle -\langle e^{\Delta k\delta}y_0,z_\delta^* \rangle
 = \langle  y(k\delta;0,v_\delta), z_\delta^* \rangle.
\end{eqnarray}
Meanwhile, by Lemma \ref{lemma-control-op-dual} (where $(f,g)$ are taken as the zero extensions of $(v^*_\delta,\chi_\omega\varphi)$ and $(v_\delta,\chi_\omega\varphi)$
over $\mathbb{R}^+$ respectively), and by (\ref{adjoint-solution-P-1}) and
(\ref{adjoint-solution-P}), one can easily verify that
\begin{eqnarray}\label{wang4.30}
\langle v_\delta^*, \chi_\omega\overline\varphi_\delta\rangle _{L^2(0,k\delta)}
=\langle v_\delta^*, \chi_\omega \varphi \rangle_{L^2(0,k\delta)}\;\;\mbox{and}\;\;
\langle v_\delta, \chi_\omega\overline\varphi_\delta\rangle_{L^2(0,k\delta)}
=\langle v_\delta, \chi_\omega\varphi\rangle_{L^2(0,k\delta)}.
\end{eqnarray}
Since $v_\delta$ and $v^*_\delta$ are piece-wise constant functions (see (\ref{L2-delta}) and
(\ref{heat-Linfty-delta})), it follows from   (\ref{PMP-Ndelta-0}),  (\ref{wang4.30}) and (\ref{wang4.28}) that
\begin{eqnarray*}
 & & \|v_\delta^*\|_{L^2(0,k\delta)}\|\chi_\omega\overline\varphi_\delta
 \|_{L^2(0,k\delta)}
 =\langle v_\delta^*, \chi_\omega\overline\varphi_\delta\rangle_{L^2(0,k\delta)}
 =\langle v_\delta^*,\chi_\omega\varphi\rangle_{L^2(0,k\delta)}=
    \langle  y(k\delta;0,v_\delta^*), z_\delta^* \rangle\nonumber\\
 &\leq& \langle  y(k\delta;0,v_\delta), z_\delta^* \rangle
 = \langle v_\delta, \chi_\omega\varphi\rangle_{L^2(0,k\delta)}
 =\langle v_\delta, \chi_\omega\overline\varphi_\delta\rangle_{L^2(0,k\delta)}
  \leq \|v_\delta\|_{L^2(0,k\delta)}  \|\chi_\omega\overline\varphi_\delta\|_{L^2(0,k\delta)}.
\end{eqnarray*}
This, along with  (\ref{nonzero-varphi-minimizer}), yields that
$\|v_\delta^*\|_{L^2(0,k\delta)}
 \leq \|v_\delta\|_{L^2(0,k\delta)}$.
Because $v_\delta$ is an arbitrarily fixed admissible control to $(NP)_\delta^{k\delta,y_0}$, we see that $v_\delta^*$ is an optimal control to $(NP)_\delta^{k\delta,y_0}$.

Finally, we prove the uniqueness of the optimal control to $(NP)_\delta^{k\delta,y_0}$. By contradiction, we suppose that
$(NP)_\delta^{k\delta,y_0}$ had two different optimal controls $v^*_{\delta,1}$
and $v^*_{\delta,2}$. Then one could easily check that $(v^*_{\delta,1}+v^*_{\delta,2})/2$
is still an optimal control. Since $v^*_{\delta,1}\neq v^*_{\delta,2}$, we can use
the parallelogram law to get that
\begin{eqnarray*}
\|(v^*_{\delta,1}+v^*_{\delta,2})/2\|_{L^2(0, k\delta)}<N_\delta(k\delta,y_0),
 \end{eqnarray*}
 which contradicts the optimality of
$N_\delta(k\delta,y_0)$ to $(NP)_\delta^{k\delta,y_0}$.
 This proves the conclusion (ii) of Theorem~\ref{Lemma-NP-bangbang}.

(iii) Taking $z=z_\delta^*$ in (\ref{0221-NP-delta-ii-1}) leads to that
\begin{eqnarray*}
\langle y_0,\varphi(0)\rangle
 + r \|z_\delta^*\|
 = -\|\chi_\omega\overline\varphi_\delta\|^2_{L^2(0,k\delta)}.
 \end{eqnarray*}
Since $z_\delta^*$ is the minimizer of $J_\delta^{k\delta,y_0}$, the above equality, along with (\ref{fun-P}), indicates that
\begin{eqnarray}\label{0221-NP-delta-ii-3}
 V_\delta(k\delta,y_0) = J_\delta^{k\delta,y_0}(z_\delta^*)
 = - \frac{1}{2}\| \chi_\omega\overline\varphi_\delta\|^2_{L^2(0,k\delta)}.
\end{eqnarray}
Meanwhile, from (ii) of Theorem~\ref{Lemma-NP-bangbang}, we see that
\begin{eqnarray*}
N_\delta(k\delta,y_0)=\|v_\delta^*\|_{L^2(0,k\delta)}
 =\|\chi_\omega
 \overline\varphi_\delta\|_{L^2(0,k\delta)}.
 \end{eqnarray*}
This, along with (\ref{0221-NP-delta-ii-3}) leads to the conclusion (iii) of Theorem~\ref{Lemma-NP-bangbang}.

 In summary, we end the proof of Theorem~\ref{Lemma-NP-bangbang}.
\end{proof}

\section{Several auxiliary estimates}

This section presents several estimates, as well as properties,  on
minimizers (of $J_{\delta}^{k\delta,y_0}$ and $J^{T,y_0}$), the minimal norm functions and the minimal time functions. These estimate  will play important roles in the proofs of the main theorems.

\subsection{Some estimates on minimizers}

The following theorem concerns with the  $H_0^1(\Omega)$-estimates on  the minimizers of the functionals $J_{\delta}^{k\delta,y_0}$ and $J^{T,y_0}$.

\begin{theorem} \label{Lemma-minizer-P-L2-uni-bdd}
Let $y_0\in L^2(\Omega)\setminus B_r(0)$. Let $(\delta,k)\in\mathcal P_{T^*_{y_0}}$ (given by (\ref{P-T*})) and
$0<T<T^*_{y_0}$ (given by (\ref{T*})).
Write $z_\delta^*$ and $z^*$ for the minimizers of $J_{\delta}^{k\delta,y_0}$ and $J^{T,y_0}$, respectively. Then the following conclusions are true:

 (i) There is a positive constant $C_1 \triangleq C_1(\Omega,\omega)$ so that
\begin{eqnarray}
   \|z_\delta^*\| &\leq&    e^{C_1(1+\frac{1}{k\delta})}\|y_0\|^{4} r^{-3};
   \label{z-L2-uni-bdd}
\\
   \|\partial_t\varphi(\cdot,k\delta;z_\delta^*)\|_{L^2(0,k\delta;L^2(\Omega))}
   &\leq&  \|z_\delta^*\|_{H_0^1(\Omega)}
    \leq     e^{C_1(1+k\delta+\frac{1}{k\delta})}  \|y_0\|^{6} r^{-5}.
    \label{z-beta-uni-bdd}
\end{eqnarray}

 (ii) There is a positive constant $C_2 \triangleq C_2(\Omega,\omega)$ so that
\begin{eqnarray}\label{z-L2-bdd}
   \|z^*\| &\leq&    e^{C_2(1+\frac{1}{T})}\|y_0\|^{4} r^{-3};
\\ \label{z-bdd}
   \|\partial_t\varphi(\cdot,k\delta;z^*)\|_{L^2(0,T;L^2(\Omega))}
   &\leq&  \|z^*\|_{H_0^1(\Omega)}
    \leq     e^{C_2(1+T+\frac{1}{T})}  \|y_0\|^{6} r^{-5}.
\end{eqnarray}

\end{theorem}

\begin{proof}
Throughout the proof, $C(\Omega,\omega)$ stands for a positive constant depending only on $\Omega$ and $\omega$. It may vary in different contexts.

(i) We begin with proving (\ref{z-L2-uni-bdd}). From (\ref{op-control-1}), we find  that
\begin{eqnarray*}
  \frac{r}{2} \|z_\delta^*\| - 2e^{C(1+\frac{1}{k\delta})} \|y_0\|^{4}r^{-2} \leq J_{\delta}^{k\delta,y_0}(z_\delta^*),\;\;\mbox{for some}\;\;C=C(\Omega,\omega).
\end{eqnarray*}
Since $z_\delta^*$ is the minimizer of $J_{\delta}^{k\delta,y_0}$, the above inequality, along with   (\ref{fun-P}), implies that
\begin{eqnarray*}
\frac{r}{2} \|z_\delta^*\| - 2e^{C(1+\frac{1}{k\delta})} \|y_0\|^{4}r^{-2}
\leq J_{\delta}^{k\delta,y_0}(0)=0,
\end{eqnarray*}
which leads to (\ref{z-L2-uni-bdd}).

To  show (\ref{z-beta-uni-bdd}), we need two estimates related to the optimal control $u_\delta^*$ of $(NP)_\delta^{k\delta,y_0}$.
We first claim  that
\begin{eqnarray}\label{control-L2-uni-bdd}
  \|u_\delta^*\|_{L^2(0,k\delta;L^2(\Omega))}
   \leq e^{C(1+\frac{1}{k\delta})} \|y_0\|^{2} r^{-1}\;\;\mbox{for some}\;\;C\triangleq C(\Omega,\omega).
 \end{eqnarray}
  Indeed, since $(\delta,k)\in\mathcal P_{T^*_{y_0}}$, we have that $k\geq 2$. Thus, by  (ii)  of Theorem \ref{Theorem-eq-controllable-observable},  Equation (\ref{heat-2}) has the  $L^2$-approximate null controllability with a cost. From this,
  Definition~\ref{Def-ap-controllable-P} (see (\ref{P-p-appro-controllable})), and (iii) of Theorem \ref{Theorem-eq-controllable-observable} (see (\ref{WGS2.13})), we find that for $\varepsilon_0 = (r/\|y_0\|)^2$, there is $u_\delta\in L_\delta^2(0,k\delta;L^2(\Omega))$ so that
 \begin{eqnarray}\label{0113-sec3-1}
  \frac{\varepsilon_0}{e^{C(1+\frac{1}{k\delta})}}  \|u_\delta\|_{L^2(0,k\delta;L^2(\Omega))}^2
  + \frac{1}{\varepsilon_0} \|y(k\delta;y_0,u_\delta)\|^2
  \leq \|y_0\|^2\;\;\mbox{for some}\;\; C\triangleq C(\Omega,\omega).
 \end{eqnarray}
   Since $\varepsilon_0 = (r/\|y_0\|)^2$, it follows from (\ref{0113-sec3-1}) that $u_\delta$ is an admissible control to $(NP)_\delta^{k\delta,y_0}$. Then by the optimality of $u^*_\delta$ and  $N_\delta(k\delta,y_0)$,
    and by (\ref{0113-sec3-1}), we find that
 \begin{eqnarray*}
\|u_\delta^*\|_{L^2(0,k\delta;L^2(\Omega))}=N_\delta(k\delta,y_0)\leq \|u_\delta\|_{L^2(0,k\delta;L^2(\Omega))}
  \leq e^{\frac{C}{2}(1+\frac{1}{k\delta})} \|y_0\|^{2} r^{-1},
 \end{eqnarray*}
which leads to (\ref{control-L2-uni-bdd}).

Next, we claim   that
\begin{eqnarray}\label{0113-sec3-2}
\|y(k\delta;y_0,u_\delta^*)\|_{H_0^1(\Omega)}
\leq
e^{C(1+k\delta+\frac{1}{k\delta})}\|y_0\|^{2} r^{-1}\;\;\mbox{for some}\;\;C \triangleq C(\Omega,\omega).
\end{eqnarray}
For this purpose, we consider the following equation:
\begin{eqnarray}\label{0113-sec3-3}
\left\{\begin{array}{lll}
        \partial_t y - \Delta y = f  &\mbox{in}  &\Omega\times\mathbb R^+,\\
        y=0 &\mbox{on}  &\partial\Omega\times\mathbb R^+,\\
        y(0)=z &\mbox{in}  &\Omega,
       \end{array}
\right.
\end{eqnarray}
where $z\in C_0^\infty(\Omega)$ and $f\in C_0^\infty(\Omega\times \mathbb R^+)$.
Multiplying $y$ on both sides of Equation (\ref{0113-sec3-3}) and then integrating it over $\Omega$, after some  computations, we obtain that
for each $S>0$,
\begin{eqnarray}\label{0113-sec3-4}
  \int_0^S\int_\Omega  |\nabla y(x,t)|^2 \,\mathrm dx \mathrm dt
  \leq e^S\int_0^S\int_\Omega  |f(x,t)|^2 \,\mathrm dx \mathrm dt + e^S\int_\Omega |z(x)|^2 \,\mathrm dx.
\end{eqnarray}
Meanwhile,  multiplying $-t\Delta y(t)$ on both sides of Equation  (\ref{0113-sec3-3}) and then integrating it over $\Omega$, after some  computations,
we obtain that for each $S>0$,
 \begin{eqnarray}\label{wang5.10}
 \int_\Omega S |\nabla y(x,S)|^2 \,\mathrm dx
 \leq  \int_0^S\int_\Omega t|f(x,t)|^2 \,\mathrm dx \mathrm dt + \int_0^S\int_\Omega |\nabla y(x,t)|^2 \,\mathrm dx \mathrm dt.
\end{eqnarray}
From (\ref{wang5.10})  and (\ref{0113-sec3-4}), we deduce that for each $S>0$,
$z\in C_0^\infty(\Omega)$ and  $f\in C_0^\infty(\Omega\times \mathbb R^+)$,
\begin{eqnarray*}
\int_\Omega  |\nabla y(x,S)|^2 \,\mathrm dx
 \leq e^{1+S+1/S} \Big[ \int_0^S\int_\Omega |f(x,t)|^2 \,\mathrm dx\mathrm dt
      + \int_\Omega |z(x)|^2 \,\mathrm dx \Big].
\end{eqnarray*}
Then by a standard density argument, we can easily derive from the above inequality that
\begin{eqnarray*}
\|y(k\delta;y_0,u_\delta^*)\|_{H_0^1(\Omega)}
\leq e^{1+k\delta+\frac{1}{k\delta}} \big( \|u_\delta^*\|_{L^2(0,k\delta;L^2(\Omega))} + \|y_0\|  \big).
\end{eqnarray*}
Since $\|y_0\|>r$, the above, along with (\ref{control-L2-uni-bdd}), leads to (\ref{0113-sec3-2}).

We now show the second inequality in (\ref{z-beta-uni-bdd}).  From   (\ref{op-control-final-state}), we see that
\begin{eqnarray*}
\|  z_\delta^*\|_{H_0^1(\Omega)}= \frac{\|z_\delta^*\|}{r}
  \|y(k\delta;y_0,u_\delta^*)\|_{H_0^1(\Omega)},
  \end{eqnarray*}
which, together with (\ref{z-L2-uni-bdd}) and (\ref{0113-sec3-2}),  leads to the second inequality in (\ref{z-beta-uni-bdd}).

Then, we show the first inequality in (\ref{z-beta-uni-bdd}).
Simply write $\varphi(\cdot)$ for $\varphi(\cdot;k\delta,z^*_\delta)$.
Multiplying by $\Delta\varphi$ on both sides of the  equation
satisfied by $\varphi(\cdot;k\delta,z^*_\delta)$, and then integrating it over $\Omega$, after some computations, we obtain that
\begin{eqnarray*}
 \int_\Omega |\nabla\varphi(x,0)|^2 \,\mathrm dx  + \int_0^{k\delta} \int_\Omega |\Delta\varphi(x,t)|^2 \,\mathrm dx\mathrm dt
 = \int_\Omega |\nabla\varphi(x,k\delta)|^2 \,\mathrm dx.
\end{eqnarray*}
From this, it follows that
\begin{eqnarray*}
 \int_0^{k\delta} \int_\Omega |\partial_t\varphi(x,t)|^2 \,\mathrm dx\mathrm dt=\int_0^{k\delta} \int_\Omega |\Delta\varphi(x,t)|^2 \,\mathrm dx\mathrm dt
 \leq \int_\Omega |\nabla z_\delta^*(x)|^2 \,\mathrm dx,
\end{eqnarray*}
which leads to the first inequality in (\ref{z-beta-uni-bdd}). This ends the proof of the conclusion (i).

(ii)
 Arbitrarily fix $k_0\in \mathbb N^+$ so that $k_0\geq \max\{2,2/T\}$. For each integer $k\geq k_0$, let $n_k$ be the integer so that
\begin{eqnarray}\label{0824-pf2-1}
 kT-1<n_k \leq kT.
\end{eqnarray}
We first claim
\begin{eqnarray}\label{0824-pf2-2}
 \liminf_{k\rightarrow\infty} V_{1/k}(n_k/k,y_0)\leq V(T,y_0)\;\;\mbox{for all}\;\;
 k\geq k_0.
\end{eqnarray}
In fact, for each $k\geq k_0$,
 $(NP)_{1/k}^{n_k/k,y_0}$ has a unique optimal control  $v^*_{k}$ (see  (ii) of Theorem \ref{Lemma-NP-bangbang}).
Then, by (\ref{0824-pf2-1}),  one can easily check that the zero extension of $v^*_k$ over $(0,T)$
  is an admissible control to $(NP)^{T,y_0}$. From this and the optimality of $N(T,y_0)$, one can easily check that
  \begin{eqnarray}\label{YUwang5.13}
  N(T,y_0) \leq N_{1/k}(n_k/k,y_0)\;\;\mbox{for all}\;\;k\geq k_0.
  \end{eqnarray}
  Since $0<T<T^*_{y_0}$ and because $(1/k,n_k)\in \mathcal P_{T^*_{y_0}}$ for all $k\geq k_0$ (which follows from (\ref{0824-pf2-1}) and (\ref{P-T*})), we can apply
    (iii) of Theorem \ref{Lemma-Norm-fun} and (iii) of Theorem \ref{Lemma-NP-bangbang} (with $(\delta,k)$=$(1/k,n_k)$), and use (\ref{YUwang5.13}) to obtain (\ref{0824-pf2-2}).

For each integer $k\geq k_0$, write $z_{1/k}^*$ for the minimizer of $(JP)_{1/k}^{n_k/k,y_0}$. The key is to show that on a subsequence of $\{z^*_{1/k}\}_{k\geq k_0}$, still denoted in the same manner,
\begin{eqnarray}\label{haoren5.14}
z^*_{1/k} \rightarrow z^*
 \;\;\mbox{weakly in}\;\; H_0^1(\Omega); \;\;\mbox{strongly in}\;\; L^2(\Omega),
 \;\;\mbox{as}\;\;
 k\rightarrow\infty.
\end{eqnarray}
(Here, $z^*$ is the minimizer of  $J^{T,y_0}$.) To this end,
we notice that  $(1/k,n_k)\in \mathcal P_{T^*_{y_0}}$ for all $k\geq k_0$ (which follows from (\ref{0824-pf2-1}) and (\ref{P-T*})). Thus, we can use
  the second inequality in (\ref{z-beta-uni-bdd}) (where $\delta=1/k$; $k=n_k$)
to find that  $\{z^*_{1/k}\}_{k\geq k_0}$ is bounded in $H_0^1(\Omega)$. So there exists a subsequence of $\{z^*_{1/k}\}_{k\geq k_0}$, still denoted in the same manner, and some $\hat z\in H_0^1(\Omega)$ so that
\begin{eqnarray}\label{0824-pf2-4}
 z^*_{1/k} \rightarrow \hat z
 \;\;\mbox{weakly in}\;\; H_0^1(\Omega); \;\;\mbox{strongly in}\;\; L^2(\Omega),
 \;\;\mbox{as}\;\;
 k\rightarrow\infty.
\end{eqnarray}
From the above, we see that in order to show (\ref{haoren5.14}), it suffices to prove that $z^*=\hat z$. For this purpose, we first claim that for each $k\geq k_0$,
 \begin{eqnarray}\label{wang5.14}
 &&\|\varphi(0;T,\hat z)-\varphi(0;n_k/k,z^*_{1/k})\|
  \nonumber\\
 &\leq& \sup_{0\leq s\leq \hat t\leq s+\frac{1}{k}\leq T} \| \varphi(\hat t;T,\hat z)  -  \varphi(s;T,\hat z)\|  +  \| \hat z- z^*_{1/k}\|;
 \end{eqnarray}
 \begin{eqnarray}\label{wang5.15}
 & & \| \varphi(t;T,\hat z)  -  \overline\varphi_{1/k}(t;n_k/k,z^*_{1/k})\|
     \nonumber\\
   &\leq& 2 \sup_{0\leq s\leq \hat t\leq s+\frac{1}{k}\leq T} \| \varphi(\hat t;T,\hat z)  -  \varphi(s;T,\hat z)\|  +  \| \hat z- z^*_{1/k}\|,\;\forall\;t\in(0,n_k/k).
 \end{eqnarray}
To show (\ref{wang5.14}), we arbitrarily fix $k\geq k_0$. By  (\ref{0824-pf2-1}), we see that  $0\leq T-n_k/k\leq 1/k$. This, along with the time-invariance of  Equation (\ref{adjoint-1}), yields
 \begin{eqnarray}\label{wang5.16}
 \|\varphi(0;T,\hat z)-\varphi(0;n_k/k,\hat z)\|&=&
 \|\varphi(0;T,\hat z)-\varphi(T-n_k/k; T,\hat z)\|\nonumber\\
 &\leq& \sup_{0\leq s\leq t\leq s+\frac{1}{k}\leq T} \| \varphi(t;T,\hat z)  -  \varphi(s;T,\hat z)\|
  \end{eqnarray}
  Meanwhile, since $\{e^{t\Delta} : t\geq 0\}$ is contractive, we have that
  \begin{eqnarray}\label{wang5.17}
  \|\varphi(0;n_k/k,\hat z)-\varphi(0;n_k/k,z^*_{1/k})\|\leq \| \hat z- z^*_{1/k}\|.
  \end{eqnarray}
  Using the triangle inequality, by (\ref{wang5.16}) and (\ref{wang5.17}), we obtain
  (\ref{wang5.14}).

  To show (\ref{wang5.15}), we arbitrarily fix $k\geq k_0$ and  $t\in(0,n_k/k)$.
 Three facts are given in order. Fact one: Since $0\leq T-n_k/k\leq 1/k$, we can use the time-invariance of  Equation (\ref{adjoint-1}) to get
that
  \begin{eqnarray}\label{YUyongyu5.18}
  \|\varphi(t;T,\hat z)-\varphi(t;n_k/k,\hat z)\|
  &=& \|\varphi(t;T,\hat z)-\varphi(T-n_k/k+t;T,\hat z)\|\nonumber\\
  &\leq& \sup_{0\leq s\leq \hat t\leq s+\frac{1}{k}\leq T} \| \varphi(\hat t;T,\hat z)-  \varphi(s;T,\hat z)\|.
   \end{eqnarray}
   Fact two: Since $0\leq T-n_k/k\leq 1/k$, by (\ref{adjoint-solution-P}) and
    the time-invariance of  Equation (\ref{adjoint-1}),
    we can easily check that
  \begin{eqnarray}\label{YUyongyu5.19}
  & & \|\varphi(t;n_k/k,\hat z)-\overline\varphi_{1/k}(t;n_k/k,\hat z)\|
  \\
    &=& \Big\| \sum_{i=1}^{n_k} \chi_{((i-1)/k,i/k]}(t) k
  \int_{(i-1)/k}^{i/k} \big[\varphi(t;n_k/k,\hat z)-\varphi(s;n_k/k,\hat z) \big] \,\mathrm ds\Big\|
  \nonumber\\
  &\leq &
  \sup_{0\leq s\leq \hat t\leq s+\frac{1}{k}\leq n_k/k} \| \varphi(\hat t;n_k/k,\hat z)-  \varphi(s;n_k/k,\hat z)\|\nonumber\\
    &\leq& \sup_{0\leq s\leq \hat t\leq s+\frac{1}{k}\leq T} \| \varphi(\hat t;T,\hat z)-  \varphi(s;T,\hat z)\|.
  \nonumber
     \end{eqnarray}
   Fact three: Since $\{e^{t\Delta} : t\geq 0\}$ is contractive, by (\ref{adjoint-solution-P}), we see that
  \begin{eqnarray}\label{YUyongyu5.20}
  \|\overline\varphi_{1/k}(t;n_k/k,\hat z)-\overline\varphi_{1/k}(t;n_k/k,z^*_{1/k})\|
  = \|\overline\varphi_{1/k}(t;n_k/k,\hat z-z^*_{1/k})\|
  \leq
  \| \hat z- z^*_{1/k}\|.
      \end{eqnarray}
  The above three facts (\ref{YUyongyu5.18}), (\ref{YUyongyu5.19}) and (\ref{YUyongyu5.20}), together with the triangle inequality, leads to (\ref{wang5.15}).

Two observations are given in order: First, since $\varphi(\cdot; T,\hat z)$ is uniformly continuous on $[0,T]$, we see that two
supremums in (\ref{wang5.14}) and (\ref{wang5.15}) tend to zero as $k\rightarrow\infty$. Second, it follows by (\ref{0824-pf2-1}) that
 $\lim_{k\rightarrow\infty} n_k/k=T$. From these two observations,    (\ref{0824-pf2-4}), (\ref{wang5.14}) and  (\ref{wang5.15}), one can easily check  that
\begin{eqnarray*}
\langle y_0,\varphi(0;T,\hat z) \rangle
&=& \lim_{k\rightarrow\infty} \langle y_0,\varphi(0;n_k/k,z^*_{1/k}) \rangle;
\nonumber\\
\int_0^T \|\chi_\omega\varphi(t;T,\hat z)\|^2 \,\mathrm dt
&=& \lim_{k\rightarrow\infty}  \int_0^{n_k/k} \|\chi_\omega\overline\varphi_{1/k}(t;n_k/k,z^*_{1/k})\|^2 \,\mathrm dt.
\end{eqnarray*}
These, together with (\ref{fun-0}), (\ref{fun-P}) and (\ref{0824-pf2-4}), indicate that
\begin{eqnarray*}
 J^{T,y_0}(\hat z) = \lim_{k\rightarrow\infty} J_{1/k}^{n_k/k,y_0}(z^*_{1/k})=\lim_{k\rightarrow\infty} V_{1/k}(n_k/k,y_0).
\end{eqnarray*}
This, along with (\ref{0824-pf2-2}) and (\ref{fun-0}), yields that
\begin{eqnarray*}
J^{T,y_0}(\hat z)=V(T,y_0)=\inf_{z\in L^2(\Omega)} J^{T,y_0}(z).
\end{eqnarray*}
Hence,  $\hat z$ is a minimizer of $J^{T,y_0}$. Then, by the uniqueness of the minimizer, we see that $\hat z=z^*$. Hence, (\ref{haoren5.14}) is true.

Finally,  since $0<T<T^*_{y_0}$ and because $(1/k,n_k)\in \mathcal P_{T^*_{y_0}}$ for all $k\geq k_0$ (which follows from (\ref{0824-pf2-1}) and (\ref{P-T*})),
the conclusion (i)  in Theorem~\ref{Lemma-minizer-P-L2-uni-bdd} is available for $(\delta,k)$=$(1/k,n_k)$. Thus,
by   (\ref{z-L2-uni-bdd}),
 the second inequality in  (\ref{z-beta-uni-bdd}) (with $(\delta,k)$=$(1/k,n_k)$)
 and  (\ref{haoren5.14}),
   using the fact that
 $n_k/k\rightarrow T$ (see (\ref{0824-pf2-1})), we can easily obtain (\ref{z-L2-bdd}) and the second inequality in (\ref{z-bdd}). Besides, by the same way as that used to prove the first inequality in (\ref{z-beta-uni-bdd}), we get the first inequality in (\ref{z-bdd}).

In summary, we end the proof of Theorem~\ref{Lemma-minizer-P-L2-uni-bdd}.
\end{proof}

\subsection{Some estimates related to minimal norm functions}

Several inequalities related to the minimal norm functions $T\rightarrow N(T,y_0)$ and $k\delta\rightarrow
N_\delta(k\delta,y_0)$ will be presented in the following two theorems.

\begin{theorem}\label{Proposition-NP-Lip-T}
 Let $y_0\in L^2(\Omega)\setminus B_r(0)$. Then there is $C_3\triangleq C_3(\Omega,\omega)>0$ so that for each pair $(T_1,T_2)$, with $0<T_1\leq T_2<T^*_{y_0}$
 (given by (\ref{T*})),
 \begin{eqnarray}\label{NP-Lip-T}
  \lambda_1^{3/2} r (T_2-T_1)\leq N(T_1,y_0)-N(T_2,y_0)
  \leq e^{C_3(1+\frac{1}{T_1})} \|y_0\|  (T_2-T_1).
 \end{eqnarray}
 \end{theorem}
\begin{proof}
Arbitrarily fix a  pair $(T_1,T_2)$,  with $0<T_1<T_2<T^*_{y_0}$.
The proof is organized by  the following two steps:

\textit{Step 1. To show the first inequality in (\ref{NP-Lip-T})}

 By (i) of Theorem~\ref{Theorem-eq-cont-discrete}, we have that
\begin{eqnarray}\label{0223-low-estimate-NT}
  M_1 \triangleq N(T_1,y_0)  >     N(T_2,y_0)  \triangleq  M_2.
 \end{eqnarray}
Then by  (iii) in Theorem \ref{Theorem-eq-cont-discrete}, we see that
 \begin{eqnarray}\label{NP-Lip-1}
    0<T(M_1,y_0)=T_1    <   T_2=  T(M_2,y_0)<T^*_{y_0}.
 \end{eqnarray}
Let  $u^*_1$ be  an optimal control to $(TP)^{M_1,y_0}$. Then we find that
\begin{eqnarray}\label{NP-Lip-2}
  \|y(T(M_1,y_0);y_0,u^*_1)\|\leq r
  \;\;\mbox{and}\;\;
  \|u^*_1\|_{L^2(\mathbb R^+;L^2(\Omega))}\leq M_1.
\end{eqnarray}
It follows from the first inequality in (\ref{NP-Lip-2})  that
\begin{eqnarray*}
  \big\|y(T(M_1,y_0);y_0,\frac{M_2}{M_1}u^*_1)\big\|
  &\leq& \big\|y(T(M_1,y_0);y_0,\frac{M_2}{M_1}u^*_1) -y(T(M_1,y_0);y_0,u^*_1)\big\|
  \nonumber\\
  & & + \|y(T(M_1,y_0);y_0,u^*_1)\|
  \nonumber\\
  &\leq& \frac{M_1-M_2}{M_1} \int_0^{T(M_1,y_0)} \|e^{\Delta(T(M_1,y_0)-t)} \chi_\omega u^*_1(t)\| \,\mathrm dt  +r.
\end{eqnarray*}
Since
\begin{eqnarray*}
\|e^{\Delta t}\|_{\mathcal L(L^2(\Omega),L^2(\Omega))} \leq e^{-\lambda_1 t}\;\;
\mbox{for each}\;\;t\geq 0,
 \end{eqnarray*}
 the above, along with H\"{o}lder's inequality and the second inequality in (\ref{NP-Lip-2}), yields that
\begin{eqnarray}\label{NP-Lip-3}
  \big\|y(T(M_1,y_0);y_0,\frac{M_2}{M_1}u^*_1)\big\|
  &\leq& r+\frac{M_1-M_2}{M_1} \frac{1}{\sqrt{2\lambda_1}} M_1
   \nonumber\\
  &\leq& r+(M_1-M_2)/\sqrt{\lambda_1}.
\end{eqnarray}

Next, we define a control $u_2$ over $\mathbb{R}^+$ as follows:
  \begin{eqnarray}\label{NP-Lip-6}
   u_2(t)=\left\{\begin{array}{ll}
         \frac{M_2}{M_1}u^*_1(t),~&t\in (0,T(M_1,y_0)],\\
         0,~&t\in (T(M_1,y_0),\infty).
         \end{array}
       \right.
  \end{eqnarray}
From (\ref{NP-Lip-6}) and the second inequality in (\ref{NP-Lip-2}), it follows that
\begin{eqnarray}\label{NP-Lip-7}
 \|u_2\|_{L^2(\mathbb R^+;L^2(\Omega))}\leq M_2.
\end{eqnarray}
Meanwhile, we let
\begin{eqnarray}\label{NP-Lip-4}
 \widehat T \triangleq \frac{1}{\lambda_1}\ln \Big(1+\frac{1}{\lambda_1^{1/2}r}(M_1-M_2) \Big)
 \leq \frac{1}{\lambda_1^{3/2}r}(M_1-M_2).
\end{eqnarray}
Since $u_2=0$ over $\big(T(M_1,y_0),\infty\big)$, by  (\ref{NP-Lip-6}),   (\ref{NP-Lip-3}) and (\ref{NP-Lip-4}), one can easily check that
\begin{eqnarray}\label{NP-Lip-7-1}
 \|y(T(M_1,y_0)+\widehat T;y_0,u_2)\| &\leq& e^{-\lambda_1 \widehat T}\|y(T(M_1,y_0);y_0,u_2)\|
 \nonumber\\
 &\leq& e^{-\lambda_1 \widehat T} \big(r+(M_1-M_2)/\lambda_1^{1/2}\big) = r.
\end{eqnarray}
Now, it follows from  (\ref{NP-Lip-7}) and (\ref{NP-Lip-7-1})  that $u_2$ is an admissible control to $(TP)^{M_2,y_0}$, which drives the solution to $B_r(y_0)$
at time  $T(M_1,y_0)+\widehat T$. This, along with  the optimality of $T(M_2,y_0)$, yields that
\begin{eqnarray*}
T(M_2,y_0)\leq T(M_1,y_0)+\widehat T.
\end{eqnarray*}
From this,  (\ref{NP-Lip-1}) and (\ref{NP-Lip-4}), we find that
\begin{eqnarray*}
 T_2-T_1=T(M_2,y_0)-T(M_1,y_0)\leq\hat T\leq\frac{1}{\lambda_1^{3/2}r}(M_1-M_2)
 \leq\hat T.
\end{eqnarray*}
Since $M_1\triangleq(N(T_1,y_0)$ and $M_2\triangleq N_2(T_2,y_0))$ (see (\ref{0223-low-estimate-NT})), the above  leads to the first inequality in  (\ref{NP-Lip-T}). This ends the proof of Step 1.

\textit{Step 2. To show the second inequality in (\ref{NP-Lip-T})}

 Let $z^*_1$ be the minimizer of $J^{T_1,y_0}$.
 {\it Throughout this step, we simply write $\varphi_1(\cdot)$ and $\varphi_2(\cdot)$ for $\varphi(\cdot;T_1,z_1^*)$
 and  $\varphi(\cdot;T_2,z_1^*)$ respectively.}
  First, we claim that
\begin{eqnarray}\label{160904-lip-s2-p2.1}
 \|\varphi_2(T_2-T_1)\|_{H^2(\Omega)\cap H_0^1(\Omega)}
  \leq   e^{C_{21}(1+\frac{1}{T_1})} N(T_1,y_0)\;\;\mbox{for some}\;\;C_{21}\triangleq C_{21}(\Omega,\omega).
\end{eqnarray}
Indeed, according to \cite[Theorem 6.13 in Chapter 2]{Pazy}, there is $C_{22}\triangleq C_{22}(\Omega)>0$ so that
\begin{eqnarray*}
\|\Delta e^{\Delta s}\|_{\mathcal L(L^2(\Omega),L^2(\Omega))} \leq C_{22}/s\;\;\mbox{for each}\;\; s>0.
\end{eqnarray*}
From this,
 we see  that
\begin{eqnarray*}
 \|\varphi_2(T_2-T_1)\|_{H^2(\Omega)\cap H_0^1(\Omega)}
 =\|\Delta \varphi_2(T_2-T_1)\|
 =\|\Delta e^{\Delta \frac{T_1}{2} } \varphi_2(T_2-T_1/2)\|
  \leq \frac{2C_{22}}{T_1} \| \varphi_2(T_2-T_1/2)\|.
\end{eqnarray*}
This, along with \cite[Proposition 3.1]{FZ}, yields that for some $C_{23}\triangleq C_{23}(\Omega,\omega)>0$,
\begin{eqnarray}\label{160904-lip-s2-p2.1-1}
 \|\varphi_2(T_2-T_1)\|_{H^2(\Omega)\cap H_0^1(\Omega)}\leq\frac{2C_{22}}{T_1}  e^{C_{23}(1+\frac{2}{T_1})}
\|\chi_\omega\varphi_2(\cdot)\|_{L^2(T_2-T_1/2,T_2;L^2(\Omega))}.
\end{eqnarray}
Meanwhile, by (ii) of Theorem \ref{op-control-0} and  the time-invariance of Equation (\ref{adjoint-1}), we see   that
\begin{eqnarray*}
N(T_1,y_0)=\|\chi_\omega\varphi_1(\cdot)\|_{L^2(0,T_1;L^2(\Omega))}
= \|\chi_\omega\varphi_2(\cdot)\|_{L^2(T_2-T_1,T_2;L^2(\Omega))}.
\end{eqnarray*}
This, along with (\ref{160904-lip-s2-p2.1-1}), yields that
\begin{eqnarray*}
 \|\varphi_2(T_2-T_1)\|_{H^2(\Omega)\cap H_0^1(\Omega)}
 \leq e^{2C_{22}} e^{\frac{1}{T_1}}
 e^{C_{23}(1+\frac{2}{T_1})}   N(T_1,y_0),
\end{eqnarray*}
which leads to (\ref{160904-lip-s2-p2.1}).

Next, since $0<T_1\leq T_2<T^*_{y_0}$, it follows by (i) and (ii) of Theorem \ref{Theorem-eq-cont-discrete} that $N(T_1,y_0)\geq N(T_2,y_0)$. From this and (iii) of Theorem \ref{Lemma-Norm-fun}, it follows that
\begin{eqnarray}\label{160904-lip-s2-1}
 V(T_1,y_0)=-\frac{1}{2} N(T_1,y_0)^2
 \leq
 -\frac{1}{2} N(T_2,y_0)^2= V(T_2,y_0).
\end{eqnarray}
This, along with  (\ref{fun-0}), yields that
\begin{eqnarray}\label{160904-lip-s2-2}
 0&\leq& V(T_2,y_0)-V(T_1,y_0)
 \leq J^{T_2,y_0}(z^*_1)-J^{T_1,y_0}(z^*_1)
 \nonumber\\
 &\leq& \frac{1}{2} \Big[\int_0^{T_2} \|\chi_\omega\varphi_2(t)\|^2 \,\mathrm dt -  \int_0^{T_1} \|\chi_\omega\varphi_1(t)\|^2 \,\mathrm dt\Big]
  +  \langle y_0,\varphi_2(0)-\varphi_1(0) \rangle.
\end{eqnarray}
At the same time, by the time-invariance of  Equation (\ref{adjoint-1}), we have that
\begin{eqnarray}\label{160904-lip-s2-3-1}
 \varphi_1(t)=\varphi_2(t+T_2-T_1)
 \;\;\mbox{for each}\;\;t\in(0,T_1).
\end{eqnarray}
Since the semigroup $\{e^{\Delta t}\}_{t\geq 0}$ is contractive, from (\ref{160904-lip-s2-3-1}), we see that
\begin{eqnarray}\label{160904-lip-s2-3}
   \int_0^{T_2} \|\chi_\omega\varphi_2(t)\|^2 \,\mathrm dt-\int_0^{T_1} \|\chi_\omega\varphi_1(t)\|^2 \,\mathrm dt
   \leq (T_2-T_1) \|\varphi_2(T_2-T_1)\|^2.
\end{eqnarray}
From (\ref{160904-lip-s2-3-1}), we also have that
\begin{eqnarray}\label{160904-lip-s2-4}
 & &    \langle y_0,\varphi_2(0)-\varphi_1(0) \rangle
   = \langle y_0,\varphi_2(0)-\varphi_2(T_2-T_1) \rangle
\\
 &\leq& \|y_0\| \big\|\int_0^{T_2-T_1} \partial_t\varphi_2(t)\,\mathrm dt \big\|
 = \|y_0\| \big\|\int_0^{T_2-T_1} e^{\Delta(T_2-T_1-t)}\Delta \varphi_2(T_2-T_1)\,\mathrm dt \big\|
 \nonumber\\
 &\leq& (T_2-T_1)  \|y_0\|\|\varphi_2(T_2-T_1) \|_{H^2(\Omega)\cap H_0^1(\Omega)}.
 \nonumber
\end{eqnarray}

Now, by  (\ref{160904-lip-s2-2}), (\ref{160904-lip-s2-3}) and (\ref{160904-lip-s2-4}), we obtain that
\begin{eqnarray*}
 0&\leq& V(T_2,y_0)-V(T_1,y_0)
 \nonumber\\
  &\leq& (T_2-T_1)\Big[\|\varphi_2(T_2-T_1)\|_{H^2(\Omega)\cap H_0^1(\Omega)}^2
   +\|y_0\| \|\varphi_2(T_2-T_1)\|_{H^2(\Omega)\cap H_0^1(\Omega)} \Big].
\end{eqnarray*}
By this, (\ref{160904-lip-s2-1}) and (\ref{160904-lip-s2-p2.1}), we get that
\begin{eqnarray}\label{160904-lip-s2-6}
 N(T_1,y_0)-N(T_2,y_0)&\leq& \frac{2}{N(T_2,y_0)+N(T_1,y_0)} \big(V(T_2,y_0)-V(T_1,y_0)\big)
 \nonumber\\
 &\leq& 2e^{C_{21}(1+\frac{1}{T_1})}(N(T_1,y_0)+\|y_0\|) (T_2-T_1).
\end{eqnarray}

Finally, by \cite[Proposition 3.1]{FZ}, we can find  $u_{T_1}\in L^2(0,T_1;L^2(\Omega))$ so that
\begin{eqnarray*}
 y(T_1;y_0,u_{T_1})=0
 \;\;\mbox{and}\;\;
 \|u_{T_1}\|_{L^2(0,T_1;L^2(\Omega))} \leq e^{C_{24}(1+\frac{1}{T_1})} \|y_0\|\;\;\mbox{for some}\;\;C_{24}\triangleq C_{24}(\Omega,\omega).
\end{eqnarray*}
From the first equality in the above, we see that $u_{T_1}$ is an admissible to $(NP)^{T_1,y_0}$. This, along with the second inequality in the above and the optimality of $N(T_1,y_0)$, indicates
\begin{eqnarray*}
 N(T_1,y_0)\leq \|u_{T_1}\|_{L^2(0,T_1;L^2(\Omega))} \leq e^{C_{24}(1+\frac{1}{T_1})} \|y_0\|,
\end{eqnarray*}
which, along with (\ref{160904-lip-s2-6}), leads to the second inequality in (\ref{NP-Lip-T}) for some $C_3\triangleq C_3(\Omega,\omega)$.

In summary, we finish the proof of Theorem~\ref{Proposition-NP-Lip-T}.
\end{proof}

\begin{theorem}\label{theorem-NP-delta-error}
Let $y_0\in L^2(\Omega)\setminus B_r(0)$. Let $\mathcal P_{T^*_{y_0}}$ and $T^*_{y_0}$
be given by (\ref{P-T*}) and (\ref{T*}) respectively.
Then there is $C_4\triangleq C_4(\Omega,\omega)>0$ so that for each $(\delta,k)\in \mathcal P_{T^*_{y_0}}$,
\begin{eqnarray}\label{NP-delta-error-0}
0\leq N_\delta(k\delta,y_0)-N(k\delta,y_0) \leq e^{C_4 (1+ T^*_{y_0}+\frac{1}{k\delta}+\frac{1}{T^*_{y_0}-k\delta} )} \|y_0\|^{12} r^{-11} \delta^2.
\end{eqnarray}
\end{theorem}
\begin{proof}
Arbitrarily fix $(\delta,k)\in \mathcal P_{T^*_{y_0}}$. Let  $z^*_\delta$ be the minimizer of $J^{k\delta,y_0}_\delta$.
{\it Throughout the proof of Theorem~\ref{theorem-NP-delta-error}, we simply write respectively $\varphi(\cdot)$ and $\overline\varphi_\delta(\cdot)$
for $\varphi(\cdot;k\delta,z^*_\delta)$ (see (\ref{adjoint-1})) and $\overline \varphi(\cdot;k\delta,z^*_\delta)$ (see (\ref{adjoint-solution-P})).}
 We organize the proof by  several steps as follows:

\textit{Step 1. To prove that
\begin{eqnarray}\label{NP-delta-error-1}
 0\leq V(k\delta,y_0)-V_\delta(k\delta,y_0)
 \leq  \|\chi_\omega\varphi\|^2_{L^2(0,k\delta; L^2(\Omega))}- \|\chi_\omega\overline\varphi_\delta\|^2_{L^2(0,k\delta; L^2(\Omega))}
 \end{eqnarray}}
$\;\;$ Since $ L^2_\delta (0,k\delta;L^2(\Omega))
 \subset  L^2(0,k\delta;L^2(\Omega))$ (see (\ref{heat-Linfty-delta})), we find that each admissible control to $(NP)_\delta^{k\delta,y_0}$ is also an admissible control to $(NP)^{k\delta,y_0}$. This, along with  (\ref{0113-NP-0}) and (\ref{0113-NP-delta}), yields that
$N(k\delta,y_0)\leq N_\delta(k\delta,y_0)$, from which, as well as (iii) of Theorems \ref{Lemma-Norm-fun} and (iii) of Theorem~\ref{Lemma-NP-bangbang}, it follows that
\begin{eqnarray}\label{NP-delta-error-2}
 V_\delta(k\delta,y_0)=-\frac{1}{2} N_\delta(k\delta,y_0)^2
 \leq
 -\frac{1}{2}N(k\delta,y_0)^2= V(k\delta,y_0).
\end{eqnarray}
This, along with  (\ref{fun-0}) and (\ref{fun-P}), yields that
\begin{eqnarray*}
 0&\leq& V(k\delta,y_0)-V_\delta(k\delta,y_0)
 \leq J^{k\delta,y_0}(z^*_\delta)-J_\delta^{k\delta,y_0}(z^*_\delta)
\\
  &\leq& \frac{1}{2}\big[\|\chi_\omega\varphi\|^2_{L^2(0,k\delta; L^2(\Omega))}- \|\chi_\omega\overline\varphi_\delta\|^2_{L^2(0,k\delta; L^2(\Omega))}\big],
\end{eqnarray*}
which leads to (\ref{NP-delta-error-1}).

\textit{Step 2. To show that
\begin{eqnarray}\label{NP-delta-error-3}
 \|\chi_\omega\varphi\|^2_{L^2(0,k\delta; L^2(\Omega))}- \|\chi_\omega\overline\varphi_\delta\|^2_{L^2(0,k\delta; L^2(\Omega))}=
  \|\chi_\omega\varphi-
 \chi_\omega\overline\varphi_\delta\|^2_{L^2(0,k\delta; L^2(\Omega))}
 \end{eqnarray}}
$\;\;$ First, we claim that for each $f\in L^2(\mathbb R^+;L^2(\Omega))$,
\begin{eqnarray}\label{NP-delta-error-4}
 \|f\|^2_{L^2(\mathbb R^+;L^2(\Omega))}
 = \|\bar f_\delta\|^2_{L^2(\mathbb R^+;L^2(\Omega))}
 +\|f-\bar f_\delta\|^2_{L^2(\mathbb R^+;L^2(\Omega))},
\end{eqnarray}
where $\bar f_\delta$ is given by (\ref{adjoint-solution-P-1}). Indeed, for an arbitrarily fixed $f\in L^2(\mathbb R^+;L^2(\Omega))$, one can directly check that
\begin{eqnarray}\label{NP-delta-error-5}
 \|f\|^2_{L^2(\mathbb R^+;L^2(\Omega))}
 =\|\bar f_\delta\|^2_{L^2(\mathbb R^+;L^2(\Omega))}
 +\|f-\bar f_\delta\|^2_{L^2(\mathbb R^+;L^2(\Omega))}
 +2\langle \bar f_\delta,f- \bar f_\delta\rangle_{L^2(\mathbb R^+;L^2(\Omega))}
 \nonumber.
\end{eqnarray}
Meanwhile,  it follows by (\ref{adjoint-solution-P-1}) that $\bar g_\delta=0$, where $g\triangleq f- \bar f_\delta$. Then by Lemma \ref{lemma-control-op-dual}, we obtain that
\begin{eqnarray*}
 \langle \bar f_\delta,f- \bar f_\delta\rangle_{L^2(\mathbb R^+;L^2(\Omega))}
 = \langle \bar f_\delta,g\rangle_{L^2(\mathbb R^+;L^2(\Omega))}
 = \langle \bar f_\delta,\bar g_\delta\rangle_{L^2(\mathbb R^+;L^2(\Omega))}
 =0.
\end{eqnarray*}
This, along with (\ref{NP-delta-error-5}), leads to (\ref{NP-delta-error-4}).

Next, by taking $f$ to be the zero extension of $\varphi$ over $\mathbb R^+$  in (\ref{NP-delta-error-4}), we obtain (\ref{NP-delta-error-3}). Here, we used the fact that in this case, $\bar f_\delta(\cdot)$ is the zero extension of $\chi_\omega\overline \varphi_\delta(\cdot)$ over $\mathbb{R}^+$, which follows from
 (\ref{adjoint-solution-P-1}) and (\ref{adjoint-solution-P}).

\textit{Step 3. To verify that there exists $C_{41}\triangleq C_{41}(\Omega,\omega)>0$ so that
\begin{eqnarray}\label{NP-delta-error-6}
 \|\chi_\omega\varphi-
 \chi_\omega\overline\varphi_\delta\|^2_{L^2(0,k\delta; L^2(\Omega))}
  \leq  e^{C_{41}(1+k\delta+\frac{1}{k\delta})}\|y_0\|^{12} r^{-10} \delta^2
\end{eqnarray}}
$\;\;$ From (\ref{adjoint-solution-P}), it follows that
\begin{eqnarray*}
 & & \int_0^{k\delta}
 \|\chi_\omega\varphi(t)-
 \chi_\omega\overline\varphi_\delta(t)\|^2
 \,\mathrm dt
  =  \sum_{j=1}^{k}
 \int_{(j-1)\delta}^{j\delta} \Big\|\chi_\omega\varphi(t)- \frac{1}{\delta} \int_{(j-1)\delta}^{j\delta} \chi_\omega\varphi(s) \,\mathrm ds \Big\|^2 \,\mathrm dt
 \nonumber\\
 &=& \sum_{j=1}^{k}
 \int_{(j-1)\delta}^{j\delta} \Big\| \frac{1}{\delta} \int_{(j-1)\delta}^{j\delta} \int_s^t \chi_\omega\partial_\tau  \varphi(\tau) \,\mathrm d\tau\mathrm ds \Big\|^2 \,\mathrm dt
 \leq \sum_{j=1}^{k}
 \int_{(j-1)\delta}^{j\delta} \Big( \int_{(j-1)\delta}^{j\delta}\| \partial_\tau  \varphi(\tau) \| \,\mathrm d\tau \Big)^2  \, \mathrm dt.
 \end{eqnarray*}
 Applying  the H\"{o}lder inequality to the above leads to that
 \begin{eqnarray*}
  \|\chi_\omega\varphi-
 \chi_\omega\overline\varphi_\delta\|^2_{L^2(0,k\delta; L^2(\Omega))}
  \leq  \delta^2\|\partial_t  \varphi\|^2_{L^2(0,k\delta; L^2(\Omega))}.
 \end{eqnarray*}
This, along with (\ref{z-beta-uni-bdd}), implies (\ref{NP-delta-error-6}) for some $C_{41}\triangleq C_{41}(\Omega,\omega)$.

\textit{Step 4. To show (\ref{NP-delta-error-0})}

 We first claim that
\begin{eqnarray}\label{NP-delta-error-7}
 N(k\delta,y_0)
 \geq  e^{-\frac{2}{\lambda_1}-\frac{1}{T^*_{y_0}-k\delta}} r.
 \end{eqnarray}
In fact, by (i) of Theorem \ref{Theorem-eq-cont-discrete}, we have that
\begin{eqnarray*}
\lim_{T_2\rightarrow {T^*_{y_0}}^-} N(T_2,y_0)=0.
\end{eqnarray*}
 This, along with
   the first inequality in  (\ref{NP-Lip-T}) (where $T_1=k\delta$), yields that
 \begin{eqnarray}\label{NP-delta-error-8}
  N(k\delta,y_0)&=& \lim_{T_2\rightarrow {T^*_{y_0}}^-} (N(k\delta,y_0)-N(T_2,y_0))
  \nonumber\\
  &\geq& \lim_{T_2\rightarrow {T^*_{y_0}}^-} \lambda_1^{3/2}  r  (T_2-k\delta)
  =  \lambda_1^{3/2}r  (T^*_{y_0}-k\delta).
 \end{eqnarray}
 Since we clearly have that
 \begin{eqnarray*}
 \lambda_1 \geq e^{-\frac{1}{\lambda_1}}\;\;\mbox{and}\;\;T^*_{y_0}-k\delta \geq e^{-\frac{1}{T^*_{y_0}-k\delta}},
  \end{eqnarray*}
  (\ref{NP-delta-error-7}) follows from  (\ref{NP-delta-error-8}) at once.

 Meanwhile, from (\ref{NP-delta-error-1}), (\ref{NP-delta-error-3}) and (\ref{NP-delta-error-6}), we obtain that
 \begin{eqnarray*}
  0\leq V(k\delta,y_0) - V_\delta(k\delta,y_0)  \leq e^{C_{41}(1+k\delta+\frac{1}{k\delta})}  \|y_0\|^{12} r^{-10} \delta^2.
 \end{eqnarray*}
 From this, (\ref{NP-delta-error-2}) and (\ref{NP-delta-error-7}), we find that
 \begin{eqnarray*}
  0\leq N_\delta(k\delta,y_0)-N(k\delta,y_0)
  &=& \frac{2V(k\delta,y_0)-2V_\delta(k\delta,y_0)}
  {N(k\delta,y_0)+N_\delta(k\delta,y_0)}
  \nonumber\\
  &\leq& 2 e^{\frac{2}{\lambda_1}+\frac{1}{T^*_{y_0}-k\delta}}
  e^{C_{41}(1+k\delta+\frac{1}{k\delta})}  \|y_0\|^{12} r^{-11} \delta^2.
 \end{eqnarray*}
 Since $k\delta<T^*_{y_0}$, the above leads to  (\ref{NP-delta-error-0}) for some $C_4\triangleq C_4(\Omega,\omega)$.

 In summary, we end the proof of Theorem~\ref{theorem-NP-delta-error}.
\end{proof}

\subsection{Some properties on minimal time functions}

Some inequalities, as well as properties, related to the minimal time functions $M\rightarrow T_\delta(M,y_0)$ and $M\rightarrow T(M,y_0)$ will be given in this subsection.

\begin{theorem}\label{theorem-optimality}
Let $y_0\in L^2(\Omega)\setminus B_r(0)$. For each $M>0$ and $\eta\in(0,1)$, there is a measurable subset $\mathcal A_{M,\eta}\subset (0,1)$
(depending also on $y_0$ and $r$), with $\lim_{h\rightarrow 0^+} \frac{1}{h} |\mathcal A_{M,\eta}\cap (0,h)|=\eta$, so that for each $\delta\in\mathcal A_{M,\eta}$, there is $a_\delta\in (0,\eta)$ so that
\begin{eqnarray}\label{ineqs-optimality}
 T_\delta(M,y_0) - T(M,y_0) = (1-a_\delta) \delta
 \;\;\mbox{and}\;\;
  M\geq N_\delta(T_\delta(M,y_0),y_0) + \frac{1}{2} \lambda_1^{3/2} r (1-\eta) \delta.
\end{eqnarray}
\end{theorem}
\begin{proof}
Arbitrarily fix $M>0$ and $\eta\in(0,1)$.
{\it Throughout the proof of Theorem~\ref{theorem-optimality}, we simply write $T^*\triangleq T(M,y_0)$.}
For each  $k\in\mathbb N^+$ and $a\in(0,\eta)$, we define a subset of $\mathbb R^+$ in the following manner:
\begin{eqnarray}\label{160907-th1-op-1}
 \mathcal B_{M,\eta}^{k,a} \triangleq
 \{\delta>0 ~:~
  (k+a)\delta=T^*\}.
\end{eqnarray}
We then define another subset of $\mathbb R^+$ as follows:
\begin{eqnarray}\label{160907-th1-op-2}
\mathcal B_{M,\eta} \triangleq  \cup_{k\in\mathbb N^+ }\cup_{a\in(0,\eta)} \mathcal B_{M,\eta}^{k,a}.
\end{eqnarray}
The rest proof is divided into the following two steps:

\textit{Step 1. To prove that $\lim_{h\rightarrow0^+} \frac{1}{h} |\mathcal B_{M,\eta}\cap(0,h)|=\eta$}

From (\ref{160907-th1-op-1}), we see that
\begin{eqnarray*}
\cup_{a\in(0,\eta)} \mathcal B_{M,\eta}^{k,a}=\left({T^*}/{(k+\eta)},{T^*}/{k}\right)
\;\;\mbox{for each}\;\;k\in\mathbb N^+.
\end{eqnarray*}
From this and (\ref{160907-th1-op-2}), it follows that
\begin{eqnarray}\label{160907-th1-op-3}
 \mathcal B_{M,\eta} = \cup_{k\in\mathbb N^+ } \left({T^*}/{(k+\eta)},{T^*}/{k}\right).
\end{eqnarray}
For each $j\in \mathbb{N}^+$, we let $h_j\triangleq T^*/j$. For each $h\in (0,T^*)$, we let $j(h)$ be the integer
 so that
$h_{j(h)+1}\leq h<h_{j(h)}$. Then, by (\ref{160907-th1-op-3}), one can easily verify that
\begin{eqnarray}\label{Wang7.19}
\lim_{h\rightarrow 0^+} \frac{|\mathcal{B}_{M,\eta}\cap (0,h_{j(h)})|}{h_{j(h)}} =\eta;\;\;\lim_{h\rightarrow0^+}\frac{h_{j(h)+1}}{h}
=\lim_{h\rightarrow0^+}\frac{h_{j(h)}}{h}=1;
\end{eqnarray}
\begin{eqnarray}\label{Wang7.21}
\frac{h_{j(h)+1}}{h} \frac{|\mathcal{B}_{M,\eta}\cap (0,h_{j(h)+1})|}{h_{j(h)+1}}
\leq \frac{|\mathcal{B}_{M,\eta}\cap (0,h)|}{h}
\leq \frac{|\mathcal{B}_{M,\eta}\cap (0,h_{j(h)})|}{h_{j(h)}} \frac{h_{j(h)}}{h}.
\end{eqnarray}
From (\ref{Wang7.19}) and (\ref{Wang7.21}), we can easily
obtain the conclusion in Step 1.

\textit{Step 2. To show   (\ref{ineqs-optimality})}

We first claim that for each $\delta\in\mathcal B_{M,\eta}\cap (0,1)$, there is a unique pair $(k_\delta,a_\delta)$  so that
\begin{eqnarray}\label{160907-th1-op-5}
  (k_\delta+a_\delta)\delta=T^*
  \;\;\mbox{with}\;\;
  k_\delta\in\mathbb N^+
  \;\;\mbox{and}\;\;
  a_\delta\in(0,\eta).
\end{eqnarray}
Indeed, the existence  of such a pair follows from
   (\ref{160907-th1-op-2}) and (\ref{160907-th1-op-1}) at once, while the uniqueness of such pairs can be directly checked.

   Thus, for each $\delta\in\mathcal B_{M,\eta}\cap (0,1)$, we can define $k_\delta$ to be the first component of the unique pair satisfying (\ref{160907-th1-op-5}).
     We next claim that there exists $\delta^1_{M,\eta}\in(0,1)$ so that
\begin{eqnarray}\label{160907-th1-op-4}
M\geq N_\delta((k_\delta+1)\delta,y_0)
+ \frac{1}{2} \lambda_1^{3/2} r (1-\eta)\delta\;\;\mbox{for each}\;\;\delta\in \mathcal B_{M,\eta} \cap (0,\delta^1_{M,\eta}).
\end{eqnarray}
To this end, we  notice that $T^*<T^*_{y_0}$ (see (iii) of Theorem~\ref{Theorem-eq-cont-discrete}). Arbitrarily fix $\delta\in \mathcal B_{M,\eta}\cap(0,1)$ so that
\begin{eqnarray}\label{160907-th1-op-6}
 0<\delta<\min\{T^*/2,(T^*_{y_0}-T^*)/2\}.
\end{eqnarray}
(The existence of such $\delta$ is ensured by (\ref{160907-th1-op-3}).) Then it follows from (\ref{160907-th1-op-6}) and (\ref{160907-th1-op-5}) that
\begin{eqnarray*}
 2\delta<T^*< (k_\delta+1)\delta<T^*+\delta<{(T^*_{y_0}+T^*)}/{2}<T^*_{y_0}.
\end{eqnarray*}
This, along with (\ref{P-T*}), yields that
\begin{eqnarray*}
(\delta,k_\delta)\in \mathcal{P}^*_{T^*_{y_0}}\;\;\mbox{and}\;\;
2\delta<T^*<(k_\delta+1)\delta<T^*_{y_0}.
\end{eqnarray*}
By these, we can apply Theorem \ref{theorem-NP-delta-error} (with $(\delta,k)=(\delta, k_\delta)$)
and Theorem~\ref{Proposition-NP-Lip-T} (with $T_1=T^*$ and $T_2=(k_\delta+1)\delta$) to get that
\begin{eqnarray}\label{160907-th1-op-7}
N_\delta((k_\delta+1)\delta,y_0)
&\leq& N(( k_\delta+1)\delta,y_0) + e^{C_4 \big[1+ T^*_{y_0}+\frac{1}{( k_\delta+1)\delta}+\frac{1}{T^*_{y_0}-( k_\delta+1)\delta} \big]} \|y_0\|^{12} r^{-11} \delta^2
\nonumber\\
&\leq& N(T^*,y_0)- \lambda_1^{3/2} r  \big((k_\delta+1)\delta - T^* \big) +
\nonumber\\
& &   e^{C_4 \big[1+ T^*_{y_0}+\frac{1}{( k_\delta+1)\delta}+\frac{1}{T^*_{y_0}-( k_\delta+1)\delta} \big]} \|y_0\|^{12} r^{-11} \delta^2,
\end{eqnarray}
where $C_4$ is given by (\ref{NP-delta-error-0}).
Meanwhile, by (\ref{160907-th1-op-5}) and (\ref{160907-th1-op-6}), we find that
\begin{eqnarray*}
 (k_\delta+1)\delta - T^*\geq (1-\eta) \delta
 \;\;\mbox{and}\;\;
 T^*< ( k_\delta+1)\delta < (T^*_{y_0}+T^*)/2.
\end{eqnarray*}
These, along with (\ref{160907-th1-op-7}) and (ii) of  Theorem \ref{Theorem-eq-cont-discrete}, yield that
\begin{eqnarray*}
  N_\delta((k_\delta+1)\delta,y_0)
  &\leq& N(T^*,y_0) - \lambda_1^{3/2} r  (1-\eta)\delta +
   e^{C_4 \big[1+ T^*_{y_0}+\frac{1}{T^*}+\frac{2}{T^*_{y_0}-T^*} \big]} \|y_0\|^{12} r^{-11} \delta^2
  \nonumber\\
  &=&
   M - \lambda_1^{3/2} r  (1-\eta)\delta +
   e^{C_4 \big[1+ T^*_{y_0}+\frac{1}{T^*}+\frac{2}{T^*_{y_0}-T^*} \big]} \|y_0\|^{12} r^{-11} \delta^2.
\end{eqnarray*}
By this and (\ref{160907-th1-op-6}), we obtain (\ref{160907-th1-op-4}).

Define a set $\mathcal A_{M,\eta}$ in the following manner:
 \begin{eqnarray}\label{Yuyongyu5.56}
 \mathcal A_{M,\eta}\triangleq \mathcal B_{M,\eta} \cap (0,\delta^1_{M,\eta}),
  \;\;\mbox{with}\;\;\delta^1_{M,\eta}\;\;\mbox{given by}\;\; (\ref{160907-th1-op-4}).
  \end{eqnarray}
We now show that the second conclusion in (\ref{ineqs-optimality})
holds for each $\delta$ in  $\mathcal A_{M,\eta}$ defined by (\ref{Yuyongyu5.56}). To this end, we arbitrarily fix $\delta\in\mathcal A_{M,\eta}$.
We claim that
\begin{eqnarray}\label{gwang7.28}
  T_\delta(M,y_0)\leq (k_\delta+1)\delta\;\;\mbox{and}\;\; T_\delta(M,y_0)>k_\delta \delta.
  \end{eqnarray}
To show the first inequality in (\ref{gwang7.28}), we let $u_\delta$ be an admissible control to $(NP)^{( k_\delta+1)\delta,y_0}_\delta$ and let $\tilde u_\delta$ be the zero extension of $u_\delta$ over $\mathbb R^+$.
 Since $N_\delta((k_\delta+1)\delta, y_0)\leq M$ (see (\ref{160907-th1-op-4})),
one can easily check
 that $\tilde u_\delta$ is an admissible control (to $(TP)^{M,y_0}_\delta$), which drives the solution to $B_r(0)$ at time $(k_\delta+1)\delta$.
  This, along with  the optimality of $T_\delta(M,y_0)$, leads to the first inequality in (\ref{gwang7.28}). To prove the second inequality in (\ref{gwang7.28}), we notice that
       $\mathcal U^M_\delta\subset\mathcal U^M$. This, along with  (\ref{time-1}) and (\ref{time-2}), yields  that $T^* \leq T_\delta(M,y_0)$. From this and (\ref{160907-th1-op-5}), we obtain the second inequality in (\ref{gwang7.28}).

           Since $T_\delta(M,y_0)$ is a multiple of $\delta$ (see (\ref{U-ad-delta})),
           it follows from (\ref{gwang7.28})  that
       \begin{eqnarray}\label{0919-s5-1}
        T_\delta(M,y_0)=(k_\delta+1)\delta.
       \end{eqnarray}
      This, along with (\ref{160907-th1-op-4}),  implies that the second conclusion in (\ref{ineqs-optimality}).

       Finally, from (\ref{0919-s5-1}) and  (\ref{160907-th1-op-5}), we see that   the first conclusion in (\ref{ineqs-optimality})  is true for each $\delta\in \mathcal A_{M,\eta}$ defined by (\ref{Yuyongyu5.56}).

       In summary, we end the proof of Theorem~\ref{theorem-optimality}.
\end{proof}

\begin{theorem}\label{Lemma-T*}
 For each $M\geq 0$,
\begin{eqnarray}\label{0916-TP-lip}
 |T(M,y_1)-T(M,y_2)| \leq
 \frac{1}{\lambda_1 r} \|y_1-y_2\|\;\;\mbox{for all}\;\;y_1,\,y_2\in L^2(\Omega)\setminus B_r(0).
\end{eqnarray}
 Besides, $T(0,y_0)=T^*_{y_0}$ for each $y_0\in L^2(\Omega)\setminus B_r(0)$. Here, $T^*_{y_0}$ is given by (\ref{T*}).
\end{theorem}

\begin{proof}
It follows from (\ref{T*}) and (\ref{time-1}) that $T(0,y_0)=T^*_{y_0}$. To prove (\ref{0916-TP-lip}), we arbitrarily fix $y_1,\,y_2\in L^2(\Omega)\setminus B_r(0)$.
 We first claim that
\begin{eqnarray}\label{0916-TP-lip-1}
 T(M,y_1)-T(M,y_2) \leq
 \frac{1}{\lambda_1 r} \|y_1-y_2\|.
\end{eqnarray}
Let   $\hat u_2 \triangleq \chi_{(0,T(M,y_2))} u_2$,
where $u_2$ is the optimal control to $(TP)^{M,y_2}$. Then we have  that
\begin{eqnarray}\label{ST-12-1}
 y(T(M,y_2);y_2,\hat u_2)\in B_r(0)
 \;\;\mbox{and}\;\;
 \|\hat u_{2}\|_{L^2(\mathbb R^+;L^2(\Omega))}\leq M.
\end{eqnarray}
Define a number $\widehat T_1$ by
\begin{eqnarray}\label{0118-sec-102-1}
 \widehat T_1  \triangleq \frac{1}{\lambda_1} \ln \big(1+\frac{1}{r}\|y_1-y_2\|\big)
 \leq \frac{1}{\lambda_1 r} \|y_1-y_2\|.
\end{eqnarray}
Since $\hat u_2=0$ over $\big(T(M,y_2),\infty\big)$, by the first conclusion in (\ref{ST-12-1}) and (\ref{0118-sec-102-1}), we see that
\begin{eqnarray*}
 \|y(T(M,y_2)+\widehat T_1;y_1, \hat u_2)\|
  &\leq&  e^{-\lambda_1 \widehat T_1 } \big(\|y(T(M,y_2);y_2, \hat u_2)\|+\|y(T(M,y_2);y_1-y_2, 0)\|\big)
 \nonumber\\
 &\leq& e^{-\lambda_1 \widehat T_1 } \big(r + \|y_1-y_2\|\big)
 =r.
 \end{eqnarray*}
 From this and the second inequality in (\ref{ST-12-1}), we see that
    $\hat u_2$ is an admissible control (to $(TP)^{M,y_1}$) which drives the solution to $B_r(0)$ at time $T(M,y_2)+\widehat T_1$. This, along with  the optimality of $T(M,y_1)$, yields that
$T(M,y_1) \leq T(M,y_2) + \widehat T_1$.
From this and (\ref{0118-sec-102-1}),  we deduce (\ref{0916-TP-lip-1}).

Next, by a very similar way as that to show (\ref{0916-TP-lip-1}), we can easily check that
\begin{eqnarray}\label{0916-TP-lip-2}
 T(M,y_2)-T(M,y_1) \leq
 \frac{1}{\lambda_1 r} \|y_1-y_2\|.
\end{eqnarray}

 Finally, (\ref{0916-TP-lip}) follows from (\ref{0916-TP-lip-1}) and (\ref{0916-TP-lip-2}) immediately. This completes the proof of Theorem~\ref{Lemma-T*}.
\end{proof}

\section{The proofs of the main theorems}

In this section, we will prove Theorems  \ref{Theorem-time-order}-\ref{Theorem-time-control-convergence-1}. The strategies to show them  have been introduced in Subsection 1.3.

\subsection{The proof of Theorem \ref{Theorem-time-order}}

\begin{proof}[Proof of Theorem \ref{Theorem-time-order}]
 Let $y_0\in L^2(\Omega)\setminus B_r(0)$. We will prove the conclusions (i) and (ii) in Theorem \ref{Theorem-time-order} one by one.

(i)
 Arbitrarily fix $0<M_1<M_2$. Then arbitrarily fix $M\in [M_1,M_2]$. From the conclusion (iii) in Theorem \ref{Theorem-eq-cont-discrete}, it follows
 \begin{eqnarray}\label{0906-NP-Lip-1}
    0<T(M_2,y_0)\leq T(M,y_0)\leq T(M_1,y_0)
   < T^*_{y_0}.
 \end{eqnarray}
 We take $\delta$ so that
 \begin{eqnarray}\label{0906-NP-Lip-2}
  0<\delta< \min\left\{T(M_2,y_0)/2,\big(T^*_{y_0}-T(M_1,y_0)\big)/4 \right\}\triangleq \delta_1.
 \end{eqnarray}
 Let $\hat k_\delta\in \mathbb{N}$ satisfy that
 \begin{eqnarray}\label{0906-NP-Lip-3}
  (\hat k_\delta-1)\delta< T(M,y_0)
  \leq \hat k_\delta \delta.
 \end{eqnarray}
   We first claim that
  \begin{eqnarray}\label{0906-NP-Lip-8}
   N_\delta((\hat k_\delta+1)\delta,y_0)
   \leq M+e^{C_4 \big[1+ T^*_{y_0}+\frac{1}{T(M_2,y_0)}+\frac{2}{T^*_{y_0}-T(M_1,y_0)} \big]} \|y_0\|^{12} r^{-11} \delta^2-\lambda_1^{3/2} r  \delta.
  \end{eqnarray}
  Indeed, from  (\ref{P-T*}) and (\ref{0906-NP-Lip-1})-(\ref{0906-NP-Lip-3}), one can easily check that
  \begin{eqnarray}\label{Yuwanggs6.5}
  0<T(M,y_0)<(\hat k_\delta+1)\delta<T^*_{y_0}\;\;\mbox{and}\;\;(\delta,\hat k_\delta+1)\in\mathcal P_{T^*_{y_0}}.
  \end{eqnarray}
  Three facts are given in order: (a) By the second conclusion in (\ref{Yuwanggs6.5}),  we can apply Theorem \ref{theorem-NP-delta-error}, with $(\delta,k)=(\delta,\hat k_\delta+1)$, to obtain that
  \begin{eqnarray*}\label{0906-NP-Lip-6}
   N_\delta((\hat k_\delta+1)\delta,y_0)
   \leq
    N((\hat k_\delta+1)\delta,y_0) + e^{C_4 \big[1+ T^*_{y_0}+\frac{1}{(\hat k_\delta+1)\delta}+\frac{1}{T^*_{y_0}-(\hat k_\delta+1)\delta} \big]} \|y_0\|^{12} r^{-11} \delta^2,
  \end{eqnarray*}
  where $C_4\triangleq C_4(\Omega,\omega)$ is given by (\ref{NP-delta-error-0}).
    (b) By the first conclusion in (\ref{Yuwanggs6.5}), we can
      use the first inequality in (\ref{NP-Lip-T}) in Theorem~\ref{Proposition-NP-Lip-T} (where $T_1=T(M,y_0)$ and $T_2=(\hat k_\delta+1)\delta$) to get that
  \begin{eqnarray*}
   N((\hat k_\delta+1)\delta,y_0)
   &\leq& N(T(M,y_0),y_0) -\lambda_1^{3/2} r   |(\hat k_\delta+1)\delta-T(M,y_0)|
   \nonumber\\
   &\leq&  N(T(M,y_0),y_0) - \lambda_1^{3/2} r  \delta;
  \end{eqnarray*}
  (c) By (ii) of  Theorem \ref{Theorem-eq-cont-discrete}, we have that
  $N(T(M,y_0),y_0)=M$.

   From above three facts (a)-(c), we find that
  \begin{eqnarray}\label{0906-NP-Lip-7}
   N_\delta((\hat k_\delta+1)\delta,y_0)
   \leq M+e^{C_4 \big[1+T^*_{y_0}+\frac{1}{(\hat k_\delta+1)\delta}+\frac{1}{T^*_{y_0}-(\hat k_\delta+1)\delta} \big]}\|y_0\|^{12} r^{-11}\delta^2-\lambda_1^{3/2}r\delta.
    \end{eqnarray}
  Meanwhile, from (\ref{0906-NP-Lip-3}), (\ref{0906-NP-Lip-2}) and
  (\ref{0906-NP-Lip-1}), one can easily check that
  \begin{eqnarray*}
    T(M_2,y_0)  \leq (\hat k_\delta-1)\delta+2\delta \leq T(M_1,y_0)+\frac{T^*_{y_0}-T(M_1,y_0)}{2}= \frac{T^*_{y_0}+T(M_1,y_0)}{2}.
  \end{eqnarray*}
  This, along with  (\ref{0906-NP-Lip-7}), leads to (\ref{0906-NP-Lip-8}).

We next  claim that
  \begin{eqnarray}\label{0906-NP-Lip-9}
   T_\delta(M,y_0) \leq (\hat k_\delta+1)\delta\;\;\mbox{for each}\;\;0<\delta<\delta_0\triangleq\min\{\delta_1, \delta_2\}),
  \end{eqnarray}
  where $\delta_1$ is given by (\ref{0906-NP-Lip-2}) and $\delta_2$ is defined by
  \begin{eqnarray}\label{0906-NP-Lip-10}
   \delta_2 \triangleq  \frac{1}{2}  \lambda_1^{3/2}  e^{-C_4 \big[1+ T^*_{y_0}+\frac{1}{T(M_2,y_0)}+\frac{2}{T^*_{y_0}-T(M_1,y_0)} \big]} \|y_0\|^{-12} r^{12}.
  \end{eqnarray}
 In fact, for an arbitrarily fixed $\delta\in(0,\delta_0)$,
 by (\ref{0906-NP-Lip-8}) and (\ref{0906-NP-Lip-10}), after some computations, we find that
  \begin{eqnarray}\label{0906-NP-Lip-10-1}
    M \geq N_\delta((\hat k_\delta+1)\delta,y_0) + \frac{1}{2} \lambda_1^{3/2} r  \delta
   > N_\delta((\hat k_\delta+1)\delta,y_0).
  \end{eqnarray}
  Let $u_\delta$ be the zero extension of an admissible control (to $(NP)^{(\hat k_\delta+1)\delta,y_0}_\delta$)  over $\mathbb R^+$. Then by (\ref{0906-NP-Lip-10-1}), one can easily check that
 $u_\delta$ is an admissible control (to $(TP)^{M,y_0}_\delta$), which drives the solution to $B_r(0)$ at time $(\hat k_\delta +1)\delta$. This, along with  the optimality of $T_\delta(M,y_0)$, leads to  (\ref{0906-NP-Lip-9}).

We now show (\ref{TP-order-0}) with  $\delta_0$ given by (\ref{0906-NP-Lip-9}). For this purpose, we arbitrarily fix $\delta\in(0,\delta_0)$. Since $\mathcal U^M_\delta\subset \mathcal U^M$, it follows by (\ref{time-1}) and (\ref{time-2}) that $T(M,y_0) \leq T_\delta(M,y_0)$. This, along with (\ref{0906-NP-Lip-3}) and (\ref{0906-NP-Lip-9}), leads to (\ref{TP-order-0}).

Finally, we will prove that the above  $\delta_0$  is continuous in $(M_1,M_2,y_0,r)$. Indeed, according to the definitions of
$\delta_0$,
$\delta_1$ and $\delta_2$ (see  (\ref{0906-NP-Lip-9}), (\ref{0906-NP-Lip-2}) and (\ref{0906-NP-Lip-10}), respectively), it suffices to show that the  maps $y_0\rightarrow T^*_{y_0}$   and $(\hat M,y_0)\rightarrow T(\hat M,y_0)$ are  continuous.
To show the first one, we notice from Theorem~\ref{Lemma-T*} that for each  $y_0\in L^2(\Omega)\setminus B_r(0)$,  $T(0,y_0)=T^*_{y_0}$ and that the map $y_0\rightarrow
T(0,y_0)$ is continuous outside $B_r(0)$. Hence, the  map $y_0\rightarrow T^*_{y_0}$
 is continuous outside $B_r(0)$. To show the continuity of the map $(\hat M,y_0)\rightarrow T(\hat M,y_0)$, two observations are given in order: First, it follows from Theorem~\ref{Lemma-T*} that   $T(\hat M,y_0)$  continuously depends on $y_0\in L^2(\Omega)\setminus B_r(0)$; Second, according to the conclusion (iii) of Theorem \ref{Theorem-eq-cont-discrete}, for an fixed $y_0\in L^2(\Omega)\setminus B_r(0)$, the map $\hat M\rightarrow T(\hat M,y_0)$ is continuous and decreasing.
From these two observations,  the desired continuity follows at once. This ends the proof of the conclusion (i).

(ii)  Let $\mathcal A_{M,\eta}$, with $M>0$ and $\eta\in(0,1)$, be given by Theorem \ref{theorem-optimality}. Then the conclusion (ii) of Theorem \ref{Theorem-time-order}  follows from the first conclusion in (\ref{ineqs-optimality}) at once.

  In summary, we end the proof of Theorem \ref{Theorem-time-order}.
\end{proof}

\subsection{The proof of Theorem \ref{Theorem-time-control-convergence}}

\begin{proof}[Proof of Theorem \ref{Theorem-time-control-convergence}]
Arbitrarily fix $y_0\in L^2(\Omega)\setminus B_r(0)$.
For each $M>0$ and $\delta>0$, we let $u_M^*$ and $u^*_{\delta,M}$ be the  optimal control and the  optimal control with the minimal norm to $(TP)^{M,y_0}$
and $(TP)^{M,y_0}_\delta$ respectively (see Theorem~\ref{Lemma-existences-TP}).
We will prove the conclusions (i)-(ii) one by one.

 (i)  Let $0<M_1<M_2$. Let $\delta_0=\delta_0(M_1,M_2,y_0,r)$ and $C_3=C_3(\Omega,\omega)$ be given by Theorem \ref{Theorem-time-order} and Theorem \ref{Proposition-NP-Lip-T}, respectively. Arbitrarily fix $M\in [M_1,M_2]$ and $\delta>0$.
 {\it In the proof of (i) of Theorem \ref{Theorem-time-control-convergence},
 we simply write $T^*$ and $T^*_\delta$ for $T(M,y_0)$ and $T_\delta(M,y_0)$, respectively; simply write $u^*$ and $u^*_\delta$ for $u_M^*$ and $u^*_{\delta,M}$, respectively.}

 Since $T^*_\delta$ is a multiple of $\delta$ (see (\ref{time-2})), we can write
 \begin{eqnarray}\label{yuanhuangyuan6.9}
 T^*_\delta\triangleq k_\delta \delta\;\;{with}\;\;k_\delta\in \mathbb{N}^+.
 \end{eqnarray}

 In the case that
\begin{eqnarray}\label{WWGSS6.34}
 \delta\geq \min\Big\{ \delta_0,\frac{T^*_{y_0}-T(M_1,y_0)}{4},  \frac{T(M_2,y_0)}{3},\frac{M_1}{4 e^{C_3(1+\frac{1}{T(M_2,y_0)})} (1+\|y_0\|)} \Big\},
\end{eqnarray}
one can easily show (\ref{TP-control-converge-0}) and the desired continuity of
$C(M_1,M_2,y_0,r)$.
In fact,  it follows
 from (\ref{WWGSS6.34})  that
\begin{eqnarray}\label{WWGSS6.35}
  \|u^* - u_\delta^*\|_{L^2(\mathbb R^+;L^2(\Omega))}
 &\leq& \|u^* \|_{L^2(\mathbb R^+;L^2(\Omega))}
 + \| u_\delta^*\|_{L^2(\mathbb R^+;L^2(\Omega))}
 \nonumber\\
 &\leq& 2M\leq 2M_2\leq \frac{6M_2}{T(M_2,y_0)} \delta
\triangleq \hat C(M_1,M_2,y_0,r)\delta.
\end{eqnarray}
Since   $(M_2,y_0)\rightarrow T(M_2,y_0)$ is continuous (see the proof of the continuity of $\delta_0$ with respect to $(M_1,M_2,y_0,r)$ in   Theorem \ref{Theorem-time-order}),  and because
$T(M_2,y_0)>0$ for all $M_2>0$ and $y_0\in L^2(\Omega)\setminus B_r(0)$,  we see
from (\ref{WWGSS6.35})  that $\hat C(M_1,M_2,y_0,r)$  is continuous in $M_1$, $M_2$, $y_0$ and $r$.

Thus, we only need to show (\ref{TP-control-converge-0}), as well as the continuous dependence of $C$ in  (\ref{TP-control-converge-0}), for  the case that
\begin{eqnarray}\label{0322-th1.4-delta}
  0<\delta< \min\Big\{ \delta_0, \frac{T^*_{y_0}-T(M_1,y_0)}{4},  \frac{T(M_2,y_0)}{3},\frac{M_1}{4 e^{C_3(1+\frac{1}{T(M_2,y_0)})} (1+\|y_0\|)} \Big\}.
 \end{eqnarray}
   For this purpose, some preliminaries are needed. First we claim that
 \begin{eqnarray}\label{0322-T-Tdelta-range}
  0< T(M_2,y_0)\leq T^* \leq T_\delta^*
  \leq {(T(M_1,y_0)+T^*_{y_0})}/{2}
  <T^*_{y_0}.
 \end{eqnarray}
 Indeed, (\ref{0322-T-Tdelta-range}) follows from the next three facts at once.
 Fact one: From Theorem \ref{Theorem-eq-cont-discrete}, we have  that $0<T(M_2,y_0)\leq T^*\leq T(M_1,y_0)<T^*_{y_0}$;
 Fact two:
 Since $\mathcal U^M_\delta\subset \mathcal U^M$,  we find from (\ref{time-1}) and (\ref{time-2})  that $T^*\leq T_\delta^*$;
 Fact three: By Theorem \ref{Theorem-time-order} and (\ref{0322-th1.4-delta}), we see that
 \begin{eqnarray*}
  T_\delta^* \leq T^*+2\delta \leq T(M_1,y_0)+2\delta\leq {(T(M_1,y_0)+T^*_{y_0})}/{2}.
 \end{eqnarray*}

   Then it follows from  (\ref{0322-T-Tdelta-range}), (\ref{0322-th1.4-delta}) and (\ref{P-T*}) that
 \begin{eqnarray}\label{0322-Tdelta-P}
  3\delta\leq T^*_\delta<T^*_{y_0},
  \;\;\mbox{i.e.},\;\;
  (\delta,T^*_\delta/\delta)\triangleq(\delta,k_\delta)\in \mathcal P_{T^*_{y_0}},\;\;\mbox{with}\;\;k_\delta\;\;\mbox{given by}\;\;(\ref{yuanhuangyuan6.9}).
 \end{eqnarray}

   Next, we let $z^*\neq 0$  and $z^*_{\delta}\neq 0$ be the minimizers of $J^{T^*,y_0}$ and $J_{\delta}^{T^*_\delta,y_0}$, respectively (see Theorem \ref{Lemma-Norm-fun} and Theorem \ref{Lemma-NP-bangbang}).
 Write
 \begin{eqnarray}\label{z-normalization}
  \hat z^*\triangleq z^*/\|z^*\|
  \;\;\mbox{and}\;\;
  \hat z^*_\delta\triangleq z^*_\delta/\|z^*_\delta\|.
 \end{eqnarray}
 By (iii) of Theorem \ref{Theorem-eq-cont-discrete}, $u^*|_{(0,T^*)}$ is the optimal control to $(NP)^{T^*,y_0}$. Then by (ii) of Theorem \ref{Lemma-Norm-fun} (with $T=T^*\in (0,T^*_{y_0})$),
 we see that   $u^*(\cdot)=\chi_\omega\varphi(\cdot;T^*,z^*)$ over $(0,T^*)$. Meanwhile, by (ii) of
 Theorem \ref{Theorem-eq-cont-discrete}, we find that $N(T^*,y_0)=M$. These, along with
 (\ref{z-normalization}), yield that
   \begin{eqnarray}\label{th2-2}
  u^*(t) =  \chi_\omega\varphi(t;T^*,z^*)
    = M \frac{\chi_\omega\varphi(t;T^*,\hat z^*)}{\|\chi_\omega\varphi(\cdot;T^*,\hat z^*)\|_{L^2(0,T^*;L^2(\Omega))}},~t\in (0,T^*).
 \end{eqnarray}

     Finally,  it follows from (iii)
 of Theorem \ref{Lemma-existences-TP} that  $u_\delta^*|_{(0, T^*_\delta)}$ is  an optimal control to $(NP)_\delta^{T^*_\delta,y_0}$. This, along with the fact that $u^*_\delta$ is an optimal control to $(TP)_\delta^{M,y_0}$, yields
 \begin{eqnarray}\label{yuanhen6.18}
 N_\delta(T^*_\delta,y_0)=\|u^*_\delta\|_{L^2(0,T^*_\delta;L^2(\Omega))}
 \leq \|u^*_\delta\|_{L^2(\mathbb{R}^+;L^2(\Omega))}\leq M.
 \end{eqnarray}
  Meanwhile,
  by (\ref{0322-Tdelta-P}), we can apply   Theorem  \ref{Lemma-NP-bangbang}
  (with $(\delta,k)=(\delta, T^*_\delta/\delta)\triangleq(\delta,k_\delta)$), as well as (\ref{z-normalization}), to obtain that
    \begin{eqnarray}\label{0322-th1.2-udelta}
  u^*_{\delta}(t) =
  \chi_\omega \overline\varphi_\delta(t;T^*_\delta,z_\delta^*)
  = M_\delta \frac{\chi_\omega\overline\varphi_\delta(t;T^*_\delta, \hat z_\delta^*)}
  {\|\chi_\omega \overline\varphi_\delta(\cdot;T^*_\delta, \hat z_\delta^*)\|_{L^2(0,T^*_\delta;L^2(\Omega))}}
  ,~t\in (0,T^*_\delta).
 \end{eqnarray}
 Here, $\overline\varphi_\delta(\cdot;T^*_\delta,z_\delta^*)$
 and $\overline\varphi_\delta(\cdot;T^*_\delta,\hat z_\delta^*)$
 are given by (\ref{adjoint-solution-P}) with $(\delta,k)=(\delta,k_\delta)$ and
 $M_\delta$ is defined by
 \begin{eqnarray}\label{0322-th1.4-s1-ii}
  M_\delta \triangleq N_\delta(T^*_\delta,y_0).
 \end{eqnarray}

We now prove
 (\ref{TP-control-converge-0}) for an arbitrarily fixed $\delta$ (satisfying (\ref{0322-th1.4-delta})) by  several steps.

 \textit{Step 1. To show that
 \begin{eqnarray}\label{0322-appendix-4.44}
 \|y(T^*;y_0,u_\delta^*)\|_{H^2(\Omega)\cap H_0^1(\Omega)}
 \leq \|y_0\|/T^* + 2 \left( \sqrt{T^*_\delta}\|z^*_\delta\|_{H_0^1(\Omega)}+\|z^*_\delta\|
 \right)
\end{eqnarray}}
$\;\;$ One can easily check the following two estimates:
\begin{eqnarray*}
\|y(T^*;y_0,u_\delta^*)\|_{H^2(\Omega)\cap H_0^1(\Omega)}
\leq  \|e^{\Delta T^*} y_0\|_{H^2(\Omega)\cap H_0^1(\Omega)}
+ \|y(T^*;0,u_\delta^*)\|_{H^2(\Omega)\cap H_0^1(\Omega)};
\end{eqnarray*}
\begin{eqnarray*}
  \|e^{\Delta T^*} y_0\|_{H^2(\Omega)\cap H_0^1(\Omega)}
  = \| \Delta  e^{\Delta T^*} y_0\|
  \leq \|y_0\|/T^*.
\end{eqnarray*}
From these, we see that to prove (\ref{0322-appendix-4.44}), it suffices to show that
\begin{eqnarray}\label{appendix-p1-1}
 \|y(T^*;0,u_\delta^*)\|_{H^2(\Omega)\cap H_0^1(\Omega)}
 \leq  2\big(\sqrt{T^*_\delta}\|z^*_\delta\|_{H_0^1(\Omega)}
 +\|z^*_\delta\|\big).
\end{eqnarray}
For this purpose, let $\bar k$ be the integer so that
$\bar k \delta< T^* \leq (\bar k +1 )\delta$. Because $T^*\leq T^*_\delta$ (see (\ref{0322-T-Tdelta-range})) and $T^*_\delta$ is a multiple of $\delta$ (see (\ref{time-2})), we find that $T^*_\delta\geq (\bar k+1)\delta$.
Since $u^*_\delta$ is a piece-wise constant function over $(0,T^*_\delta)$,
and because
\begin{eqnarray*}
\Delta e^{\Delta(T^*-t)}fdt=-\frac{d}{dt}(e^{\Delta(T^*-t)}f)\;\;\mbox{for each}\;\;
 f\in L^2(\Omega),
 \end{eqnarray*}
  one can easily check that
\begin{eqnarray*}
\Delta  y(T^*;0,u_\delta^*)
&=& \Delta \int_{\bar k \delta}^{T^*}  e^{\Delta(T^*-t)} \chi_\omega u_\delta^*((\bar k+1)\delta) \,\mathrm dt
 +  \sum_{j=1}^{\bar k}  \Delta \int_{(j-1) \delta}^{j \delta}  e^{\Delta(T^*-t)} \chi_\omega u^*_\delta(j\delta) \,\mathrm dt
 \nonumber\\
&=& \sum_{j=1}^{\bar k}  e^{\Delta(T^*-j \delta)}
\chi_\omega\big( u^*_\delta((j+1) \delta) - u^*_\delta(j \delta) \big)
+ e^{\Delta T^*} \chi_\omega u^*_\delta(\delta) -\chi_\omega u^*_\delta((\bar k+1)\delta).
\end{eqnarray*}
This yields that
\begin{eqnarray}\label{appendix-p1-2}
 & & \|y(T^*;0,u_\delta^*)\|_{H^2(\Omega)\cap H_0^1(\Omega)}
 = \| \Delta  y(T^*;0,u_\delta^*) \|
 \nonumber\\
 &\leq& \sum_{j=1}^{\bar k} \| u^*_\delta((j+1) \delta) - u^*_\delta(j \delta)\|
 + \|u^*_\delta(\delta)\| + \|u^*_\delta((\bar k+1)\delta)\|.
\end{eqnarray}
Meanwhile,
from the first equality in (\ref{0322-th1.2-udelta}), one can easily verify that
when  $j=1,\dots,\bar k$,
\begin{eqnarray*}
 & & \| u^*_\delta((j+1) \delta) - u^*_\delta(j \delta)\|
  \nonumber\\
 &=& \Big\| \frac{1}{\delta} \int_{(j-1)\delta}^{j\delta} \chi_\omega\varphi(s+\delta;T^*_\delta,z^*_\delta) \,\mathrm ds
 - \frac{1}{\delta} \int_{(j-1)\delta}^{j\delta} \chi_\omega\varphi(s;T^*_\delta,z^*_\delta) \,\mathrm ds\Big\|
 \nonumber\\
 &=& \Big\| \chi_\omega \frac{1}{\delta} \int_{(j-1)\delta}^{j\delta}
 \int_s^{s+\delta} \partial_\tau \varphi(\tau;T^*_\delta,z^*_\delta) \,\mathrm d\tau \mathrm ds  \Big\|\leq \int_{(j-1)\delta}^{(j+1)\delta}
 \| \partial_\tau \varphi(\tau;T^*_\delta,z^*_\delta) \| \,\mathrm d\tau.
\end{eqnarray*}
This, along with (\ref{appendix-p1-2}), (\ref{0322-th1.2-udelta}),
(\ref{0322-th1.4-s1-ii}) and the contractivity of $\{e^{\Delta t}\}_{t\geq 0}$, yields that
\begin{eqnarray}\label{huangyuan2.62}
  \|y(T^*;0,u_\delta^*)\|_{H^2(\Omega)\cap H_0^1(\Omega)}
  \leq 2 \int_0^{T^*_\delta}  \| \partial_\tau \varphi(\tau;T^*_\delta,z^*_\delta) \| \,\mathrm d\tau + 2\|z^*_\delta\|.
\end{eqnarray}
Since $\|\partial_\tau \varphi(\cdot;T^*_\delta,z^*_\delta)\|_{L^2(0,T^*_\delta;L^2(\Omega))} \leq \|z^*_\delta\|_{H_0^1(\Omega)}$ (see (\ref{z-beta-uni-bdd})),
by applying  the H\"{o}lder inequality to (\ref{huangyuan2.62}), we obtain (\ref{appendix-p1-1}). This ends the proof of (\ref{0322-appendix-4.44}).

\textit{Step 2. To show that
\begin{eqnarray}\label{0322-error-varphi}
  \|\varphi(\cdot;T^*,z_\delta^*) -\overline\varphi_\delta(\cdot;T^*_\delta,z_\delta^*)\|_{L^2(0,T^*;L^2(\Omega))}
  \leq \|z^*_\delta\|_{H_0^1(\Omega)} (|T_\delta^*-T^*|+\delta)
 \end{eqnarray}}
 $\;\;$ Observe that
 \begin{eqnarray*}
  & & \|\varphi(\cdot;T^*,z_\delta^*) -\overline\varphi_\delta(\cdot;T^*_\delta,z_\delta^*)\|_{L^2(0,T^*;L^2(\Omega))}
  \nonumber\\
 &\leq& \|\varphi(\cdot;T^*,z_\delta^*) -\varphi(\cdot;T^*_\delta,z_\delta^*)\|_{L^2(0,T^*;L^2(\Omega))}
 + \|\varphi(\cdot;T^*_\delta,z_\delta^*) -\overline\varphi_\delta(\cdot;T^*_\delta,z_\delta^*)\|_{L^2(0,T^*;L^2(\Omega))}
 \nonumber\\
 &\triangleq& I_1 + I_2.
 \end{eqnarray*}
 We first  claim that
 \begin{eqnarray}\label{1009-1order-4}
   I_1
   \leq \| z_\delta^* \|_{H_0^1(\Omega)} (T^*_\delta-T^*).
 \end{eqnarray}
 Write $\{\lambda_j\}_{j=1}^\infty$ for the family of all eigenvalues of $-\Delta$ with the zero Dirichlet boundary condition so that $\lambda_1<\lambda_2\leq\cdots$. Let  $\{e_j\}_{j=1}^\infty$ be the family of  the corresponding normalized eigenvectors.  Arbitrarily fix $f\in L^2(0,T^*;L^2(\Omega))$. Write $f(t)=\sum_{j=1}^\infty f_j(t) e_j$, $t\in(0,T^*)$. Then we have that
 \begin{eqnarray}\label{1009-1order-5}
  \|f\|_{L^2(0,T^*;L^2(\Omega))}^2  = \int_0^{T^*} \sum_{j=1}^\infty |f_j(t)|^2  \,\mathrm dt<\infty.
 \end{eqnarray}
  By some computations, we obtain that
 \begin{eqnarray}\label{1009-1order-3}
   & & \Big\| \int_0^{T^*} e^{\Delta (T^*-t)}f(t) \,\mathrm dt  -  \int_0^{T^*} e^{\Delta (T^*_\delta-t)}f(t) \,\mathrm dt  \Big\|_{H^{-1}(\Omega)}^2
   \nonumber\\
   &=& \sum_{j=1}^\infty  \frac{1}{\lambda_j} \Big[ \int_0^{T^*} e^{-\lambda_j (T^*-t)}f_j(t) \,\mathrm dt  -  \int_0^{T^*} e^{-\lambda_j (T^*_\delta-t)}f_j(t) \,\mathrm dt
   \Big]^2
   \nonumber\\
   &=& \sum_{j=1}^\infty  \frac{1}{\lambda_j} \Big[ (e^{-\lambda_j (T^*_\delta-T^*)}-1)  \int_0^{T^*} e^{-\lambda_j (T^*_\delta-t)}f_j(t) \,\mathrm dt
   \Big]^2.
 \end{eqnarray}
 Two facts are given in order: (a) It is clear  that
 \begin{eqnarray}\label{1009-1order-1}
  |e^{-\lambda_j (T^*_\delta-T^*)}-1 |
  = |\int_0^{T^*_\delta-T^*} \lambda_j e^{-\lambda_j t} \,\mathrm dt|
   \leq \lambda_j (T^*_\delta-T^*),~j=1,2,\dots;
 \end{eqnarray}
 (b) By the H\"{o}lder inequality, we obtain that
 \begin{eqnarray}\label{1009-1order-2}
  \int_0^{T^*} e^{-\lambda_j (T^*_\delta-t)}|f_j(t)| \,\mathrm dt
  &\leq& \Big[ \int_0^{T^*} e^{-2 \lambda_j (T^*_\delta-t)}\,\mathrm dt \Big]^{1/2}
    \Big[ \int_0^{T^*} |f_j(t)|^2 \,\mathrm dt \Big]^{1/2}
    \nonumber\\
    &\leq& \frac{1}{\sqrt{\lambda_j}} \Big[ \int_0^{T^*} |f_j(t)|^2 \,\mathrm dt \Big]^{1/2},
    ~j=1,2,\dots.
 \end{eqnarray}
 The above facts (\ref{1009-1order-1}) and (\ref{1009-1order-2}), along with (\ref{1009-1order-3}) and (\ref{1009-1order-5}), yield that
 \begin{eqnarray*}
   & & \Big\| \int_0^{T^*} e^{\Delta (T^*-t)}f(t) \,\mathrm dt  -  \int_0^{T^*} e^{\Delta (T^*_\delta-t)}f(t) \,\mathrm dt  \Big\|_{H^{-1}(\Omega)}^2
   \nonumber\\
   &\leq& \sum_{j=1}^\infty  \frac{1}{\lambda_j} \Big[\lambda_j (T^*_\delta-T^*)\Big]^2
   \Big[ \frac{1}{\lambda_j} \Big[ \int_0^{T^*} |f_j(t)|^2 \,\mathrm dt \Big]
   \nonumber\\
   &=& (T^*_\delta-T^*)^2  \int_0^{T^*}   \sum_{j=1}^\infty |f_j(t)|^2 \,\mathrm dt
   = (T^*_\delta-T^*)^2  \|f\|_{L^2(0,T^*;L^2(\Omega))}^2.
 \end{eqnarray*}
 This implies that
 \begin{eqnarray*}
  & & \int_0^{T^*} \langle  \varphi(t;T^*,z_\delta^*) -\varphi(t;T^*_\delta,z_\delta^*),
  f(t) \rangle \,\mathrm dt
  \nonumber\\
  &=& \Big\langle z_\delta^*, \int_0^{T^*} e^{\Delta (T^*-t)}f(t) \,\mathrm dt  -  \int_0^{T^*} e^{\Delta (T^*_\delta-t)}f(t) \,\mathrm dt
  \Big\rangle
    \nonumber\\
  %&\leq& \| z_\delta^* \|_{H_0^1(\Omega)} \Big\| (e^{\Delta(T^*_\delta-T^*)} - Id) \int_0^{T^*} e^{\Delta (T^*-t)}f(t) \,\mathrm dt
%  \Big\|_{H^{-1}(\Omega)}
%  \nonumber\\
  &\leq& \| z_\delta^* \|_{H_0^1(\Omega)}  (T^*_\delta-T^*)  \|f\|_{L^2(0,T^*;L^2(\Omega))}.
 \end{eqnarray*}
 Since $f$ was arbitrarily taken from $L^2(0,T^*;L^2(\Omega))$, the above indicates that
  \begin{eqnarray*}
  I_1=\|\varphi(\cdot;T^*,z_\delta^*) -\varphi(\cdot;T^*_\delta,z_\delta^*)\|_{L^2(0,T^*;L^2(\Omega))}
   \leq \| z_\delta^* \|_{H_0^1(\Omega)} (T^*_\delta-T^*),
 \end{eqnarray*}
 which leads to (\ref{1009-1order-4}).

 We next  estimate  $I_2$. Since $T^*_\delta=k_\delta \delta$ (see (\ref{yuanhuangyuan6.9})) and because
   $T^*_\delta\geq T^*$,  we see from (\ref{adjoint-solution-P}) that
 \begin{eqnarray*}
 & & I_2^2 \leq \int_0^{T^*_\delta} \| \varphi(\cdot;T^*_\delta,z_\delta^*) -\overline\varphi_\delta(\cdot;T^*_\delta,z_\delta^*) \|^2 \,\mathrm dt
 \nonumber\\
 &=&  \sum_{j=1}^{k_\delta}
 \int_{(j-1)\delta}^{j\delta} \Big\|\varphi(t;T^*_\delta,z^*_\delta)- \frac{1}{\delta} \int_{(j-1)\delta}^{j\delta} \varphi(s;T^*_\delta,z^*_\delta) \,\mathrm ds \Big\|^2 \,\mathrm dt
  \nonumber\\
 &\leq& \sum_{j=1}^{k_\delta}
 \int_{(j-1)\delta}^{j\delta} \left( \int_{(j-1)\delta}^{j\delta}\| \partial_\tau  \varphi(\tau;T^*_\delta,z^*_\delta) \| \,\mathrm d\tau \right)^2  \, \mathrm dt.
 \end{eqnarray*}
 By using the H\"{o}lder inequality in the above and by  (\ref{z-beta-uni-bdd}), we see that
 \begin{eqnarray*}
   I_2 \leq \|\partial_\tau  \varphi(\cdot;T^*_\delta,z^*_\delta)\|_{L^2(0,T^*_\delta;L^2(\Omega))} \delta \leq \| z_\delta^* \|_{H_0^1(\Omega)} \delta.
 \end{eqnarray*}
  Finally, (\ref{0322-error-varphi}) follows from the above estimates on $I_1$ and $I_2$.

 \textit{Step 3. To prove that
 \begin{eqnarray}\label{0322-th1.4-s1-i}
    |T^*-T^*_\delta| + |M-M_\delta|
   \leq  2e^{C_3(1+\frac{1}{T(M_2,y_0)})} (1+\|y_0\|) \delta
  \triangleq C_1(M_2,y_0) \delta,
 \end{eqnarray}
 where   $C_1(M_2,y_0)$ is continuous in   $M_2$ and $y_0$ }

By (\ref{0322-Tdelta-P}), we can use Theorem~\ref{theorem-NP-delta-error}, with $(\delta,k)=
(\delta, k_\delta)$ (where $k_\delta$ is given by (\ref{yuanhuangyuan6.9})),  to see that
$N_\delta(T^*_\delta,y_0)\geq N(T^*_\delta,y_0)$. By (\ref{yuanhen6.18})
and (\ref{0322-th1.4-s1-ii}), we find that $M\leq M_\delta$. These, along with  (ii) of Theorem \ref{Theorem-eq-cont-discrete}, yield that
\begin{eqnarray}\label{yuanhen6.22}
 0\leq M-M_\delta=N(T(M,y_0),y_0)-N_\delta(T^*_\delta,y_0)
 \leq
 N(T(M,y_0),y_0)-N(T^*_\delta,y_0).
\end{eqnarray}
Meanwhile, by (\ref{0322-T-Tdelta-range}), we can use Theorem~\ref{Proposition-NP-Lip-T} (with $T_1=T^*$ and $T_2=T^*_\delta$) to see that
\begin{eqnarray*}
N(T^*,y_0)-N(T^*_\delta,y_0)\leq e^{C_3(\Omega,\omega)(1+1/T^*)}\|y_0\|(T^*_\delta-T^*).
\end{eqnarray*}
where $C_3(\Omega,\omega)$ is given by (\ref{NP-Lip-T}).
Since $T(M_2,y_0)\leq T^*$ (see (\ref{0322-T-Tdelta-range})), the above, along with (\ref{yuanhen6.22}), yields that
 \begin{eqnarray*}\label{160830-th1.2-1}
 |M-M_\delta| &\leq&   e^{C_3(\Omega,\omega)(1+{1}/{T(M_2,y_0)})} \|y_0\| |T^*_\delta-T^*|.
 \end{eqnarray*}
 Since $\delta\in(0,\delta_0)$, the above, along with Theorem \ref{Theorem-time-order}, leads to (\ref{0322-th1.4-s1-i}).
Finally, since  $T(M_2,y_0)$ is continuous in $M_2$ and $y_0$ (see the proof of the continuity of $\delta_0$ with respect to $(M_1,M_2,y_0,r)$ in   Theorem \ref{Theorem-time-order}), we obtain from (\ref{0322-th1.4-s1-i}) that $C_1(M_2,y_0)$ is continuous in   $M_2$ and $y_0$.

 \textit{Step 4. To show that
 \begin{eqnarray}\label{0213-order-control-step1}
  \| \hat z^* - \hat z^*_\delta \|  \leq  C_2(M_1,M_2,y_0,r) \delta,
 \end{eqnarray}
 where  $C_2(M_1,M_2,y_0,r)$ is continuous in  $M_1$, $M_2$, $y_0$  and $r$ }

  Define an affiliated control $\hat u_\delta$ over $\mathbb{R}^+$ by
 \begin{eqnarray}\label{0213-order-control-01}
  \hat u_\delta(t) \triangleq  M \frac{\chi_\omega\varphi(t;T^*,\hat z_\delta^*)}{\|\chi_\omega\varphi(\cdot;T^*,\hat z_\delta^*)\|_{L^2(0,T^*;L^2(\Omega))}}~t\in(0,T^*);\; \hat u_\delta(t) \triangleq 0, \; t\in (T^*,\infty).
 \end{eqnarray}
 We divide the rest of the proof of Step 4 by several parts.

 \textit{Part 4.1. To prove that
 \begin{eqnarray}\label{0213-order-control-part1.1}
  \langle \hat z^* - \hat z^*_\delta, \hat z^* - \hat z^*_\delta \rangle
  \leq   -\frac{1}{r}  \big\langle \hat z^* - \hat z^*_\delta, y(T^*;y_0,\hat u_\delta) - y(T^*_\delta;y_0,u_\delta^*) \big\rangle
 \end{eqnarray}}
$\;\;$ By  (\ref{th2-2}) and (\ref{0213-order-control-01}), one can directly  check that
\begin{eqnarray*}
0\leq \langle \chi_\omega \varphi(t;T^*,\hat z^*)-\chi_\omega\varphi(t;T^*,\hat z^*_\delta), u^*(t)-\hat u_\delta(t) \big\rangle\;\;\mbox{for a.e.}\;\; t\in (0,T^*).
\end{eqnarray*}
Hence, we have that
  \begin{eqnarray*}\label{0213-order-control-00}
   0 &\leq & \big\langle \chi_\omega \varphi(\cdot;T^*,\hat z^*)- \chi_\omega\varphi(\cdot;T^*,\hat z^*_\delta), u^*(\cdot)-\hat u_\delta(\cdot) \big\rangle_{L^2(0,T^*;L^2(\Omega))}
   \nonumber\\
   &=&
   \big\langle \hat z^* - \hat z^*_\delta,  y(T^*;y_0,u^*) - y(T^*;y_0,\hat u_\delta)  \big\rangle.
    \end{eqnarray*}
 This, along with (\ref{op-final-dual}), (\ref{op-control-final-state}) and (\ref{z-normalization}), yields that
 \begin{eqnarray*}
  & & \langle \hat z^* - \hat z^*_\delta, \hat z^* - \hat z^*_\delta \rangle
  =    \Big\langle \hat z^* - \hat z^*_\delta, -\frac{1}{r} \big( y(T^*;y_0,u^*) - y(T^*_\delta;y_0,u_\delta^*) \big) \Big\rangle
    \nonumber\\
  &\leq& -\frac{1}{r} \big\langle \hat z^* - \hat z^*_\delta,   y(T^*;y_0,\hat u_\delta) - y(T^*_\delta;y_0,u_\delta^*)  \big\rangle,
 \end{eqnarray*}
 which leads to (\ref{0213-order-control-part1.1}).

\textit{Part 4.2. To show that there exists $C_{21}\triangleq C_{21}(\Omega)>0$ so that
 \begin{eqnarray}\label{0213-order-control-part1.2}
  \|y(T^*;y_0,\hat u_\delta) - y(T^*_\delta;y_0,u_\delta^*)\|
  &\leq& C_{21} \Big[ \big(1+\frac{1}{T(M_2,y_0)}+\sqrt{T^*_{y_0}}\big)
  (\|y_0\|+\|z^*_\delta\|_{H_0^1(\Omega)})
  \nonumber\\
  & & \times (T^*_\delta-T^*)  +  \|\hat u_\delta-u_\delta^*\|_{L^2(0,T^*;L^2(\Omega))}
  \Big]
 \end{eqnarray}}
$\;\;$  Three facts are given in order. Fact one: By  the H\"{o}lder inequality, we find that for some $C_{22}\triangleq C_{22}(\Omega)>0$,
 \begin{eqnarray}\label{Wang6.49}
  & &  \|y(T^*;y_0,\hat u_\delta) - y(T^*;y_0,u_\delta^*)\|
   = \big\| \int_0^{T^*} e^{\Delta(T^*-t)} \chi_\omega(\hat u_\delta-u_\delta^*)(t) \,\mathrm dt \big\|
   \nonumber\\
   &\leq& \int_0^{T^*} e^{-\lambda_1(T^*-t)} \|(\hat u_\delta-u_\delta^*)(t) \|\,\mathrm dt
   \leq C_{22} \|\hat u_\delta-u_\delta^*\|_{L^2(0,T^*;L^2(\Omega))};
 \end{eqnarray}
 Fact two: Since  $\|u^*_\delta\|_{L^\infty(0,T^*_\delta;L^2(\Omega))} \leq \|z^*_\delta\|$ (which follows from (\ref{0322-th1.2-udelta}) and (\ref{adjoint-solution-P})), and because $T^*\leq T^*_\delta$ (see
(\ref{0322-T-Tdelta-range})),   we find that
\begin{eqnarray}\label{Wang6.51}
   & & \|y(T^*;y_0,u_\delta^*) - y(T^*_\delta;y_0,u_\delta^*)\|
   \nonumber\\
   &\leq& \| y(T^*;y_0,u_\delta^*)- e^{\Delta (T^*_\delta-T^*)}  y(T^*;y_0,u_\delta^*)\|  + \Big\| \int_{T^*}^{T^*_\delta} e^{\Delta(T^*_\delta-t)} \chi_\omega u_\delta^*(t) \,\mathrm dt \Big\|
   \nonumber\\
   &\leq& (T^*_\delta-T^*) \left[  \|y(T^*;y_0,u_\delta^*)\|_{H^2(\Omega)\cap H_0^1(\Omega)}
    +  \|z_\delta^*\| \right].
 \end{eqnarray}
Fact three: By (\ref{0322-T-Tdelta-range}), we see that
\begin{eqnarray}\label{Wang6.52}
T(M_2,y_0)\leq T^*\;\;\mbox{and}\;\; T^*_\delta<T^*_{y_0}.
\end{eqnarray}
Now, by the triangle inequality, (\ref{Wang6.49}), (\ref{0322-appendix-4.44}),
 (\ref{Wang6.51}),  (\ref{Wang6.52}) and the Poincar\'{e} inequality, we obtain
(\ref{0213-order-control-part1.2}) for some $C_{21}=C_{21}(\Omega)$.

\textit{Part 4.3. To show that
there exists $C_{23}\triangleq C_{23}(\Omega)>1$ so that
 \begin{eqnarray}\label{0213-order-control-part1.3-1}
  \|\hat u_\delta-u_\delta^*\|_{L^2(0,T^*;L^2(\Omega))}
  \leq C_{23} \big(1+\|z^*_\delta\|_{H_0^1(\Omega)}^4 \big)
  \big(1+\frac{1}{M_1}\big)
  \big( |T_\delta^*-T^*| +\delta
  + |M-M_\delta| \big)
 \end{eqnarray}}
$\;\;$  Recall (\ref{z-normalization}) for the definition of $\hat z^*_\delta$.
{\it In Part 4.3,   we simply  write respectively $\varphi(\cdot)$ and $\overline\varphi_\delta(\cdot)$ for $\varphi(\cdot;T^*, \hat z^*_\delta)$ and $\overline\varphi_\delta(\cdot;T^*_\delta, \hat z^*_\delta)$; simply write $L^2(0,T^*)$,
 $L^2(0,T^*_\delta)$ and $L^2(T^*,T^*_\delta)$ for $L^2(0,T^*;L^2(\Omega))$, $L^2(0,T^*_\delta;L^2(\Omega))$ and $L^2(T^*,T^*_\delta;L^2(\Omega))$ respectively.}
 From  (\ref{0213-order-control-01}) and (\ref{0322-th1.2-udelta}), using the triangle inequality, we obtain that
 \begin{eqnarray}\label{0213-order-control-04}
  & & \|\hat u_\delta-u_\delta^*\|_{L^2(0,T^*)}
  =
  \Big\|M \frac{\chi_\omega\varphi}{\|\chi_\omega\varphi\|_{L^2(0,T^*)}}
   - M_\delta \frac{\chi_\omega\overline\varphi_\delta}
   {\|\chi_\omega\overline\varphi_\delta\|_{L^2(0,T^*_\delta)}} \Big\|_{L^2(0,T^*)}
  \nonumber\\
  &\leq& M_\delta  \Big\|\frac{\chi_\omega\varphi}{\|\chi_\omega\varphi\|_{L^2(0,T^*)}}
  - \frac{\chi_\omega\overline\varphi_\delta}
  {\|\chi_\omega\overline\varphi_\delta\|_{L^2(0,T^*_\delta)}} \Big\|_{L^2(0,T^*)}
  + |M-M_\delta|.~~~
 \end{eqnarray}
 By direct computations, we find that
 \begin{eqnarray}\label{wanggs6.55}
   & & \frac{\chi_\omega\varphi}{\|\chi_\omega\varphi\|_{L^2(0,T^*)}}
    - \frac{\chi_\omega\overline\varphi_\delta}
    {\|\chi_\omega\overline\varphi_\delta\|_{L^2(0,T^*_\delta)}}
   \\
  &=& \frac{\chi_\omega\varphi
  (\|\chi_\omega\overline\varphi_\delta\|_{L^2(0,T^*_\delta)}- \|\chi_\omega\varphi\|_{L^2(0,T^*)}) } {\|\chi_\omega\varphi\|_{L^2(0,T^*)}\|
  \chi_\omega\overline\varphi_\delta\|_{L^2(0,T^*_\delta)}}
   + \frac{\chi_\omega\varphi- \chi_\omega\overline\varphi_\delta  }
  {\|\chi_\omega\overline\varphi_\delta\|_{L^2(0,T^*_\delta)}};
  \nonumber
 \end{eqnarray}
  \begin{eqnarray}\label{wanggs6.56}
  & & \big| \|\chi_\omega\overline\varphi_\delta\|_{L^2(0,T^*_\delta)}- \|\chi_\omega\varphi\|_{L^2(0,T^*)} \big|\nonumber\\
      &\leq&  \frac{ \big|\|\chi_\omega\overline\varphi_\delta\|^2_{L^2(0,T^*)} -\|\chi_\omega\varphi\|^2_{L^2(0,T^*)}\big|
  +\|\chi_\omega\overline\varphi_\delta\|^2_{L^2(T^*,T^*_\delta)}}
  {\|\chi_\omega\overline\varphi_\delta\|_{L^2(0,T^*_\delta)} +  \|\chi_\omega\varphi\|_{L^2(0,T^*)}}
  \nonumber\\
  &\leq& \|\chi_\omega\varphi -\chi_\omega\overline\varphi_\delta\|_{L^2(0,T^*)}
  + \frac{\|\chi_\omega\overline\varphi_\delta\|^2_{L^2(T^*,T^*_\delta)}}
  {\|\chi_\omega\overline\varphi_\delta\|_{L^2(0,T^*_\delta)}}.
 \end{eqnarray}
 From  (\ref{0213-order-control-04}), (\ref{wanggs6.55}) and (\ref{wanggs6.56}), we deduce that
 \begin{eqnarray}\label{0213-order-control-06}
   \|\hat u_\delta-u_\delta^*\|_{L^2(0,T^*)}
  &\leq& \frac{M_\delta}{\|\chi_\omega\overline\varphi_\delta\|_{L^2(0,T^*_\delta)}} \Big[ 2\|\chi_\omega\varphi -\chi_\omega\overline\varphi_\delta\|_{L^2(0,T^*)}
  + \frac{ \|\chi_\omega\overline\varphi_\delta\|^2_{L^2(T^*,T^*_\delta)}}
  {\|\chi_\omega\overline\varphi_\delta\|_{L^2(0,T^*_\delta)}} \Big]\nonumber\\
  &&+ |M-M_\delta|.
 \end{eqnarray}
 Meanwhile, by (\ref{0322-th1.2-udelta}) and (\ref{z-normalization}), we see that
 $ M_\delta=
  \| z_\delta^*\| \|\chi_\omega\overline\varphi_\delta(\cdot)\|_{L^2(0,T^*_\delta)}$.
 This, together with  (\ref{0213-order-control-06}) and (\ref{z-normalization}), yields    that
 \begin{eqnarray}\label{0213-order-control-06-1}
   \|\hat u_\delta-u_\delta^*\|_{L^2(0,T^*)}
  &\leq& \| z_\delta^*\| \Big[2\|\varphi(\cdot;T^*,z_\delta^*) -\overline\varphi_\delta(\cdot;T^*_\delta,z_\delta^*)\|_{L^2(0,T^*)}
  \nonumber\\
  & & +\frac{\| z_\delta^*\|}{M_\delta}  \|\chi_\omega\overline\varphi_\delta(\cdot;T^*_\delta,z_\delta^*)\|^2_{L^2(T^*,T^*_\delta)}
  \Big]
  + |M-M_\delta|.
 \end{eqnarray}
    Since  $\|\overline\varphi_\delta(t;T^*_\delta,z_\delta^*)\| \leq \|z_\delta^*\|$ for each $t\in(0,T^*_\delta)$ (which follows from (\ref{adjoint-solution-P})), we find from (\ref{0213-order-control-06-1}) and (\ref{0322-error-varphi}) that
 \begin{eqnarray*}\label{0322-error-varphi-1}
   \|\hat u_\delta-u_\delta^*\|_{L^2(0,T^*;L^2(\Omega))}
  \leq  \left( 1+2\|z^*_\delta\|_{H_0^1(\Omega)}\|z^*_\delta\| + \frac{ \|z^*_\delta\|^4}{M_\delta} \right)\Big( |T_\delta^*-T^*| +\delta  + |M-M_\delta| \Big)
   \end{eqnarray*}
At the same time, since $M\geq M_1$, it follows from (\ref{0322-th1.4-s1-i}) and (\ref{0322-th1.4-delta}) that
 \begin{eqnarray*}
 M_\delta\geq M-2e^{C_3(1+\frac{1}{T(M_2,y_0)})} (1+\|y_0\|) \delta \geq M_1/2.
 \end{eqnarray*}
 The above two inequalities, along with  the Poincar\'{e} inequality, yield (\ref{0213-order-control-part1.3-1}).

 \textit{Part 4.4. To show (\ref{0213-order-control-step1}) and the continuity of $C_2(M_1,M_2,y_0,r)$}

 By (\ref{0213-order-control-part1.2}) and (\ref{0213-order-control-part1.3-1}), we can easily check that
 \begin{eqnarray}\label{0322-th1.2-p2.4-0}
  & &  \|y(T^*;y_0,\hat u_\delta) - y(T^*_\delta;y_0,u_\delta^*)\|
  \nonumber\\
  &\leq& C_{21} C_{23}  \left(1+\frac{1}{T(M_2,y_0)}+\sqrt{T^*_{y_0}}\right)
  \left( 1+\|y_0\|+\|z^*_\delta\|_{H_0^1(\Omega)}+ \|z^*_\delta\|_{H_0^1(\Omega)}^4  \right)
  \nonumber\\
  & & \times \big(1+\frac{1}{M_1}\big)  \times \big(|T^*_\delta-T^*|
  + |M_\delta-M|+\delta \big).
   \end{eqnarray}
 Meanwhile, by (\ref{0322-Tdelta-P}),  we can use Theorem~\ref{Lemma-minizer-P-L2-uni-bdd}
 (with $(\delta,k)=(\delta,k_\delta)$, where $k_\delta$ is given by (\ref{yuanhuangyuan6.9})),
 as well as  (\ref{0322-T-Tdelta-range}), to get that
   \begin{eqnarray}\label{160830-11}
   1+\|y_0\|+\|z^*_\delta\|_{H_0^1(\Omega)}+ \|z^*_\delta\|_{H_0^1(\Omega)}^4
     \leq 4 e^{C_1(1+T^*_{y_0}+\frac{1}{T(M_2,y_0)})} \|y_0\|^{24}r^{-20},
 \end{eqnarray}
 where $C_1=C_1(\Omega,\omega)$ is given by (\ref{z-beta-uni-bdd}). This, along with (\ref{0322-th1.2-p2.4-0}), leads to that
 \begin{eqnarray}\label{0322-th1.2-p2.4}
  \|y(T^*;y_0,\hat u_\delta) - y(T^*_\delta;y_0,u_\delta^*)\|
  \leq \hat C_3 \big(1+\frac{1}{M_1}\big) \big(|T^*_\delta-T^*|  + |M_\delta-M|+\delta \big),
 \end{eqnarray}
 where
 \begin{eqnarray*}
 \hat C_3\triangleq\hat C_3(M_1,M_2,y_0,r)\triangleq 4C_{21} C_{23}  \left(1+\frac{1}{T(M_2,y_0)}+\sqrt{T^*_{y_0}}\right)
   e^{C_1(1+T^*_{y_0}+\frac{1}{T(M_2,y_0)})} \|y_0\|^{24}r^{-20}.
 \end{eqnarray*}
Notice that $\hat C_3(M_1,M_2,y_0,r)$ is continuous in $M_1$, $M_2$, $y_0$ and $r$. This can be showed
by the same way as that used to prove  of the continuity of $\delta_0$ with respect to $(M_1,M_2,y_0,r)$ in  Theorem \ref{Theorem-time-order}.

Now it follows from (\ref{0213-order-control-part1.1}) that
 \begin{eqnarray*}
  \| \hat z^* - \hat z^*_\delta \|^2
  &\leq&   \frac{1}{2} \|\hat z^* - \hat z^*_\delta\|^2 + \frac{1}{2r^2}  \| y(T^*;y_0,\hat u_\delta) - y(T^*_\delta;y_0,u_\delta^*) \|^2.
   \end{eqnarray*}
 This, along with (\ref{0322-th1.2-p2.4}) and (\ref{0322-th1.4-s1-i}), yields (\ref{0213-order-control-step1}), with
  \begin{eqnarray*}
  C_2(M_1,M_2,y_0,r)\triangleq
    \frac{1}{r} \Big[ \hat C_3(M_1,M_2,y_0,r)\big(1+\frac{1}{M_1}\big) (C_1(M_2,y_0)+1) \Big].
 \end{eqnarray*}
 Finally, along with the continuity of $C_3(M_1,M_2,y_0,r)$ and $C_1(M_2,y_0)$, the above leads to the desired continuity of $C_2(M_1,M_2,y_0,r)$.
  This ends the proof of Step 4.

 \textit{Step 5. To show that
 \begin{eqnarray}\label{0219-error-control-Step2}
  \|u^* - u_\delta^*\|_{L^2(0,T^*;L^2(\Omega))}
  \leq C_4(M_1,M_2,y_0,r) \delta,
 \end{eqnarray}
 where  $C_4(M_1,M_2,y_0,r)$ is continuous in   $M_1$, $M_2$, $r$ and $y_0$}

Recall  (\ref{z-normalization}) for the definitions of $\hat z^*_\delta$ and $\hat z^*$.
{\it In Step 5, we simply  write $\varphi_1(\cdot)$ and $\varphi_2(\cdot)$ for $\varphi(\cdot;T^*,\hat z^*)$ and $\varphi(\cdot;T^*,\hat z^*_\delta)$, respectively;
simply write $L^2(0,T^*)$ for $L^2(0,T^*;L^2(\Omega))$}. By (\ref{th2-2}) and (\ref{0213-order-control-01}), we see that
\begin{eqnarray}\label{th1.3-pf-0809-1-1}
 \|u^* - \hat u_\delta\|_{L^2(0,T^*)}
  &\leq& M  \Big\|\frac{\chi_\omega\varphi_1}{\|\chi_\omega\varphi_1\|_{L^2(0,T^*)}} - \frac{\chi_\omega\varphi_2}{\|\chi_\omega\varphi_2\|_{L^2(0,T^*)}} \Big\|_{L^2(0,T^*)}
 \nonumber\\
 &\leq&  \frac{2M}
 {\|\chi_\omega\varphi_1\|_{L^2(0,T^*)}}\| \chi_\omega\varphi_1  - \chi_\omega\varphi_2   \|_{L^2(0,T^*)}.
\end{eqnarray}
Meanwhile,   from (\ref{th2-2}) and (\ref{z-normalization}), we find that
 \begin{eqnarray*}
 M=\| z^*\| \|\chi_\omega\varphi_1\|_{L^2(0,T^*;L^2(\Omega))}.
 \end{eqnarray*}
  This, along with (\ref{th1.3-pf-0809-1-1}), yields that
 \begin{eqnarray}\label{th1.3-pf-0809-1}
 \|u^* - \hat u_\delta\|_{L^2(0,T^*)}
  \leq 2\|z^*\|  \|\varphi_1-\varphi_2\|_{L^2(0,T^*)}
   \leq 2\|z^*\|  \|\hat z^*-\hat z^*_\delta\|.
 \end{eqnarray}
By (\ref{th1.3-pf-0809-1}), (\ref{0213-order-control-part1.3-1}) and the triangle inequality, then using (\ref{0213-order-control-step1}), we see that
\begin{eqnarray*}
 \|u^* - u_\delta^*\|_{L^2(0,T^*)}
&\leq&
2C_2(M_1,M_2,y_0,r)\|z^*\| \delta
\nonumber\\
& & + C_{23} (1+\|z^*_\delta\|_{H_0^1(\Omega)}^4)
  \big( |T_\delta^*-T^*| +\delta
  + |M-M_\delta| \big) .
\end{eqnarray*}
From this, (\ref{160830-11}) and (\ref{0322-th1.4-s1-i}),
we obtain  (\ref{0219-error-control-Step2}), with
\begin{eqnarray*}
C_4(M_1,M_2,y_0,r)\triangleq
(1+e^{C_1(1+T^*_{y_0}+\frac{1}{T(M_2,y_0)})}\|y_0\|^{24}r^{-20})(8C_2+C_{23}(1+C_1)),
\end{eqnarray*}
 where $C_1=C_1(M_2,y_0)$ is given by (\ref{0322-th1.4-s1-i}), $C_2=C_2(M_1,M_2,y_0,r)$ is given by (\ref{0213-order-control-step1})
 and $C_{23}=C_{23}(\Omega)$ is given by (\ref{0213-order-control-part1.3-1}).
 Notice that $(a)$ $C_1(M_2,y_0)$ and $C_2(M_1,M_2,y_0,r)$ are continuous; $(b)$ the maps $y_0\rightarrow T^*_{y_0}$ and $(M_2, y_0)\rightarrow T(M_2,y_0)$ are continuous
 (which were showed in the proof of the continuity of $\delta_0$ with respect to $(M_1,M_2,y_0,r)$ in
 Theorem \ref{Theorem-time-order}); $(c)$ $T(M_2,y_0)>0$. From these, we can easily show that the above  $C_4(M_1,M_2,y_0,r)$ is continuous in $M_1$, $M_2$, $y_0$ and $r$. This  ends the proof of Step 5.

In summary, we end the proof of the conclusion (i) in Theorem \ref{Theorem-time-control-convergence}.

(ii) Arbitrarily fix $M>0$ and $\eta>0$. Let $\mathcal A_{M,\eta}$ be given by Theorem \ref{theorem-optimality}. Then, by Theorem \ref{theorem-optimality}, we see that \begin{eqnarray*}
\lim_{h\rightarrow 0^+} \frac{1}{h} |\mathcal A_{M,\eta}\cap (0,h)|=\eta
 \end{eqnarray*}
 and that for each $\delta\in\mathcal A_{M,\eta}$,
\begin{eqnarray}\label{0912-th2-ii-1}
 M-N_\delta(T_\delta(M,y_0),y_0) \geq  \frac{1}{2} \lambda_1^{3/2} r (1-\eta)\delta.
\end{eqnarray}
Arbitrarily fix $\delta\in \mathcal A_{M,\eta}$.
Let $u^*_M$ and $u^*_{M,\delta}$
be  the optimal control to  $(TP)^{M,y_0}$ and  the optimal control optimal with the minimal norm   to $(TP)_\delta^{M,y_0}$, respectively. (see Theorem \ref{Lemma-existences-TP}.) Three facts are given in order:
(a)
By Theorem \ref{Theorem-eq-cont-discrete}, one can easily check
\begin{eqnarray*}
 \|u^*_M\|_{L^2(0,T(M,y_0);L^2(\Omega))}
 =\|u^*_M|_{(0,T(M,y_0))}\|_{L^2(0,T(M,y_0);L^2(\Omega))}
 =N(T(M,y_0),y_0)
 =M;
\end{eqnarray*}
(b) From (iii) of Theorem \ref{Lemma-existences-TP}, we see that
\begin{eqnarray*}
\|u^*_{M,\delta}\|_{L^2(0,T_\delta(M,y_0);L^2(\Omega))}
  =N_\delta(T_\delta(M,y_0),y_0);
  \end{eqnarray*}
(c) Since $\mathcal U^M_\delta\subset \mathcal U^M$,  by (\ref{time-1}) and (\ref{time-2}), we find that $T(M,y_0) \leq T_\delta(M,y_0)$.
Combining the above facts (a)-(c) with  (\ref{0912-th2-ii-1}), we find that
\begin{eqnarray*}
  \|u^*_M-u^*_{M,\delta}\|_{L^2(0,T(M,y_0);L^2(\Omega))}
  &\geq& \|u^*_M\|_{L^2(0,T(M,y_0);L^2(\Omega))}
  - \|u^*_{M,\delta}\|_{L^2(0,T(M,y_0);L^2(\Omega))}
  \nonumber\\
  &\geq& M-N_\delta(T_\delta(M,y_0),y_0) \geq \frac{1}{2} \lambda_1^{3/2}r (1-\eta)\delta,
\end{eqnarray*}
which leads to (\ref{TP-control-converge-0-1}). Thus, the conclusion (ii) in Theorem \ref{Theorem-time-control-convergence} is true.

In summary, we end the proof of  Theorem \ref{Theorem-time-control-convergence}.
\end{proof}

\subsection{The proof of  Theorem \ref{Theorem-time-control-convergence-1}}

This subsection devotes to the proof of Theorem \ref{Theorem-time-control-convergence-1}. To show (\ref{TP-control-converge-1-1}) in Theorem \ref{Theorem-time-control-convergence-1}, we
 need the next lemma which gives a lower bound for the diameter of the subset $\mathcal O_{M,\delta}$ (in $L^2(0,T(M,y_0);L^2(\Omega))$), which is defined by (\ref{geometic-OMdelta}).

\begin{lemma}\label{0912-diam-O}
Let $y_0\in L^2(\Omega)\setminus B_r(0)$ and $M>0$.
Then  there is  $\delta_M\triangleq \delta_M(y_0,r)>0$ so that for each  $\eta\in(0,1)$ and $\delta\in \mathcal A_{M,\eta}\cap(0,\delta_M)$ (where $\mathcal A_{M,\eta}$ is given by Theorem \ref{theorem-optimality}),
\begin{eqnarray}\label{0910-op-lower}
 \mbox{diam}\, \mathcal O_{M,\delta}
 &\triangleq& \sup\{\|u_\delta-v_\delta\|_{L^2(0,T(M,y_0);L^2(\Omega))}~:~u_\delta,\,v_\delta\in \mathcal O_{M,\delta}\}
 \nonumber\\
  &\geq&   \hat C_M \sqrt{(1-\eta)\delta}\;\;\mbox{for some}\;\;\hat C_M\triangleq \hat C_M(y_0,r).
\end{eqnarray}
\end{lemma}
\begin{proof}
Arbitrarily fix $y_0\in L^2(\Omega)\setminus B_r(0)$ and $M>0$. Let  $\delta_0>0$ be given by (i) of Theorem \ref{Theorem-time-order} with $M_1=M/2$ and $M_2=M$. From (i) and (ii) of Theorem \ref{Theorem-eq-cont-discrete}, we see that $0<T(M,y_0)<T^*_{y_0}$. Thus we can take a positive number $\delta_1$ in the following manner:
\begin{eqnarray}\label{0910-op-lower-00}
\delta_1\triangleq \min\{\delta_0,T(M,y_0)/2,(T^*_{y_0}-T(M,y_0))/2\}.
\end{eqnarray}
 Arbitrarily fix $\eta\in(0,1)$ and $\delta\in\mathcal A_{M,\eta}\cap(0,\delta_1)$.
  {\it From now on and throughout the proof of Lemma~\ref{0912-diam-O},
  We simply write $T^*$ and $T^*_\delta$ for $T(M,y_0)$ and  $T_\delta(M,y_0)$ respectively; simply write $L^2(0,T^*)$ and $L^2(0,T^*_\delta)$ for $L^2(0,T^*;L^2(\Omega))$ and $L^2(0,T^*_\delta;L^2(\Omega))$
  respectively.}

From (\ref{0910-op-lower-00}) and Theorem \ref{Theorem-time-order}, we see that
\begin{eqnarray}\label{0910-op-lower-0}
 2\delta< T^*\leq T^*_\delta
 \leq T^*+2\delta<T^*_{y_0}.
\end{eqnarray}
Meanwhile, it follows from the second conclusion in (\ref{ineqs-optimality})
in Theorem~\ref{theorem-optimality} that
\begin{eqnarray}\label{0910-op-lower-000}
 M_\delta\triangleq N_\delta(T^*_\delta,y_0)
   \leq M - \frac{1}{2} \lambda_1^{3/2} r (1-\eta)\delta.
\end{eqnarray}

   To show (\ref{0910-op-lower}), it suffices to find a subset $\mathcal O^2_{M,\delta}\subset \mathcal O_{M,\delta}$ so that for some $\hat C_M\triangleq
   C_M(y_0,r)$,
    $\hat C_M \sqrt{(1-\eta)\delta}$ is a lower bound for the "diam\,$\mathcal O^2_{M,\delta}$". To this end, we first introduce an affiliated subset  $\mathcal O_{M,\delta}^1\subset \mathcal O_{M,\delta}$ in the following manner:
     Let $u^*_\delta$ be the optimal control with the minimal norm to $(TP)^{M,y_0}_\delta$ (see (iii) of Theorem \ref{Lemma-existences-TP}). Arbitrarily fix  $\hat v_\delta\in L^2_\delta(0,T^*_\delta)$  so that
\begin{eqnarray}\label{0910-op-lower-1}
\begin{cases}
 \mbox{supp}\, \hat v_\delta \subset(0,T(M,y_0)),~
 \langle \hat v_\delta, u^*_\delta \rangle_{L^2(0,T^*_\delta)}=0,
 \\
 \|\hat v_\delta\|_{L^2(0,T^*_\delta;L^2(\Omega))}=1,~
 \langle y(T^*_\delta;0,u^*_\delta),y(T^*_\delta;0,\hat v_\delta) \rangle \leq 0.
 \end{cases}
\end{eqnarray}
(The existence of such $\hat v_\delta$ can be easily verified.)
Define $\mathcal O_{M,\delta}^1$ to be the set of all solutions $u_\delta$  to the following problem:
\begin{eqnarray}\label{0910-op-lower-3}
   u_\delta=\alpha u^*_\delta+\beta \hat v_\delta,~\alpha,\,\beta\in\mathbb R;\;\;
  \|u_\delta\|_{L^2(0,T^*_\delta)} \leq M;\;\;
  \|y(T^*_\delta;y_0,u_\delta)\|\leq r.
  \end{eqnarray}
From (\ref{0910-op-lower-3}), we see that $\mathcal O_{M,\delta}^1\subset \mathcal O_{M,\delta}$.

We next  characterize  elements in $\mathcal O_{M,\delta}^1$ via studying
the problem (\ref{0910-op-lower-3}).
To this end, we first claim
\begin{eqnarray}\label{0910-op-lower-4}
 \begin{cases}
  \|u_\delta^*\|_{L^2(0,T^*_\delta)} =N_\delta(T^*_\delta,y_0)\triangleq M_\delta,
  \\
  \|y(T^*_\delta;y_0,u_\delta^*)\|=r,
  \\
  \langle y(T^*_\delta;y_0,u^*_\delta),y(T^*_\delta;0,u^*_\delta) \rangle
  =-\frac{r M_\delta^2}{\|z_\delta^*\|},
  \\
  \langle y(T^*_\delta;y_0,u^*_\delta),y(T^*_\delta;0,\hat v_\delta) \rangle
  =0,
 \end{cases}
\end{eqnarray}
where $z^*_\delta$ denotes the minimizer of $(JP)_\delta^{T^*_\delta,y_0}$.
 Indeed, the first equality in (\ref{0910-op-lower-4})  follows from (iii) of Theorem \ref{Lemma-existences-TP}; To show the second one, two facts are given in order. Fact one:
   From (iii) of Theorem \ref{Lemma-existences-TP}, we see that the restriction of $u^*_\delta$
 over $(0,T^*_\delta)$, denoted in the same manner,
 is an optimal control to $(NP)^{T^*_\delta,y_0}_\delta$.
 Fact two: By (\ref{0910-op-lower-0}) and (\ref{P-T*}), we find that $(\delta, T^*_\delta/\delta)\in  \mathcal P_{T^*_{y_0}}$. By these two facts, we
 can use (\ref{op-control-final-state}) in
 Theorem~\ref{Lemma-NP-bangbang} (with $(\delta,k)=(\delta, T^*_\delta/\delta)$) to obtain the second equality in (\ref{0910-op-lower-4}); To show  the third equality in (\ref{0910-op-lower-4}), we recall the above two facts. Then we can apply (ii) in Theorem~\ref{Lemma-NP-bangbang} (with $(\delta,k)=(\delta, T^*_\delta/\delta)$) to get that
  \begin{eqnarray}\label{0910-op-lower-4-1}
 & &\langle y(T^*_\delta;y_0,u^*_\delta),y(T^*_\delta;0,u^*_\delta) \rangle
 =\langle -r \frac{z^*_\delta}{\|z^*_\delta\|}, y(T^*_\delta;0,u^*_\delta) \rangle
 \nonumber\\
 &=&-\frac{r}{\|z^*_\delta\|} \langle \chi_\omega \varphi(\cdot;T^*_\delta,z^*_\delta), u^*_\delta(\cdot)\rangle_{L^2(0,T^*_\delta)}
 \nonumber\\
 &=&-\frac{r}{\|z^*_\delta\|} \langle \chi_\omega \overline\varphi_\delta(\cdot;T^*_\delta,z^*_\delta), u^*_\delta(\cdot)\rangle_{L^2(0,T^*_\delta)}
 =-\frac{r}{\|z^*_\delta\|}\|u^*_\delta\|^2_{L^2(0,T^*_\delta)}.
\end{eqnarray}
(The  first equality on the last line of (\ref{0910-op-lower-4-1}) is obtained by the same way as that used to show  (\ref{wang4.30}).) Then the third equality in (\ref{0910-op-lower-4}) follows from (\ref{0910-op-lower-4-1}) and the first equality in (\ref{0910-op-lower-4}) at once; To show the last equality in (\ref{0910-op-lower-4}),
 we still recall the above two facts (given in the proof of the second equality in (\ref{0910-op-lower-4})). Then we can apply (ii) in Theorem~\ref{Lemma-NP-bangbang}
  (with $(\delta,k)=(\delta, T^*_\delta/\delta)$) to see that
\begin{eqnarray}\label{0910-op-lower-4-2}
 & &\langle y(T^*_\delta;y_0,u^*_\delta),y(T^*_\delta;0,\hat v_\delta) \rangle
 = \langle -r \frac{z^*_\delta}{\|z^*_\delta\|}, y(T^*_\delta;0,\hat v_\delta)\rangle
 \nonumber\\
 &=&-\frac{r}{\|z^*_\delta\|}\langle \chi_\omega \varphi(\cdot;T^*_\delta,z^*_\delta), \hat v_\delta(\cdot) \rangle_{L^2(0,T^*_\delta)}
 \nonumber\\
 &=&
 -\frac{r}{\|z^*_\delta\|}\langle \chi_\omega \overline\varphi_\delta(\cdot;T^*_\delta,z^*_\delta), \hat v_\delta(\cdot) \rangle_{L^2(0,T^*_\delta)}
  =-\frac{r}{\|z^*_\delta\|}\langle u^*_\delta(\cdot), \hat v_\delta(\cdot) \rangle_{L^2(0,T^*_\delta)}.
\end{eqnarray}
(The first  equality on the last line in (\ref{0910-op-lower-4-2}) is obtained by the same way as that used to show  (\ref{wang4.30}).)
From (\ref{0910-op-lower-4-2}) and  (\ref{0910-op-lower-1}), we are led to the last equality in (\ref{0910-op-lower-4}). Hence,  (\ref{0910-op-lower-4}) has been proved.

With the aid of (\ref{0910-op-lower-4}), we can characterize elements $u_\delta$ of  $\mathcal O_{M,\delta}^1$ as follows:
 \begin{eqnarray}\label{0910-op-lower-5}
 \left\{
 \begin{array}{l}
  u_\delta=\alpha u^*_\delta+\beta \hat v_\delta,~\alpha,\,\beta\in\mathbb R,\\
  \alpha^2 M_\delta^2+\beta^2 \leq M^2,\\
   a_\delta^2\beta^2+2(\alpha-1)b_\delta \beta \leq 2(\alpha-1) \frac{rM_\delta^2}{\|z^*_\delta\|} - (\alpha-1)^2 c_\delta^2,
 \end{array}
 \right.
\end{eqnarray}
where the pair $(a_\delta,b_\delta,c_\delta)$ is given by
\begin{eqnarray}\label{0910-op-lower-6}
 \left\{
 \begin{array}{l}
  a_\delta \triangleq \|y(T^*_\delta;0,\hat v_\delta)\|,\\
  b_\delta \triangleq  \langle y(T^*_\delta;0,u^*_\delta),y(T^*_\delta;0,\hat v_\delta) \rangle,\\
  c_\delta \triangleq \|y(T^*_\delta;0,u^*_\delta)\|.
 \end{array}
 \right.
\end{eqnarray}
Indeed,  for each $u_\delta=\alpha u^*_\delta+\beta \hat v_\delta$, with $\alpha, \beta\in \mathbb{R}$, we have that
\begin{eqnarray*}
y(T^*_\delta;y_0,u_\delta)=y(T^*_\delta;y_0,u^*_\delta)
+(\alpha-1)y(T^*_\delta;y_0,u^*_\delta)
+\beta y(T^*_\delta;0,\hat v_\delta).
\end{eqnarray*}
Thus,  from  (\ref{0910-op-lower-4}), (\ref{0910-op-lower-1}) and
 (\ref{0910-op-lower-3}), we can easily verify that  $u_\delta$ is a solution to the problem
 (\ref{0910-op-lower-3})
 if and only if  $u_\delta$ is a solution to the problem   (\ref{0910-op-lower-5}).

We now on the position to introduce the desired subset $(\mathcal{O})_{M,\delta}^2$.
 Define a number $\hat\lambda$ by
\begin{eqnarray}\label{0910-op-lower-7}
 \hat\lambda \triangleq
 \min \left\{ \frac{rM_\delta^3}{\|z_\delta^*\| c_\delta^2(M-M_\delta)}, \frac{1}{2} \right\}.
\end{eqnarray}
(Notice that since $\delta\in\mathcal A_{M,\eta}$, it follows from (\ref{ineqs-optimality})
in Theorem~\ref{theorem-optimality} that $M>M_\delta$.)
Let $(\mathcal{O})_{M,\delta}^2$ be the set of solutions $u_\delta$ to the following problem:
\begin{eqnarray}\label{0910-op-lower-8}
  \left\{
 \begin{array}{l}
 u_\delta=\alpha u^*_\delta+\beta \hat v_\delta,~
  \alpha= 1+ \hat \lambda \frac{M-M_\delta}{M_\delta},~\beta>0,\\
  \beta^2 \leq M(1-\hat\lambda)(M-M_\delta),\\
  a_\delta^2 \beta^2 \leq  \frac{\hat \lambda rM_\delta(M-M_\delta)}{\|z^*_\delta\|}.
 \end{array}
 \right.
 \quad\quad
\end{eqnarray}
We claim that
\begin{eqnarray}\label{0910-op-lower-9}
 (\mathcal{O})_{M,\delta}^2 \subset (\mathcal{O})_{M,\delta}^1\subset (\mathcal{O})_{M,\delta}.
\end{eqnarray}
Since the second conclusion in (\ref{0910-op-lower-9}) has been proved, we only need to show the first one. Arbitrarily fix $\hat u_\delta\triangleq\hat \alpha u^*_\delta+\hat\beta \hat v_\delta\in (\mathcal{O})_{M,\delta}^2$. We will show that $(\hat u_\delta, \hat\alpha,\hat\beta)$ satisfies (\ref{0910-op-lower-5}). Since $\hat\lambda\in(0,1)$ and $M>M_\delta$ (see (\ref{0910-op-lower-7}) and (\ref{0910-op-lower-000})),
it follows from (\ref{0910-op-lower-8})  that
\begin{eqnarray*}
 \hat \alpha^2 M_\delta^2 + \hat\beta^2
 &\leq& \big(M_\delta+\hat\lambda (M-M_\delta) \big)^2 + M(1-\hat\lambda)(M-M_\delta)
 \nonumber\\
 &=& M^2 - (1-\hat\lambda)(M-M_\delta) \big(M -(1-\hat\lambda) (M-M_\delta) \big)
 \leq  M^2.
\end{eqnarray*}
Meanwhile, since $b_\delta \leq 0$ (see (\ref{0910-op-lower-6})
and (\ref{0910-op-lower-1})), we find from (\ref{0910-op-lower-8}) and (\ref{0910-op-lower-7}) that
\begin{eqnarray*}
 a_\delta^2 \hat\beta^2+2(\alpha-1)b_\delta \hat\beta
 &\leq& a_\delta^2 \hat\beta^2 \leq \frac{\hat \lambda rM_\delta(M-M_\delta)}{\|z^*_\delta\|}
 \nonumber\\
 &=& (\hat\alpha-1) \frac{rM_\delta^2}{\|z^*_\delta\|}
 \leq 2(\hat\alpha-1) \frac{rM_\delta^2}{\|z^*_\delta\|} - (\hat\alpha-1)^2 c_\delta^2.
\end{eqnarray*}
From these, we see that  $(\hat u_\delta,\hat\alpha,\hat\beta)$ verifies (\ref{0910-op-lower-5}). Hence,  (\ref{0910-op-lower-9}) is true.
By (\ref{0910-op-lower-9}), (\ref{0910-op-lower-8}) and (\ref{0910-op-lower-1}), we find that
\begin{eqnarray}\label{0910-op-lower-10}
 \mbox{diam}\,(\mathcal{O})_{M,\delta}
 \geq \sup\{\|u_\delta-u^*_\delta\|_{L^2(0,T^*)}
 : u_\delta\in (\mathcal{O})_{M,\delta}^2\}
 \geq \beta \|\hat v_\delta\|_{L^2(0,T^*)}
 =\beta,
\end{eqnarray}
when $\beta$  satisfies that
\begin{eqnarray*}
0<\beta^2 \leq \min\Big\{ M(1-\hat\lambda)(M-M_\delta), \frac{\hat \lambda rM_\delta(M-M_\delta)}{a_\delta^2\|z^*_\delta\|}
\Big\}.
\end{eqnarray*}
(Here, we agree that $\frac{1}{0}\triangleq \infty$.)
Then by (\ref{0910-op-lower-10}) and (\ref{0910-op-lower-7}),  we get that
\begin{eqnarray}\label{0910-op-lower-11}
 \mbox{diam}\,(\mathcal{O})_{M,\delta}
 \geq C_{M,\delta} \min\{\sqrt{M-M_\delta},1\},
\end{eqnarray}
where $C_{M,\delta}$ is defined by
\begin{eqnarray}\label{0910-op-lower-12}
 C_{M,\delta} \triangleq C_{M,\delta}(y_0,r)\triangleq
 \min \Big\{
 \sqrt{\frac{M}{2}}, \frac{r M_\delta^2}{a_\delta c_\delta \|z^*_\delta\|}, \sqrt{\frac{r M_\delta}{2a_\delta^2\|z^*_\delta\| }}
 \Big\}.
\end{eqnarray}
To get a lower bound of $C_{M,\delta}$ w.r.t. $\delta$, we first present the following inequalities (their proofs will be given at the end of the proof of this lemma):
\begin{eqnarray}\label{0910-op-lower-13}
a_\delta \leq \frac{1}{\sqrt{2\lambda_1}};\; c_\delta \leq \frac{M_\delta}{\sqrt{2\lambda_1}};\;M_\delta\geq M-2 e^{C_3(1+\frac{1}{T^*})}\|y_0\| \delta;\;\|z^*_\delta\|\leq e^{C_1(1+\frac{1}{T^*})}\|y_0\|^{4} r^{-3},
\end{eqnarray}
where $C_3$ and $C_1$ are given by Theorem \ref{Proposition-NP-Lip-T} and (i) of Theorem \ref{Lemma-minizer-P-L2-uni-bdd}, respectively.
We next define
\begin{eqnarray}\label{0910-op-lower-14}
 \delta_M\triangleq\delta_M(y_0,r) \triangleq
 \min \left\{\delta_1, \frac{1}{4} Me^{-C_3(1+\frac{1}{T(M,y_0)})} \|y_0\|^{-1},
  \lambda_1^{-3/2}r^{-1}
 \right\},
\end{eqnarray}
where $\delta_1$ is given by (\ref{0910-op-lower-00}). From (\ref{0910-op-lower-13}) and (\ref{0910-op-lower-14}), we have that
\begin{eqnarray*}
M_\delta\geq M/2\;\;\mbox{for each}\;\;\delta\in \mathcal A_{M,\eta} \cap (0,\delta_M).
\end{eqnarray*}
 This, along with (\ref{0910-op-lower-12}) and (\ref{0910-op-lower-13}), yields that for each $\delta\in \mathcal A_{M,\eta} \cap (0,\delta_M)$,
\begin{eqnarray*}
 C_{M,\delta}
 &\geq& \min \Big\{
 \sqrt{\frac{M}{2}}, \frac{2\lambda_1 r M_\delta}{ \|z^*_\delta\|}, \sqrt{\frac{\lambda_1 r M_\delta}{\|z^*_\delta\| }}
 \Big\}
 \geq  \min \Big\{
 \sqrt{\frac{M}{2}}, \frac{\lambda_1 r M}{ \|z^*_\delta\|}, \sqrt{\frac{\lambda_1 r M}{2\|z^*_\delta\| }}
 \Big\}.
\end{eqnarray*}
By this and the last inequality in (\ref{0910-op-lower-13}), we can find  $C_M^\prime\triangleq C_M^\prime(y_0,r)>0$ so that
$C_{M,\delta} \geq C_M^\prime$, when  $\delta\in \mathcal A_{M,\eta} \cap (0,\delta_M)$.
({\it Hence, $C_M^\prime$ is a lower bound for $C_{M,\delta}$ w.r.t. $\delta$.})
This, along with (\ref{0910-op-lower-11}), (\ref{0910-op-lower-000}) and (\ref{0910-op-lower-14}), yields that for each $\delta\in \mathcal A_{M,\eta} \cap (0,\delta_M)$,
\begin{eqnarray*}
 \mbox{diam}\,(\mathcal{O})_{M,\delta}
 \geq C_{M}^\prime \min\Big\{\sqrt{\frac{1}{2} \lambda_1^{3/2} r (1-\eta)\delta},1 \Big\}
 = \Big(\frac{1}{\sqrt{2}} C_{M}^\prime \lambda_1^{3/4} \sqrt{r} \Big) \sqrt{(1-\eta)\delta}.
\end{eqnarray*}
By the above and (\ref{0910-op-lower-14}), we obtain (\ref{0910-op-lower}), with $\hat C_M=\frac{1}{\sqrt{2}} C_{M}^\prime \lambda_1^{3/4} \sqrt{r}$.

Finally, we show (\ref{0910-op-lower-13}). By  the H\"{o}lder inequality, (\ref{0910-op-lower-6}) and (\ref{0910-op-lower-1}), we find that
\begin{eqnarray*}
 a_\delta \leq \int_0^{T^*_\delta} \|e^{\Delta (T^*_\delta-t)}\| \|\hat v_\delta(t)\| \,\mathrm dt
 \leq  \int_0^{T^*_\delta} e^{-\lambda_1(T^*_\delta-t)} \|\hat v_\delta(t)\| \,\mathrm dt
 \leq 1/\sqrt{2\lambda_1}.
\end{eqnarray*}
Similarly,  from (\ref{0910-op-lower-6}) and (\ref{0910-op-lower-4}), we can obtain  the estimate for $c_\delta$ in (\ref{0910-op-lower-13}). We now show the third inequality in (\ref{0910-op-lower-13}). By (\ref{0910-op-lower-0}), we have that
 \begin{eqnarray}\label{YUyongyu6.66}
 0<T^*<T^*_{y_0},\;\; (\delta, T^*_\delta/\delta)\in \mathcal{P}_{y^*_{y_0}}\;\;\mbox{and}\;\; 0<T^*_\delta-T^*<2\delta.
 \end{eqnarray}
 From the first two conclusions in (\ref{YUyongyu6.66}), we can apply  (ii) of Theorem \ref{Theorem-eq-cont-discrete} and the first inequality in (\ref{NP-delta-error-0}) in Theorem~\ref{theorem-NP-delta-error} (with
 $(\delta,k)=(\delta, T^*_\delta/\delta)$), as well as (\ref{0910-op-lower-000}), to get that
 \begin{eqnarray}\label{YUyongyu6.67}
 M-M_\delta=N(T^*,y_0)-N_\delta(T^*_\delta,y_0)
 \leq
 N(T^*,y_0)-N(T^*_\delta,y_0).
\end{eqnarray}
From (\ref{YUyongyu6.67}),  the second inequality in (\ref{NP-Lip-T}) in Theorem~\ref{Proposition-NP-Lip-T}, with $T_1=T^*$ and $T_2=T^*_\delta$, (Notice that $T^*_\delta>T^*$.)
 and the last inequality in (\ref{YUyongyu6.66}), we can easily derive
 the   last inequality in (\ref{0910-op-lower-13}). Hence, (\ref{0910-op-lower-13}) is true. This ends the proof of Lemma~\ref{0912-diam-O}.
\end{proof}

We are now on the position to prove Theorem \ref{Theorem-time-control-convergence-1}.

\begin{proof}[Proof of Theorem \ref{Theorem-time-control-convergence-1}]
Arbitrarily fix $y_0\in L^2(\Omega)\setminus B_r(0)$.
For each $M>0$ and $\delta>0$, we let $u_M^*$ and $u^*_{\delta,M}$ be the  optimal control and the  optimal control with the minimal norm to $(TP)^{M,y_0}$
and $(TP)^{M,y_0}_\delta$ respectively (see Theorem~\ref{Lemma-existences-TP}).
{\it Throughout the proof of Theorem \ref{Theorem-time-control-convergence-1},
we  simply write respectively $T^*$ and $T^*_\delta$ for $T(M,y_0)$ and $T_\delta(M,y_0)$;
simply write $L^2(0,T^*)$ and $L^2(0,T^*_\delta)$ for $L^2(0,T^*;L^2(\Omega))$ and $L^2(0,T^*_\delta;L^2(\Omega))$ respectively}.
 We will  prove the conclusions (i)-(ii) of Theorem \ref{Theorem-time-control-convergence-1} one by one.

  (i) Arbitrarily fix $M_1$ and $M_2$, with  $0<M_1<M_2$. Then arbitrarily fix $M\in [M_1,M_2]$.
   For each $\delta>0$, there are only two possibilities:  either (\ref{WWGSS6.34}) or (\ref{0322-th1.4-delta}) holds. In the case when $\delta$ verifies (\ref{WWGSS6.34}), we can obtain (\ref{TP-control-converge-1}) by the  similar way to that used to show (\ref{TP-control-converge-0}). We next consider the case that $\delta$ satisfies (\ref{0322-th1.4-delta}).
   Recall (\ref{geometic-OMdelta}) for the subset $\mathcal O_{M,\delta}$ (which  consists of all optimal controls to $(TP)_\delta^{M,y_0}$).
 Then it follows from  Definition \ref{wgsdefinition1.1} that
 \begin{eqnarray}\label{0219-error-control-02}
  \|u_\delta^*\|_{L^2(0,T_\delta^*)}
  \leq \|v_\delta\|_{L^2(0,T_\delta^*)}\leq M
  \;\;\mbox{for each}\;\;
  v_\delta \in \mathcal O_{M,\delta}.
 \end{eqnarray}
 Arbitrarily fix  $v_\delta\in \mathcal O_{M,\delta}$. One can directly check that
 \begin{eqnarray*}
 \lambda v_\delta + (1-\lambda) u_\delta^*\in\mathcal O_{M,\delta}\;\;\mbox{for each}\;\;
 \lambda\in(0,1).
 \end{eqnarray*}
   From this and (\ref{0219-error-control-02}), we find that for each $\lambda\in(0,1)$,
 \begin{eqnarray*}
  \|u_\delta^*\|_{L^2(0,T_\delta^*)}^2
  &\leq& \|\lambda (v_\delta-u^*_\delta)+u_\delta^*\|_{L^2(0,T_\delta^*)}^2
  \nonumber\\
  &=&\|u_\delta^*\|_{L^2(0,T_\delta^*)}^2+2\lambda \langle v_\delta-u_\delta^*, u_\delta^* \rangle_{L^2(0,T_\delta^*)}
   + \lambda^2 \|v_\delta-u_\delta^*\|_{L^2(0,T_\delta^*)}^2.
 \end{eqnarray*}
 Dividing the above by $\lambda$ and then sending $\lambda\rightarrow\infty$, we obtain that
\begin{eqnarray*}
\langle u_\delta^*, u_\delta^*\rangle_{L^2(0,T_\delta^*)}
   \leq \langle v_\delta, u_\delta^* \rangle_{L^2(0,T_\delta^*)}.
   \end{eqnarray*}
  From this, (\ref{0219-error-control-02}) and (\ref{0322-th1.2-udelta}) (as well as (\ref{0322-th1.4-s1-ii})), one can directly check that
 \begin{eqnarray}\label{0219-error-control-03}
  \|v_\delta-u_\delta^*\|^2_{L^2(0,T_\delta^*)}
  \leq 2M(M-N_\delta(T_\delta^*,y_0))\triangleq 2M(M-M_\delta).
 \end{eqnarray}
 (Here, we used the fact that $M\geq N_\delta(T^*_\delta,y_0)$, which follows from
 (iii) of Theorem~\ref{Lemma-existences-TP}.)
Hence,  from (\ref{0219-error-control-Step2}), (\ref{0219-error-control-03}) and (\ref{0322-th1.4-s1-i}), we find that
 \begin{eqnarray*}\label{0219-error-control-Step3}
  & & \|u^* - u_\delta\|_{L^2(0,T^*)}
  \leq \|u^* - u_\delta^*\|_{L^2(0,T^*)}
  +\|u^*_\delta - u_\delta\|_{L^2(0,T^*)}
  \nonumber\\
  &\leq& \big[C_4(M_1,M_2,y_0,r)+2M_2C_1(M_2,y_0)\big]\delta
  \triangleq C_5(M_1,M_2,y_0,r)\delta,
 \end{eqnarray*}
 where $C_1(M_2,y_0)$ and $C_4(M_1,M_2,y_0,r)$ are respectively given by (\ref{0322-th1.4-s1-i}) and (\ref{0219-error-control-Step2}).  The continuity of $C_5(M_1,M_2,y_0,r)$
 follows from the continuity of $C_1(M_2,y_0)$ and $C_4(M_1,M_2,y_0,r)$.
 This ends the proof of the conclusion (i) of Theorem \ref{Theorem-time-control-convergence-1}.

 \vskip 5pt
 (ii) We mainly use  Lemma \ref{0912-diam-O} to prove (\ref{TP-control-converge-1-1}).  Arbitrarily fix $M>0$ and $\eta\in (0,1)$. Let $\mathcal A_{M,\eta}$ be given by Theorem \ref{theorem-optimality}. Let $\hat C_M$ and $\delta_M$ be given by Lemma \ref{0912-diam-O}. Arbitrarily fix $\delta\in \mathcal A_{M,\eta}\cap (0,\delta_M)$. We claim that there is $\hat u_{M,\delta} \in \mathcal O_{M,\delta}$ so that
 \begin{eqnarray}\label{0912-th3-op-1}
  \|\hat u_{M,\delta}-u_\delta^*\|_{L^2(0,T^*)} \geq \hat C_M\sqrt{(1-\eta)\delta}/3.
 \end{eqnarray}
 By contradiction, we suppose that it were not true. Then we would find that
 \begin{eqnarray*}
   \| v_{\delta}-u_\delta^*\|_{L^2(0,T^*)}
   \leq \hat C_M\sqrt{(1-\eta)\delta}/3,
   ~\forall\, v_{\delta} \in \mathcal O_{M,\delta}.
 \end{eqnarray*}
 This, along with the definition of $\mathcal O_{M,\delta}$ (see (\ref{0910-op-lower})), implies that
 \begin{eqnarray*}
  \mbox{diam}\,\mathcal O_{M,\delta}
  &\leq& \sup\{\|v_\delta^1-v_\delta^2\|_{L^2(0,T^*)}
  ~:~v_\delta^1,\,v_\delta^2\in \mathcal O_{M,\delta}\}
  \nonumber\\
  &\leq& \sup\{2\|v_\delta-u_\delta^*\|_{L^2(0,T^*)}
  ~:~v_\delta\in \mathcal O_{M,\delta}\}
  \nonumber\\
  &\leq& 2\hat C_M \sqrt{(1-\eta)\delta}/3,
 \end{eqnarray*}
  which contradicts Lemma \ref{0912-diam-O}. Thus, (\ref{0912-th3-op-1}) is true.

Now, we arbitrarily fix $\hat u_{M,\delta} \in \mathcal O_{M,\delta}$ satisfying
(\ref{0912-th3-op-1}). Then by (\ref{0912-th3-op-1}) and  by (i) of Theorem \ref{Theorem-time-control-convergence} (with $M_1=M/2$ and $M_2=M$), there is $C(M,y_0,r)>0$ so that
  \begin{eqnarray}\label{yuanhen6.72}
    \|\hat u_{M,\delta}-u^*\|_{L^2(0,T^*)}
   &\geq& \|\hat u_{M,\delta}-u_\delta^*\|_{L^2(0,T^*)}
   - \|u_\delta^*-u^*\|_{L^2(0,T^*)}
   \nonumber\\
   &\geq& \hat C_M\sqrt{(1-\eta)\delta}/3-C(M,y_0,r)\delta.
  \end{eqnarray}
  Write
  \begin{eqnarray}\label{yuanhen6.73}
  \delta_{M,\eta}\triangleq \min\big\{\delta_M,\big(\hat C_M/(6C(M,y_0,r))\big)^2(1-\eta)\big\};\;\;\hat{\mathcal A}_{M,\eta}\triangleq
  \mathcal A_{M,\eta}\cap (0,\delta_{M,\eta})
  \end{eqnarray}
     Then, one can easily check that
     \begin{eqnarray*}
     \lim_{h\rightarrow 0^+} \frac{1}{h} |\hat{\mathcal A}_{M,\eta}\cap (0,h)|=\eta.
      \end{eqnarray*}
      From (\ref{yuanhen6.72}), and (\ref{yuanhen6.73}), one can easily verify that
     \begin{eqnarray*}
     \|\hat u_{M,\delta}-u^*\|_{L^2(0,T^*)}
  \geq \hat C_M\sqrt{(1-\eta)\delta}/6\;\;\mbox{for each}\;\;\delta\in\hat{\mathcal A}_{M,\eta},
     \end{eqnarray*}
       which leads to (\ref{TP-control-converge-1-1}), with $C_M\triangleq \hat C_M/6$.
   This ends the proof of  Theorem \ref{Theorem-time-control-convergence-1}.
   \end{proof}

\subsection{Further discussions on the main results}

From (ii) of Theorem \ref{Theorem-time-order}, we see that when $\delta\in \mathcal A_{M,\eta}$, $T_\delta(M,y_0)-T(M,y_0)$ has a lower bound $(1-\eta)\delta$.
  The next Theorem~\ref{theorem6.2-wang} tells us that  when  $\delta\notin \mathcal A_{M,\eta}$,
  $(1-\eta)\delta$ will not be a lower bound for $T_\delta(M,y_0)-T(M,y_0)$.

\begin{theorem}\label{theorem6.2-wang}
Let $y_0\in L^2(\Omega)\setminus B_r(0)$. Then for each $M>0$, there is $k_0\in \mathbb{N}^+$ and  $\{\delta_k\}_{k=k_0}^\infty\subset \mathbb R^+$, with $\lim_{k\rightarrow\infty} \delta_k=0$,  so that when $k\geq k_0$,
\begin{eqnarray}\label{111111}
  T_{\delta_k}(M,y_0)-T(M,y_0)=C_M \delta^2_k\;\;\mbox{for some}\;\;C_M\triangleq C_M(y_0,r).
\end{eqnarray}
\end{theorem}
\begin{proof}
Arbitrarily fix $M>0$ and $\eta\in(0,1)$.
{\it Throughout the proof of Theorem~\ref{theorem6.2-wang}, we simply write $T^*$ for $T(M,y_0)$.} Let
\begin{eqnarray}\label{0919-constant-delta2}
\hat k_0\triangleq 4aT^*,\;\;\mbox{with}\;\;a\triangleq a(M,y_0,r) \triangleq 2 \lambda_1^{-3/2} e^{C_4 \big[1+ T^*_{y_0}+\frac{1}{T^*}+\frac{2}{T^*_{y_0}-T^*} \big]} \|y_0\|^{12} r^{-12}.
\end{eqnarray}
We define a sequence $\{\delta_k\}_{k=k_0}^\infty$ of $\mathbb R^+$ in the following manner:
\begin{eqnarray}\label{160919-th1-op-1-1}
 \delta_k \triangleq \frac{2T^*}{(k+1)+\sqrt{(k+1)^2-4aT^*}},~k\geq \hat k_0.
\end{eqnarray}
One can easily check that
\begin{eqnarray}\label{160919-th1-op-1}
 \delta_k\in(0,1/a) \;\;\mbox{and}\;\;
  (k+1)\delta_k-a \delta^2_k=T^*\;\;\mbox{for all}\;\;k\geq \hat k_0.
\end{eqnarray}

  We now claim that there exists $k_0\geq \hat k_0$ so that
\begin{eqnarray}\label{160919-th1-op-4}
M\geq N_{\delta_k}((k+1)\delta_k,y_0)
+ \frac{1}{2} \lambda_1^{3/2} r a\delta_k^2\;\;\mbox{for all}\;\;k\geq k_0.
\end{eqnarray}
In fact, by (\ref{160919-th1-op-1-1}),  we can
choose  $\hat k_1\geq k_0$ large enough so that
\begin{eqnarray}\label{160919-th1-op-6}
 0<\delta_k<\min\{T^*/2,(T^*_{y_0}-T^*)/2\},\;\;\mbox{when}\;\;k\geq \hat k_1.
\end{eqnarray}
Arbitrarily fix $k\geq \hat k_1$.  Since $T^*<T^*_{y_0}$ (see (iii) in Theorem \ref{Theorem-eq-cont-discrete}),
 from (\ref{160919-th1-op-6}) and (\ref{160919-th1-op-1}), we can easily check that
\begin{eqnarray*}
 2\delta_k<T^*< (k+1)\delta_k<T^*+\delta_k<(T^*_{y_0}+T^*)/{2}<T^*_{y_0}.
\end{eqnarray*}
These, along with (\ref{P-T*}), yield that
\begin{eqnarray}\label{0919-gwang7.26}
2\delta_k<T^*<(k+1)\delta_k<T^*_{y_0}\;\;\mbox{and}\;\; (\delta_k,k+1)\in \mathcal{P}^*_{T^*_{y_0}}.
\end{eqnarray}
By (\ref{0919-gwang7.26}), we can apply Theorem \ref{theorem-NP-delta-error}
(see the second inequality in (\ref{NP-delta-error-0}), where $(\delta,k)$ is replaced by $(\delta,k+1)$)
and  Theorem~\ref{Proposition-NP-Lip-T} (see the first inequality in (\ref{NP-Lip-T}),
with $T_1=T^*$ and $T_2=(k+1)\delta_k$) to obtain that
\begin{eqnarray}\label{160919-th1-op-7}
N_{\delta_k}((k+1)\delta_k,y_0)
&\leq& N(( k+1)\delta_k,y_0) + e^{C_4 \big[1+ T^*_{y_0}+\frac{1}{( k+1)\delta_k}+\frac{1}{T^*_{y_0}-( k+1)\delta_k} \big]} \|y_0\|^{12} r^{-11} \delta^2_k
\nonumber\\
&\leq& N(T^*,y_0)- \lambda_1^{3/2} r  \big((k+1)\delta_k - T^* \big) +
\nonumber\\
& &   e^{C_4 \big[1+ T^*_{y_0}+\frac{1}{( k+1)\delta_k}+\frac{1}{T^*_{y_0}-( k+1)\delta_k} \big]} \|y_0\|^{12} r^{-11} \delta^2_k,
\end{eqnarray}
where $C_4\triangleq C_4(\Omega,\omega)$ is given by (\ref{NP-delta-error-0}).
Meanwhile, by (\ref{160919-th1-op-1}) and (\ref{160919-th1-op-6}), we find
\begin{eqnarray*}
 (k+1)\delta_k - T^*= a \delta_k^2
 \;\;\mbox{and}\;\;
 T^*< ( k+1)\delta_k < (T^*_{y_0}+T^*)/2.
\end{eqnarray*}
These, along with (\ref{160919-th1-op-7}) and (ii) of  Theorem \ref{Theorem-eq-cont-discrete}, yield that
\begin{eqnarray*}
  N_{\delta_k}((k+1)\delta_k,y_0)
  &\leq& N(T^*,y_0) - \lambda_1^{3/2} r  a\delta_k^2 +
   e^{C_4 \big[1+ T^*_{y_0}+\frac{1}{T^*}+\frac{2}{T^*_{y_0}-T^*} \big]} \|y_0\|^{12} r^{-11} \delta^2_k
  \nonumber\\
  &=&
   M - \lambda_1^{3/2} r  a\delta_k^2 +
   e^{C_4 \big[1+ T^*_{y_0}+\frac{1}{T^*}+\frac{2}{T^*_{y_0}-T^*} \big]} \|y_0\|^{12} r^{-11} \delta^2_k.
\end{eqnarray*}
This, together with (\ref{0919-constant-delta2}), leads to (\ref{160919-th1-op-4}), with
$k_0=\hat k_1$.

Next, we arbitrarily fix $k \geq k_0\triangleq\hat k_1$.
 Let $u_{\delta_k}$ be an admissible control to $(NP)^{( k+1)\delta_k,y_0}_{\delta_k}$. Let $\tilde u_{\delta_k}$ be the zero extension of $u_{\delta_k}$ over $\mathbb R^+$. Then by (\ref{160919-th1-op-4}),
one can easily check
 that $\tilde u_{\delta_k}$ is an admissible control (to $(TP)^{M,y_0}_{\delta_k}$),
 which drives the solution to $B_r(0)$ at time $(k+1)\delta_k$.
  This, along with  the optimality of $T_\delta(M,y_0)$, yields  that
  \begin{eqnarray}\label{0919-gwang7.28}
  T_{\delta_k}(M,y_0)\leq (k+1)\delta_k.
  \end{eqnarray}
     Meanwhile, Since $\mathcal U^M_{\delta_k}\subset\mathcal U^M$, we find from (\ref{time-1}) and (\ref{time-2})  that $T^* \leq T_{\delta_k}(M,y_0)$. From this and (\ref{160919-th1-op-1}), we get that
     \begin{eqnarray}\label{0919-gwang7.29}
     T_{\delta_k}(M,y_0)\geq k\delta_k+\delta_k(1-a\delta_k)>k \delta_k.
     \end{eqnarray}
           Since $T_{\delta_k}(M,y_0)$ is a multiple of $\delta_k$ (see (\ref{U-ad-delta})), from (\ref{0919-gwang7.28}) and (\ref{0919-gwang7.29}),
           we obtain  that
       \begin{eqnarray*}
       T_{\delta_k}(M,y_0)=(k+1)\delta_k.
       \end{eqnarray*}
            This, along with (\ref{160919-th1-op-1}) and (\ref{0919-constant-delta2}),  yields (\ref{111111}), with $C_M=a(M,y_0,r)$ and with $k_0$ given by (\ref{160919-th1-op-4}). Thus, we end the proof of Theorem~\ref{theorem6.2-wang}.
\end{proof}

\begin{Remark}\label{remark6.3-yuan}
(i) Let $y_0\in L^2(\Omega)\setminus B_r(0)$. The above theorem implies that the following conclusion is not true: For each $M>0$, there exists $\delta_1>0$ and $C>0$ so that
\begin{eqnarray*}
 |T_\delta(M,y_0)-T(M,y_0)| \geq C \delta
 \;\;\mbox{for each}\;\;
 \delta\in(0,\delta_1).
\end{eqnarray*}

(ii) We think of  that the similar result to that  in Theorem \ref{Theorem-time-order} can be obtained for optimal controls. But it seems for us that the corresponding proof will be more complicated.
\end{Remark}

\section{Appendix}
The next Lemma~\ref{lemma-0428-fn} is the copy of \cite[Lemma 5.1]{WangWangZ}.

\begin{lemma}\label{lemma-0428-fn}
Let $\mathbb K$ be either $\mathbb R$ or $\mathbb C$. Let $X$, $Y$ and $Z$ be three Banach spaces over  $\mathbb K$, with their dual spaces $X^*$, $Y^*$ and $Z^*$. Let $R\in \mathcal L(Z,X)$ and $O\in \mathcal L(Z,Y)$. Then the following two propositions are equivalent:

\noindent (i) There exists $\widehat C_0>0$ and $\hat\varepsilon_0>0$ so that for each $z\in Z$,
\begin{eqnarray}\label{lemma-0428-fn-ii}
 \| R z \|^2_X  \leq   \widehat C_0 \|Oz\|^2_Y
  +  \hat\varepsilon_0 \|z\|_Z^2.
\end{eqnarray}

\noindent (ii) There is  $C_0>0$ and $\varepsilon_0>0$ so that for each $x^*\in X^*$, there is
$y^*\in Y^*$ satisfying that
\begin{eqnarray}\label{lemma-0428-fn-i}
 \frac{1}{C_0} \|y^*\|^2_{Y^*} + \frac{1}{\varepsilon_0} \|R^*x^*-O^*y^*\|^2_{Z^*}
 \leq \|x^*\|^2_{X^*}.
\end{eqnarray}

Furthermore, when one of the above two propositions holds, the  pairs $(C_0,\varepsilon_0)$ and $(\widehat C_0,\hat\varepsilon_0)$ can be chosen to be the same.

\end{lemma}

\begin{proof}[Proof of Lemma \ref{lemma-0428-fn}]
The proof is divided into the following several steps:

\textit{Step 1. To show that (ii)$\Rightarrow$(i)}

 Suppose that (ii) is true. Then, for each $x^*\in X^*$, there is $y^*_{x^*}\in Y^*$ so that (\ref{lemma-0428-fn-i}), with $y^*=y^*_{x^*}$, is true. From this, one can easily check that   for any  $x^*\in X^*$ and $z\in Z$,
\begin{eqnarray*}
 \langle Rz, x^* \rangle_{X,X^*}
= \langle z, R^*x^* - O^* y^*_{x^*} \rangle_{Z,Z^*} + \langle O z,  y^*_{x^*} \rangle_{Y,Y^*}.
\end{eqnarray*}
By this and the Cauchy-Schwarz inequality, we deduce that for each  $x^*\in X^*$ and $z\in Z$,
\begin{eqnarray*}
 | \langle Rz, x^* \rangle_{X,X^*}  |
  \leq \big( C_0\|z\|_Z^2  +  \varepsilon_0\|O z\|_Y^2 \big)^{1/2}   \|x^*\|_{X^*}.
\end{eqnarray*}
Hence,  (\ref{lemma-0428-fn-ii}), with
$(\widehat C_0,\hat\varepsilon_0)$$=(C_0,\varepsilon_0)$, is true.

\vskip 5pt
\textit{Step 2. To show that (i)$\Rightarrow$(ii)}

\noindent Suppose that (i) is true.
Define a    subspace $E$ of $Y\times Z$ in the following manner:
\begin{eqnarray*}
 E\triangleq\Big\{ \left(\sqrt{\widehat C_0} Oz,\sqrt{\hat\varepsilon_0} z \right) ~:~  z\in Z \Big\}.
\end{eqnarray*}
The norm of $E$ is inherited form the following usual norm of $Y\times Z$:
\begin{eqnarray}\label{0426-5}
\|(f,g)\|_{Y\times Z} \triangleq
\left( \|f\|_{Y}^2 + \|g\|_Z^2 \right)^{1/2},\; (f,g)\in Y\times Z.
\end{eqnarray}
Arbitrarily fix   $x^*\in X^*$. Define  an operator $\mathcal T_{x^*}: E\rightarrow  \mathbb K$ in the following manner:
\begin{eqnarray}\label{0428-th1.2-control-4}
  \Big(\sqrt{\widehat C_0} Oz,\sqrt{\hat\varepsilon_0} z \Big) \mapsto \langle x^*, Rz \rangle_{X^*,X}.
\end{eqnarray}
By (\ref{lemma-0428-fn-ii}) and (\ref{0428-th1.2-control-4}), we can easily check that $\mathcal T_{x^*}$ is well defined and linear.
We now claim
\begin{eqnarray}\label{0428-th1.2-control-5}
  \| \mathcal T_{x^*}\|_{\mathcal L(E,\mathbb K)}
 \leq   \|x^*\|_{X^*}.
\end{eqnarray}
Indeed, by the definition of $E$, we see that
  given $(f,g)\in E$, there is $z\in Z$ so that
$$
(f,g)= \left(\sqrt{\widehat C_0} Oz,\sqrt{\hat\varepsilon_0} z \right).
$$
Then by (\ref{0428-th1.2-control-4}), we find that
$|\mathcal T_{x^*} \big((f,g)\big)|
   = | \langle x^*, Rz \rangle_{X^*,X} |
  \leq \|x^*\|_{X^*} \|Rz\|_X$.
This, along with   (\ref{lemma-0428-fn-ii}), shows (\ref{0428-th1.2-control-5}).

Since $\mathcal T_{x^*}$ is a linear and bounded functional, we can apply  the Hahn-Banach extension theorem to find
 $\widetilde{ \mathcal T}_{x^*}$ in $(Y\times Z)^*$ so that
\begin{eqnarray}\label{0428-th1.2-control-7}
 \widetilde{ \mathcal T}_{x^*}\big((f,g)\big) = \mathcal T_{x^*}\big((f,g)\big)
 \;\;\mbox{for all}\;\;
 (f,g)\in E
\end{eqnarray}
and so that
\begin{eqnarray}\label{0428-th1.2-control-8}
 \|\widetilde{ \mathcal T}_{x^*}\|_{\mathcal L(Y\times Z,\mathbb K)}
 = \| \mathcal T_{x^*}\|_{\mathcal L(E,\mathbb K)}.
\end{eqnarray}
These, together with (\ref{0426-5}) and (\ref{0428-th1.2-control-5}), yield that
for all $f\in T$ and $g\in Z$,
\begin{eqnarray*}
 |\widetilde{ \mathcal T}_{x^*}\big((f,0)\big)|
 \leq  \|x^*\|_{X^*} \|f\|_{Y};\;\;
  |\widetilde{ \mathcal T}_{x^*}\big((0,g)\big)|
\leq  \|x^*\|_{X^*} \|g\|_{Z}.
\end{eqnarray*}
Thus,    there exists   $(y^*_{x^*},z^*_{x^*})\in Y^*\times Z^*$ so that
for all $(f,g)\in Y\times Z$,
\begin{eqnarray*}
 \widetilde{ \mathcal T}_{x^*}\big((f,0)\big)
 =  \langle y^*_{x^*},f \rangle_{Y^*,Y};\;\;
 \widetilde{ \mathcal T}_{x^*}\big((0,g)\big)
 =  \langle z^*_{x^*},g \rangle_{Z^*,Z},
\end{eqnarray*}
from which,  it follows that
\begin{eqnarray}\label{0428-th1.2-control-9}
 \widetilde{ \mathcal T}_{x^*}\big( (f,g) \big)
 = \langle y^*_{x^*},f \rangle_{Y^*,Y}  +  \langle z^*_{x^*},g \rangle_{Z^*,Z}\;\;\mbox{for any}\;\;(f,g)\in Y\times Z.
\end{eqnarray}

 We now claim that
  \begin{eqnarray}\label{0428-th1.2-control-11}
  \|y^*_{x^*}\|_{Y^*}^2 + \|z^*_{x^*}\|_{Z^*}^2  \leq \|x^*\|_{X^*}^2;\;  R^*x^* - O^*(\sqrt{\widehat C_0}y^*_{x^*}) = \sqrt{\hat\varepsilon_0} z^*_{x^*}
  \;\;\mbox{in}\;\; Z^*.
 \end{eqnarray}
When    (\ref{0428-th1.2-control-11})
is proved, the conclusion (ii)  (with $(C_0,\varepsilon_0)$$=(\widehat C_0,\hat\varepsilon_0)$) follows at once.

 To prove the first inequality in (\ref{0428-th1.2-control-11}), we see   from (\ref{0428-th1.2-control-9}), (\ref{0428-th1.2-control-8}) and (\ref{0426-5}) that
  \begin{eqnarray*}
  | \langle y^*_{x^*},f \rangle_{Y^*,Y}  +  \langle z^*_{x^*},g \rangle_{Z^*,Z} |
  \leq \|x^*\|_{X^*}  \left( \|f\|_{Y}^2 + \|g\|_Z^2 \right)^{1/2}\;\;\mbox{for all}\;\;(f,g)\in Y\times Z.
 \end{eqnarray*}
Meanwhile, for each $\delta\in(0,1)$, we can choose $(f_\delta,g_\delta)\in Y\times Z$ so that
 \begin{eqnarray*}
  \langle y^*_{x^*},f_\delta \rangle_{Y^*,Y} &=& \|y^*_{x^*}\|^2_{Y^*}  + o_1(1), ~ \|f_\delta\|_Y=\|y^*_{x^*}\|_{Y^*},
\\
  \langle z^*_{x^*},g_\delta \rangle_{Z^*,Z} &=& \|z^*_{x^*}\|^2_{Z^*} + o_2(1), ~ \|g_\delta\|_Z=\|z^*_{x^*}\|_{Z^*},
  \end{eqnarray*}
 where $o_1(1)$ and $o_2(1)$ are so that
 $\lim_{\delta\rightarrow0^+} o_1(1)=\lim_{\delta\rightarrow0^+} o_2(1)=0$.
  From these,  it follows that
 \begin{eqnarray*}
  \|y^*_{x^*}\|_{Y^*}^2 + \|z^*_{x^*}\|_{Z^*}^2 - |o_1(1)| - |o_2(1)|  \leq \|x^*\|_{X^*} \left( \|y^*_{x^*}\|_{Y^*}^2 + \|z^*_{x^*}\|_{Z^*}^2 \right)^{1/2}.
 \end{eqnarray*}
 Sending $\delta\rightarrow 0^+$ in the above inequality leads to the first inequality in (\ref{0428-th1.2-control-11}).

To prove the second equality in   (\ref{0428-th1.2-control-11}), we find  from (\ref{0428-th1.2-control-4}), (\ref{0428-th1.2-control-7}) and (\ref{0428-th1.2-control-9}) that
\begin{eqnarray*}
 \langle x^*, Rz \rangle_{X^*,X}
 = \langle y^*_{x^*}, \sqrt{\widehat C_0} Oz \rangle_{Y^*,Y}
    + \langle z^*_{x^*}, \sqrt{\hat \varepsilon_0} z \rangle_{Z^*,Z}\;\;\mbox{for all}\;\;z\in Z,
\end{eqnarray*}
which yields that for all $z\in Z$,
\begin{eqnarray*}
 \langle R^*x^*, z \rangle_{Z^*,Z}
 = \langle O^*(\sqrt{\widehat C_0}y^*_{x^*}),  z \rangle_{Z^*,Z}
    + \langle \sqrt{\hat\varepsilon_0} z^*_{x^*},  z \rangle_{Z^*,Z}.
\end{eqnarray*}
This leads to the second equality in   (\ref{0428-th1.2-control-11}).

\vskip 5pt
\textit{Step 3. On the  pairs $(C_0,\varepsilon_0)$ and $(\widehat C_0,\hat\varepsilon_0)$}

\noindent From the proofs in Step 1 and Step 2, we see that  when one of the propositions (i) and (ii) holds,
 $(C_0,\varepsilon_0)$ and $(\widehat C_0,\hat\varepsilon_0)$ can be chosen to be the same. This ends the proof of Lemma~\ref{lemma-0428-fn}.
\end{proof}

The next Lemma~\ref{lemma-0428-fn} is a part copy of the proof of \cite[Theorem 2.1]{PWX}.
\begin{lemma}\label{lemma-1009-interpolation}
There exists $C\triangleq C(\Omega,\omega)>0$ so that for each $t>0$ and $z\in L^2(\Omega)$,
\begin{eqnarray}\label{1009-interpolation}
  \|e^{\Delta t} z\|  \leq   Ce^{C/t} \|\chi_\omega e^{\Delta t} z \| ^{1/2} \|z\|^{1/2}.
\end{eqnarray}
\end{lemma}

\begin{proof}
Write $\{\lambda_j\}_{j=1}^\infty$ and $\{e_j\}_{j=1}^\infty$ for the eigenvalues and the normalized orthogonal eigenfunctions for $-\Delta$ with the zero Dirichlet boundary condition. Arbitrarily fix $z\in L^2(\Omega)$ and $0\leq t<T$. We have that
\begin{eqnarray*}
 z=\sum_{j=1}^\infty z_j e_j \;\;\mbox{for some}\;\;  \{z_j\}_{j=1}^\infty \in  l^2.
\end{eqnarray*}
Then it follows that
\begin{eqnarray}\label{161009-a1}
 e^{\Delta t} z = \sum_{j=1}^\infty e^{-\lambda_j t} z_j e_j.
\end{eqnarray}
Recall that the Lebeau-Robbiano inequality says that there exists $C\triangleq C(\Omega,\omega)>0$ so that for each $\lambda>0$,
\begin{eqnarray*}
 \|\sum_{\lambda_j\leq \lambda} a_j e_j \| \leq  C e^{C \sqrt{\lambda}}
 \| \chi_\omega \sum_{\lambda_j\leq \lambda} a_j e_j \|,
 ~\forall\, \{a_j\}_{j=1}^\infty \in l^2.
\end{eqnarray*}
We apply the above inequality (where $a_j=e^{-\lambda_j t} z_j$), as well as (\ref{161009-a1}), to get that for each $\lambda>0$,
\begin{eqnarray*}
 \|e^{\Delta t} z\|
 &\leq& \|\sum_{\lambda_j\leq \lambda} e^{-\lambda_j t} z_j e_j \|
    +  \| \sum_{\lambda_j> \lambda} e^{-\lambda_j t} z_j e_j \|
    \nonumber\\
  &\leq& Ce^{C\sqrt{\lambda}} \|\chi_\omega \sum_{\lambda_j\leq \lambda} e^{-\lambda_j t} z_j e_j \|
    +  \| \sum_{\lambda_j> \lambda} e^{-\lambda_j t} z_j e_j \|
    \nonumber\\
 &\leq& Ce^{C\sqrt{\lambda}} \|\chi_\omega \sum_{j=1}^\infty e^{-\lambda_j t} z_j e_j \|
    +  (Ce^{C\sqrt{\lambda}}+1)  \| \sum_{\lambda_j> \lambda} e^{-\lambda_j t} z_j e_j\|
    \nonumber\\
 &\leq& Ce^{C\sqrt{\lambda}} \|\chi_\omega e^{\Delta t} z \|
    +  2Ce^{C\sqrt{\lambda}} e^{-\lambda t}  \| z\|
\end{eqnarray*}
This, as well as the following facts
$$
C\sqrt{\lambda}\leq \frac{C^2}{2t}+\frac{1}{2}\lambda t,~\forall\,\lambda>0,
$$
 indicates that for each $\lambda>0$,
\begin{eqnarray}\label{161011-a2}
 \|e^{\Delta t} z\|  \leq   2Ce^{C^2/(2t)} \Big[ e^{\lambda t/2} \|\chi_\omega e^{\Delta t} z \|
    +  e^{-\lambda t/2}  \| z\| \Big].
\end{eqnarray}
Since the function $\lambda\rightarrow e^{-\lambda t/2}$, $\lambda\in\mathbb R^+$ can take any value in $(0,1)$, it follows by (\ref{161011-a2}) that for each $\varepsilon\in (0,1)$,
\begin{eqnarray}\label{161011-a3}
\|e^{\Delta t} z\|  \leq   2Ce^{C^2/(2t)} \Big[ \frac{1}{\varepsilon} \|\chi_\omega e^{\Delta t} z \|
    +  \varepsilon  \| z\| \Big].
\end{eqnarray}
Meanwhile,  one can easily check that
\begin{eqnarray*}
  \|\chi_\omega e^{\Delta t} z \| \leq \| z\|.
\end{eqnarray*}
From this and (\ref{161011-a3}), we deduce that for each $\varepsilon>0$,
\begin{eqnarray}\label{161011-a4}
\|e^{\Delta t} z\|  \leq   2Ce^{C^2/(2t)} \Big[ \frac{1}{\varepsilon} \|\chi_\omega e^{\Delta t} z \|
    +  \varepsilon  \| z\| \Big].
\end{eqnarray}
Take the infimum  in (\ref{161011-a4}) for $\varepsilon$ over $\mathbb R^+$ to get (\ref{1009-interpolation}).
This ends the proof.

\end{proof}

\end{document}